\documentclass[11pt]{article}

\usepackage{geometry}
\usepackage[bookmarksnumbered,colorlinks,bookmarks,citecolor=blue,linkcolor=blue,urlcolor=blue,breaklinks,linktocpage]{hyperref}
\usepackage[all]{hypcap}

\usepackage{mathrsfs}
\usepackage[all, cmtip]{xy}
\usepackage[dvipsnames]{xcolor}
\usepackage{float}
\usepackage{multirow}
\usepackage{amssymb}
\usepackage{amsmath, array, graphicx, amsthm, bbm, tikz, latexsym, pgfplots, fontawesome, graphicx, multicol,musicography,mathtools}

\usepackage{pgfplots}
\pgfplotsset{compat=1.15}
\usepackage{mathrsfs}
\usetikzlibrary{arrows}
\definecolor{ccqqqq}{rgb}{0.8,0,0}
\definecolor{qqqqff}{rgb}{0,0,1}
\definecolor{uuuuuu}{rgb}{0.26666666666666666,0.26666666666666666,0.26666666666666666}
\definecolor{ffffff}{rgb}{1,1,1}
\definecolor{xdxdff}{rgb}{0.49019607843137253,0.49019607843137253,1}
\definecolor{ududff}{rgb}{0.30196078431372547,0.30196078431372547,1}

\usepackage{soul}

\newcommand{\triads}{\texttt{DualRootPosTriads}}
\newcommand{\dom}{\texttt{DomHalfDiminished}}
\newcommand{\majorsevenths}{\texttt{MajorSeventhChords}}
\newcommand{\minorsevenths}{\texttt{MinorSeventhChords}}
\newcommand{\diminishedsevenths}{\texttt{DimSeventhChords}}

\theoremstyle{plain}
\newtheorem{theorem}{Theorem}[section]

\newtheorem{proposition}[theorem]{Proposition}

\theoremstyle{definition}

\newtheorem{notation}[theorem]{Notation}

\theoremstyle{remark}
\newtheorem{remark}[theorem]{Remark}
\newtheorem{example}[theorem]{Example}

\usepackage[font=footnotesize]{caption}

\begin{document}

\title{Transformations of Triads and Seventh Chords: \\ Group Extensions and Duality}
\author{Thomas M. Fiore\textsuperscript{a}\thanks{Corresponding author. Email: tmfiore@umich.edu \\
\textsuperscript{a}Department of Mathematics and Statistics, University of Michigan-Dearborn, Dearborn, MI, USA; \\
\textsuperscript{b}Department de Teoria, Composici\'o i Direcci\'o, Escola Superior de M\'usica de Catalunya, Barcelona, Spain;\\
\textsuperscript{c}Kenyon College, Gambier, OH, USA; \\
\textsuperscript{d}Evangel University, Springfield, MO, USA; \\
\textsuperscript{e}Lewis and Clark College, Portland, OR, USA; \\
\textsuperscript{f}St.\ Olaf College, Northfield, MN, USA; \\
\textsuperscript{g}Dipartimento di Matematica, Universit\`a di Pavia, Pavia, Italy; \\
\textsuperscript{h}Facolt\`a di Biologia e Farmacia, Universit\`a degli Studi di Cagliari, Cagliari, Italy; \\
\textsuperscript{i}IRMA, Universit\'e de Strasbourg, Strasbourg, France; \\ \textsuperscript{j}IRCAM/CNRS/UMPC, Paris, France \\ },
Thomas Noll\textsuperscript{b},
Ethan Bonnell\textsuperscript{c}, \\
Hayden Pyle\textsuperscript{d},
No\'{e} Rodriguez\textsuperscript{e},
Meredith Williams\textsuperscript{f}, \\
Sonia Cannas\textsuperscript{g,h,i},  and
Moreno Andreatta\textsuperscript{i,j}
}

\maketitle
\vspace{-7.5mm}
\begin{abstract}

Transformational music theory, pioneered by David Lewin, uses simply transitive group actions to analyze music. In this paper, we construct a simply transitive group action on a disjoint union of two sets, built from a simply transitive action on each set and an equivariant bijection connecting them. Motivational examples are the omnibus progression and the reflected omnibus progression, which involve the consonant triads and the dominant/half-diminished seventh chords, connected by the inclusion bijection. We provide other examples from Jazz tunes. More generally, we combine multiple simply transitive group actions via a ``meta-rotation''; examples include a simply transitive group acting on consonant triads and a variety of seventh chords, as well as a meta-rotation that realizes the root position seventh chord sequence of the flattening transformation (described by Clough-Douthett's $J$-function). The constructions in our theorems extend Lewin dual pairs to Lewin dual pairs. We formulate the constructions in terms of short exact sequences and central extensions as well. Contextual groups are also elucidated: the interval content of the generating pitch-class segment determines whether or not a generalized contextual group is generated by contextual inversions. \\

{\bf Keywords:} neo-Riemannian group; generalized contextual group; triads; seventh chords; equivariance; internal direct product; internal semi-direct product; group extension; short exact sequence; rotation; meta-rotation \\

\noindent \textit{2020 Mathematics Subject Classification}: 00A65; 20B35 \\
\textit{2012 Computing Classification Scheme}: none

\end{abstract}

\section{Motivation, introduction, literature overview, and article summary}

In transformational music analysis, one often encounters situations in which a group acts on two different sets connected with an equivariant (or anti-equivariant) bijection. How can this conglomeration be combined into one group action, and how can Lewinian duality (if present) carry over? This is the topic of the present article, with a special focus on seventh chords working in tandem with consonant triads, prototypical passages containing them, and generalizations to a group acting on more than two sets.

A familiar example of a one-group-two-set conglomeration is the $TI$-group acting on both the 24 major and minor triads {\it and} the 24 dominant and half-diminished sevenths.\footnote{In Section~\ref{sec:notations} we describe how we encode chords in this article. Here is an abbreviated description for readers who want to immediately know more in the context of the Introduction. In this article we encode the 24 major and minor triads as pitch-class segments in {\it dualistic root position}: this means majors are written in root position, e.g. $\langle 0,4,7\rangle$, while minors are written in reverse root position, e.g. $\langle 7,3,0 \rangle$. This reverse ordering of the minor triads ensures that the $TI$-orbit of the pitch-class segment $\langle 0,4,7 \rangle$ is the set of 24 major and minor triads under consideration here. Similarly, dominant seventh chords are written in root position order, e.g. $\langle 0 , 4, 7, 10 \rangle$, while half-diminished seventh chords are written in reverse order, e.g. $\langle 7, 3, 0, 9 \rangle$ so that half-diminished chords have their root in the last entry. The $TI$-orbit of $\langle 0 , 4, 7, 10 \rangle$ is this set of 24 dominant seventh chords and half-diminished seventh chords. Each major triad sits in a unique dominant seventh chord, each minor triad sits in a unique half-diminished seventh chord, both in the first three notes, as we see here with $\langle 0, 4, 7 \rangle$ and $\langle 7,3,0 \rangle$ inside their seventh chords. The associated extension function $f$ from consonant triads to dominant sevenths and half-diminished sevenths is a $TI$-equivariant bijection from the $TI$-orbit of $\langle 0, 4, 7 \rangle$ to the $TI$-orbit of $\langle 0, 4, 7, 10 \rangle$. See Figure~\ref{fig:I7_commutes_with_f}. All of this is explained in more depth in Section~\ref{sec:notations}.} A prominent musical progression involving this conglomeration is the omnibus progression excerpt explored in \cite{hook2007cross}, pictured here in Figure~\ref{fig:omnibus}.

\begin{figure}
\begin{center}
This figure is Example 1 on page 2 of \cite{hook2007cross}, with a slight adaptation. \\
The figure is {\it not} included here because it is protected by copyright.
\end{center}
\caption{This omnibus progression excerpt (here starting on $a$ minor) involves the group actions of the $TI$-group and the $PLR$-group on the set of 24 major and minor triads and also on the set of 24 dominant and half-diminished seventh chords, though only some of those chords and transformations are engaged here. This is Example 1 on page 2 of \cite{hook2007cross}, with our addition of pitch-class segment labels and blue circles on the consonant triads. In Hook's discussion, pitch-class sets (without doubling) are used instead of selected pitch-class segments, so his $L'$ function from the pc-set class of consonant triads to the pc-set class of dominant sevenths and half-diminished sevenths is verbally defined by  ``$L'(\text{triad})$ equals that unique dominant seventh or half-diminished seventh chord that contains all the pitch classes of $L(\text{triad})$,'' as visible on the staff above. In the network, in the commutativity of $L'$ and $T_9$, we partially see the duality of the unified conglomeration as proved in Theorem~\ref{thm:construction_of_action_on_disjoint_union}, Example~\ref{examp:triadsandsevenths}, and Example~\ref{examp:omnibus} of the present article. The present article does not consider the particular voicing of the chords pictured here; instead, we focus on a selected unique ordering of each pitch-class set {\it without} octave doubling. A few comments about Hook's label ``$T_9 \;(\text{or }PR)$'' on the top horizontal arrows may be helpful. This label indicates that the analyst may choose to label the top horizontal arrows as $T_9$ or ``$PR$''. In fact, since Hook uses the ``left-function-first'' orthography, his ``$PR$'' is actually $RP$ in our paper, namely $RP(a)=R(A)=f\sh$. Moreover, $RP$ in our paper (Hook's ``$PR$'') is the {\it Schritt} $Q_{-9}=Q_3$. Later in Figure~\ref{fig:omnibus_network}, we take a different approach: we replace both bottom horizontal $T_9$ arrows of Figure~\ref{fig:omnibus} with $Q_9$, and use duality to transpose the entire 4-chord sequences.} \label{fig:omnibus}
\end{figure}

A natural bijection (we call it $f$) in this situation extends a major triad to the dominant seventh containing it, and extends a minor triad to the half-diminished seventh containing it. Written with a convenient choice of pc-segments, this bijection evaluated on $C$ and $c$ is
\begin{align}
f(C) = C^7,& \qquad f\langle 0, 4, 7\rangle = \langle 0, 4, 7, 10 \rangle \label{equ:Cmajorextension} \\
f(c) = A\textsuperscript{\o{}},& \qquad f\langle 7, 3, 0 \rangle = \langle 7,3,0,9 \rangle. \label{equ:cminorextension}
\end{align}
Clearly, this bijection $f$ is compatible with transposition, in that $f \circ T_i = T_i \circ f$. It is also compatible with all inversions $I_i$, as exemplified by the special case of $I_7$ in Figure~\ref{fig:I7_commutes_with_f}.
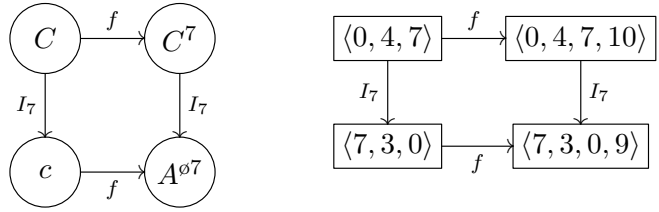
\begin{figure}
$$\entrymodifiers={=<2.2pc>[o][F-]}
\xymatrix{C \ar[r]^f \ar[d]_{I_7} & C^7 \ar[d]^{I_7} \\
c \ar[r]_f & A\textsuperscript{\o{}7}} \qquad \qquad
\entrymodifiers={+<2mm>[F-]}
\xymatrix{\langle 0,4,7 \rangle \ar[r]^-f \ar[d]_{I_7} & \langle 0,4,7,10 \rangle \ar[d]^{I_7} \\ \langle 7, 3, 0 \rangle \ar[r]_-f & \langle 7,3,0,9 \rangle}$$
\caption{The bijection $f$ extends a major triad to the dominant seventh containing it, and extends a minor triad to the half-diminished seventh containing it. As a special case of compatibility of the bijection $f$ with all inversions $I_k$, we see here the compatibility of $f$ with $I_7$ on the $C$ major chord. The left diagram shows the chord names, the right diagram shows the corresponding pitch-class segments. The circles and boxes indicate the same thing: each node is an individual object that is mapped, not a domain set or codomain set of the adjacent arrow.} \label{fig:I7_commutes_with_f}
\end{figure}
To summarize, we have described here our point of departure in a common situation: the $TI$-group acting simply transitively on two different sets that are connected with a $TI$-equivariant bijection.

The first main contribution of the present article is Theorem~\ref{thm:construction_of_action_on_disjoint_union}. It abstracts the algebraic essence of how these two $TI$-group actions on consonant triads and dominant/half-diminished seventh chords combine with the aforementioned equivariant bijection $f$ to form {\it one} larger simply transitive group action on the 48-element (disjoint) union, {\it and} to similarly construct the Lewinian dual\footnote{Recall that two subgroups $G$ and $H$ of a symmetric group $\text{Sym}(S)$ are {\it dual in the sense of Lewin} or are {\it Lewin dual} if each acts simply transitively and each is the centralizer of the other. Often we just write {\it dual} for {\it Lewin dual}.} of this larger group out of the familiar $PLR$-group acting on consonant triads and dominant/half-diminished seventh chords.\footnote{There is one difference between the extension of the $TI$-group and the extension of the $PLR$-group: the $TI$-group acts on both sets already at the outset, but the $PLR$-group acts only on the consonant triads at the outset, its action on dominant/half-diminished seventh chords arises by transferring the consonant triad action across the bijection $f$.}  Each of these two larger groups is the internal {\it direct} product of its initial smaller group with an order 2 element $\overline{f} := f \bigsqcup f^{-1}$, inside of the symmetric group on the 48 chords. The duality is already visible in the omnibus excerpt in Figure~\ref{fig:omnibus}: Hook's $L'$ functionally extends to $\overline{f} \circ \overline{L}$ in our framework, where $\overline{L}$ is $L$ on consonant triads and is $fLf^{-1}$ on dominant/half-diminished sevenths. Consequently, on consonant triads $\overline{f} \circ \overline{L}$ is $f \circ L$, and on dominant/half-diminished sevenths $\overline{f} \circ \overline{L}$ is $L \circ f^{-1}$. This extended $L'$ as $\overline{f} \circ \overline{L}$ and Hook's original $L'$ both commute with $T_9$ in the network.

Our second main contribution of the present article is to adapt the foregoing group extension construction to two simply transitive group actions connected by a $G$-anti-equivariant bijection $f\colon\; S_1 \to S_2$ in Theorem~\ref{thm:anti-equivariant}. To say $f$ is {\it $G$-anti-equivariant} means
$f( g \cdot s)=g^{-1} \cdot f(s)$ for all $g \in G$ and all $s \in S_1$.
In this case, the resulting group is not an internal direct product, but an internal {\it semi-direct} product. In this way, we reconstruct the $TI$-group, the $PLR$-group, and the generalized contextual group associated to a pitch-class segment. Duality in the anti-equivariant case is more involved, and requires additional ingredients $(H,k)$ up front, as proved in Theorem~\ref{thm:anti-equivariant} and applied in Examples~\ref{examp:anti-equivariant_duality_TI-PLR} and \ref{examp:generalizedcontextualgroupagain}. An important earlier paper on building neo-Riemannian groups through extensions was \cite{popoff2013}.

Our third main contribution is Theorem~\ref{thm:generalized}, which constructs a dual pair of groups on the disjoint union of multiple sets, out of equivariant bijections and a dual pair on just one of the sets. To give the main idea of the algebraic outcome of Theorem~\ref{thm:generalized}, we briefly sketch a direct proof of a simplified special case here, without the need of our full algebraic theorems. In Figure~\ref{fig:basicexample_multiset}, we have the sets of consonant triads, major seventh chords, dominant/half-diminished seventh chords, and minor seventh chords, each realized as the $TI$-orbit of the corresponding pitch-class segment in the right diagram. The right diagram also suggests $TI$-equivariant bijections by way of single-element assignments. The disjoint union of the maps in the left diagram is a self-bijection $\overline{f}$ on the disjoint union $\overline{S}$ of the four sets, and this disjoint union of maps $\overline{f}$ forms a ``meta-rotation'' of order 4. On each of these four sets, the $TI$-group acts simply transitively, and the generated group
\begin{equation} \label{equ:basic_TI_extension}
\langle TI\text{-group},\; \overline{f}\rangle \leq \text{Sym}(\overline{S})
\end{equation}
is an internal direct product because $\overline{f}$ commutes with all transpositions and inversions. The internal direct product is isomorphic to the external direct product
\begin{align*}
\langle TI\text{-group},\; \overline{f}\rangle & \cong TI\text{-group} \times \langle \overline{f} \rangle \\ & \cong TI\text{-group} \times \mathbb{Z}_4
\end{align*}
so this group has $24 \times 4 = 96$ elements. It acts simply transitively on the union $\overline{S}$ of the four sets in the left diagram of Figure~\ref{fig:basicexample_multiset}. On the other hand, we have the $PLR$-group acting simply transitively on \triads. The $PLR$-group is the dual group to the $TI$-group on \triads. The generators $P$, $L$, and $R$ (with our choice of dualistic root position ordering for consonant triads) are the same as the contextual inversions $K_{1,3}$, $K_{2,3}$, and $K_{1,2}$ (as in equation \eqref{equ:contextualinversion_Kij_definition} but for 3-tuples). These three contextual inversions also act on the other three sets in Figure~\ref{fig:basicexample_multiset}, via formula \eqref{equ:contextualinversion_Kij_definition}, so we have a simply transitive action of the $PLR$-group on all four sets. All four maps in Figure~\ref{fig:basicexample_multiset} commute with $K_{i,j}$ for $1 \leq i,j \leq 3$ (by Theorems~\ref{thm:conjugation_with_extension} and \ref{thm:modifications}), so $\overline{f}$ commutes with these $PLR$-group actions. Finally, the generated group
\begin{equation} \label{equ:basic_PLR_extension}
\langle PLR\text{-group},\; \overline{f}\rangle \leq \text{Sym}(\overline{S})
\end{equation}
is an internal direct product, isomorphic to the external direct product
\begin{align*}
\langle PLR\text{-group},\; \overline{f}\rangle &\cong PLR\text{-group} \times \langle \overline{f} \rangle
\\ & \cong PLR\text{-group} \times \mathbb{Z}_4
\end{align*}
so this group also has $24 \times 4 = 96$ elements. The two groups in \eqref{equ:basic_TI_extension} and \eqref{equ:basic_PLR_extension} are dual groups inside $\text{Sym}(\overline{S})$, each acting simply transitively on the union of consonant triads, major seventh chords, dominant/half-diminished seventh chords, and minor seventh chords. This example is a motivation for Theorem~\ref{thm:generalized}, which systematizes this construction, and allows the inclusion of diminished seventh chords as well, though the action on diminished seventh chords will not be through contextual inversions, rather through conjugations of contextual inversions with a modification. The $PLRQ$-group of
\cite{CannasMCM2017} does not arise in this way, because their $PLRQ$-group does not act simply transitively and does not commute with inversion $I_k$.

\begin{figure}
$$\xymatrix{\triads \ar[r] & \majorsevenths \ar[d] \\
\minorsevenths \ar[u] & \dom \ar[l] }
\quad \quad
\xymatrix{\langle 0,4,7 \rangle \ar@{|->}[r] & \langle 0,4,7,11\rangle \ar@{|->}[d] \\
\langle 0,4,7,9 \rangle \ar@{|->}[u] & \langle 0,4,7,10 \rangle \ar@{|->}[l]}
$$
\caption{The union of the four maps in the left diagram is a self-bijection of the union of the four sets. This self-bijection is an example of a meta-rotation called $\overline{f}$. In the right diagram, we indicate the functions only on a single element. All four functions are $TI$-equivariant, so their values on all other elements of the respective $TI$-classes can be computed from the right diagram. } \label{fig:basicexample_multiset}
\end{figure}
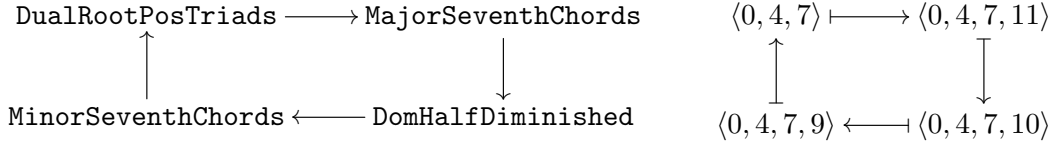

\bigskip

{\bf Overview of related literature.} The circle of ideas discussed in this paper has a long and illustrious history, here we mention a few citations. The present paper is quite self-contained, so new readers can continue without a full study of the vast literature. Neo-Riemannian theory was initiated and pioneered in \cite{LewinGMIT}; the famous neo-Riemmanian operations $P$, $L$, and $R$ were defined on page 178 of that book, and duality is defined on page 253. For a music-theoretical introduction and survey, see \cite{cohn1997} and \cite{cohnsurvey}. For a short mathematically-oriented introduction to the $PLR$-group and duality, see \cite{cransfioresatyendra}. The main proof of \cite{cransfioresatyendra} was improved in \cite{BerryFiore}, and a few gaps in the literature on dual groups were filled there as well. Volume 42 Issue 2 of the {\it Journal of Music Theory} from 1998 is a fascinating presentation of the neo-Riemannian research up to that year. That issue also included several graphical visualizations of groups in \cite{DouthettSteinbach}, one of which was discovered much earlier in \cite{Waller}. The group structure of the subgroup of $SL(3, \mathbb{Z}_{12})$ generated by the naive linear extensions of $P$, $L$ and $R$ was completely determined in \cite{fiorenoll_SIAM}. Permutations were included into the duality between the $TI$-group and the $PLR$-group in \cite{fiorenollsatyendraMCM2013}.

The analogues of $P$, $L$, and $R$ on dominant sevenths and half-diminished sevenths were studied in \cite{childs} and \cite{gollin}, without algebraically connecting them to consonant triads. In that same issue \cite{DouthettSteinbach} graphically connected consonant triads and augmented triads in their {\it cube dance} graph, but did not connect with sevenths. They made a {\it separate} graph of dominant sevenths, half-diminished sevenths, minor sevenths, and diminished sevenths called {\it power towers} in \cite{DouthettSteinbach}. Later, beyond triads and sevenths, a kind of $PLR$-group for general pitch-class segments containing at least one non-trivial interval besides a tritone was developed in \cite{fioresatyendra2005}. The larger group of {\it uniform triadic transformations} acting on root-parity pairs was introduced and thoroughly studied in \cite{hookUTTthesis2002} and \cite{hookUTT2002}. Pages 111 - 113 of \cite{hookUTT2002} applied uniform triadic transformations to dominant seventh chords and half-diminished seventh chords, and connected with \cite{childs} and \cite{gollin}. \cite{arnettbarth_sevenths} considered dominant, half-diminished, and minor seventh chords (36 chords in total) and introduced 5 functions between certain pairs of these classes, in such a way that three pitch classes are preserved and one pitch class moves by a semitone or tone. \cite{KerkezBridges} considered only major seventh chords in root position and minor seventh chords in reverse root position, and introduced functions $P$ and $S$ that preserve 3 out 4 notes and generate a dihedral group of order 24. \cite{CannasMCM2017} introduced a large group of order $5! \times 12 \times 4 = 11,520$ called the $PLRQ$-group, which acts (non-simply transitively) on dominant, half-diminished, minor, major, and diminished seventh chords, without consonant triads.

\cite{hook2007cross} used ``cross-type transformations'' to connect consonant triads with dominant and half-diminished sevenths in two separate generalized interval systems, and developed homomorphisms of generalized interval systems. Hook's motivating example in that paper is the omnibus progression reproduced here in Figure~\ref{fig:omnibus} (a different analysis of the omnibus progression is on page 116 of \cite{hookUTT2002}). For any major or minor triad $X$, Hook defined $L'(X)$ to be ``the unique major-minor or half-diminished seventh chord which contains all the notes of $L(X)$.''\footnote{Hook's definition of $L'$ is found in the right column of page 2 of \cite{hook2007cross}. There $X$, $L(X)$, and $L'(X)$ are implicitly each thought of as pitch-class sets rather than pitch-class segments. Recall that major-minor seventh chords are the same as dominant seventh chords.}  Hook remarks ``This $L'$ is a {\it cross-type transformation}, which is to say that it transforms objects of one type (triads) into objects of a different type (seventh chords)...''\footnote{This definition of cross-type transformation is in the left column of page 3 of \cite{hook2007cross}.}  Hook's $L'$ is indeed a homomorphism of the two generalized interval systems corresponding to the two $TI$-group actions on consonant triads and on dominant/half-diminished seventh chords.\footnote{Recall from the proof on pages 157-158 of \cite{LewinGMIT} that a generalized interval system is essentially the same as a simply transitive group action (the ``transpositions'' of the generalized interval system form the simply transitive group). Because of this correspondence, to show that $L'$ is a homomorphism of generalized interval systems, it is sufficient to observe that $L'$ is equivariant with respect to the $TI$-actions on the two sets, rather than spelling out what the interval functions are and working at the level of generalized interval systems. More precisely, we rely on the fact that the category of generalized interval systems is equivalent to the category of simply transitive group actions as proved in Proposition~2.9 of \cite{fiorenollsatyendraSchoenberg}. Both of these categories are equivalent to the category of groups with morphisms given by affine maps of groups ({\it not} just group homomorphisms), see Proposition~2.10 of \cite{fiorenollsatyendraSchoenberg} and see also \cite{kolman}.} \cite{hook2007cross} developed the theory of certain homomorphisms called {\it cross-type transposition} and {\it cross-type inversion}, and worked out reference points, label functions, and intervals as related to homomorphisms of generalized interval systems. Homomorphisms of generalized interval systems were studied earlier in \cite{kolman}, and subsequently in \cite{fiorenollsatyendraSchoenberg}.\footnote{The {\it morphisms} of generalized interval systems in \cite{fiorenollsatyendraSchoenberg} are the same as the {\it homomorphisms} of generalized interval systems discussed in \cite{kolman} and \cite{hook2007cross}.} To summarize this paragraph: our first main contribution of the present article solves the cross-type transformation problem in a different way than \cite{hook2007cross}, and we additionally find the Lewinian dual group in the present paper.

Although \cite{hook2007cross} kept consonant triads and dominant/half-diminished seventh chords in distinct generalized interval systems and connected them with a homomorphism, Hook did herald the perspective of our present article. He wrote:
\begin{quotation}
It will have occurred to many readers that one can readily form the set-theoretic union $S \cup T$ of the two sets—or more generally a single set combining objects of all the `types' under consideration—and can then
undertake to define appropriate functions on this larger set. Such mappings can then be considered single-type transformations in the Lewinian sense. In some cases this approach can be made to work, effectively circumventing the limitation of Lewin’s definition. For instance, the triad-to-seventh-chord mapping $L'$ described above may be taken together with its seventh-chord-to-triad inverse mapping to define a single function on the set of all 48 triads and seventh chords of the appropriate qualities, by means of which the omnibus analysis of Example 1 may be reconstructed (Footnote 8 on page 5 of \cite{hook2007cross}).
\end{quotation}
The present article adopts this perspective, and develops Lewinian duality and group extensions in this context.

\bigskip

{\bf Overview of section topics.}
In Section~\ref{sec:notations}, we very briefly recall the $P$, $L$, and $R$ transformations, dominant seventh chords, half-diminished seventh chords, and generalized contextual groups, and then focus on the consequences of extending $P$, $L$, and $R$ to dominant and half-diminished seventh chords via conjugation. Along the way we introduce notations for the rest of the article and shed some light on generalized contextual groups. This detailed and concrete treatment of the special case of dominant and half-diminished seventh chords is provided up front, so that later analogous arguments can be abbreviated.

Section~\ref{sec:construction_for_disjoint_union} contains our first main result, Theorem~\ref{thm:construction_of_action_on_disjoint_union}, which allows us to construct a simply transitive group action on the disjoint union of consonant triads with dominant and half-diminished sevenths in Example~\ref{examp:triadsandsevenths}. To make Theorem~\ref{thm:construction_of_action_on_disjoint_union} even more applicable, we develop several general $TI$-equivariant bijections that can be used for $f$ in the theorem: pitch segment extension, pitch segment truncation, and pitch segment modification.

Section~\ref{sec:anti-equivariant_disjoint_union} contains our second main result, Theorem~\ref{thm:anti-equivariant}, which allows us to reconstruct familiar transformational groups as internal semi-direct products, and express them as short exact sequences.

Section~\ref{sec:bijections_of_triads_and_seventh_chords} very briefly reviews major seventh chords, minor seventh chords, and diminished seventh chords, and then introduces $TI$-equivariant bijections from consonant triads to various $TI$-orbits of seventh pitch-class segments. The formulas for these $TI$-equivariant bijections are linear, and are compatible with every {\it Schritt} $Q_i$. All but the last of the bijections are compatible with contextual inversions. The purpose of this section is to develop the raw materials for making star-shaped diagrams as input to Theorem~\ref{thm:generalized}.

Section~\ref{sec:construction_for_disjoint_union_multiple} contains our third main result, Theorem~\ref{thm:generalized}, which constructs from a star-shaped diagram of equivariant bijections a simply transitive group action on the disjoint union of the sets in the diagram. The star-shaped diagram produces a ``meta-rotation'', as we can see in the rotation of the tetractys modes in Figure~\ref{fig:rotation_is_metarotation} and Example~\ref{examp:tetractys}. The main example of Theorem~\ref{thm:generalized} is the construction of a group that acts simply transitively on consonant triads and several kinds of seventh chords. Analytical examples include ``Autumn Leaves'' and dynamic voice leading of root position sevenths with the $J$ function.

All three of the main theorems include duality.

\section{Extension of the neo-Riemannian $P$, $L$, and $R$ transformations to dominant seventh chords and half-diminished seventh chords} \label{sec:notations}

In light of the motivating omnibus progression example, we now quickly recall the neo-Riemannian $P$, $L$, and $R$ transformations, set up notations for the rest of the article, remind mathematical readers of dominant seventh chords and half-diminished seventh chords, and extend $P$, $L$, and $R$ to these seventh chords via ``conjugation''. This section provides some formulas and conventions so that we can quickly apply Theorem~\ref{thm:construction_of_action_on_disjoint_union} in Examples~\ref{examp:triadsandsevenths} and \ref{examp:omnibus} to treat the motivating omnibus progression in Figure~\ref{fig:omnibus}. Moreover, working out the details for dominant seventh chords and half-diminished seventh chords in the present section will allow us to more quickly treat other kinds of sevenths in an analogous way in Sections~\ref{sec:bijections_of_triads_and_seventh_chords} and \ref{sec:construction_for_disjoint_union_multiple}.

The article \cite{cransfioresatyendra} is a self-contained mathematically-oriented exposition of Lewin's neo-Riemannian $PLR$-group. We use the same formalizations here. Recall that a {\it pitch-class segment} from musical set theory is an ordered tuple of elements of $\mathbb{Z}_{12}$, always written with angular brackets. For example the $C$ triad in root position is $\langle 0, 4, 7 \rangle$, and the $c$ triad\footnote{As usual, major triads are indicated with capital letters and minor triads are indicated with lowercase letters, here $C$ versus $c$.} in reverse root position is $\langle 7, 3, 0 \rangle$. An analyst is free to consider $\langle 0, 4, 7 \rangle$ as a representation of the $C$ chord in any voicing, or specifically root position, or as consecutive notes, or even as a stand-in for the underlying pitch-class set. Let $\triads$ be the collection of 24 consonant triads written as pitch-class segments in {\it dualistic root position}, this means major triads are written in root position order and minor triads are written in reverse root position order.
\begin{align*}
\triads &:= \{\;\langle w,\; w+4,\; w+7 \rangle \; \vert \; w \in \mathbb{Z}_{12} \;\}\; \\ &\bigcup \;\{ \; \langle w,\; w-4,\; w-7 \rangle \; \vert \; w \in \mathbb{Z}_{12} \;\}
\end{align*}
The root of a minor chord $\langle w,\; w-4,\; w-7 \rangle$ is the last entry $w-7$, {\it not} $w$.
The collection of dualistic root position triads as pitch-class segments is equal to the orbit of $C$ major under the entry-wise action of the $TI$-group, whose elements are $T_n(w) = w+n$ and $I_n(w) = -w+n$ for $n \in \mathbb{Z}_{12}$.
\begin{align*}
\triads = TI \langle 0, 4, 7 \rangle
\end{align*}

The representation of the 24 consonant triads in this way enables convenient formulas for the neo-Riemannian $P$, $L$, and $R$ transformations long understood as {\it contextual inversions}:
\begin{equation} \label{equ:plr}
\begin{aligned}
P\langle w, x, y \rangle &= I_{w+y}\langle w, x, y \rangle = \langle y,&\;      &-x+w+y,&\;   &w\rangle \phantom{.} \\
L\langle w, x, y \rangle &= I_{x+y}\langle w, x, y \rangle = \langle -w+x+y,&\; &y,&\;       &x\rangle \phantom{.} \\
R\langle w, x, y \rangle &= I_{w+x}\langle w, x, y \rangle = \langle x,&\;       &w,&\;       &-y+w+x\rangle.
\end{aligned}
\end{equation}
For these formulas, it is helpful to know that $I_{a+b}$ is the unique  inversion that exchanges $a$ and $b$, so that each of the transformations in \eqref{equ:plr} preserves two pitch-classes of the underlying input pitch-class set.\footnote{Notice these formulas in \eqref{equ:plr} do {\it not} imply that $P$, $L$, and $R$ are linear, because these formulas for $P$, $L$, and $R$ only hold when the input $\langle w, x, y\rangle$ is a consonant triad in dualistic root position. However, these formulas were adapted (non-linearly) to any voicing in \cite[Example~3.3]{fiorenollsatyendraMCM2013}. The naive extensions of the formulas in \eqref{equ:plr} to linear maps $\mathbb{Z}_{12}^{\times 3} \to \mathbb{Z}_{12}^{\times 3}$ generate a subgroup of $SL(3, \mathbb{Z}_{12})$; its structure was completely determined in \cite{fiorenoll_SIAM}.
Also notice the formulas in \eqref{equ:plr} do {\it not} say that $P$, $L$, and $R$ are inversions (reflections) $I_n$, because the subscripts $n=a+b$ in \eqref{equ:plr} are determined by the input chord.}

The equations in \eqref{equ:plr} imply
\begin{equation}
\begin{aligned}
P\langle w,\; w+4,\; w+7 \rangle &= \langle w+7,&\;      &w+3,&\;   &w\rangle \\
L\langle w,\; w+4,\; w+7 \rangle &= \langle w+11,&\; &w+7,&\;       &w+4\rangle \\
R\langle w,\; w+4,\; w+7 \rangle &= \langle w+4,&\;       &w,&\;       &w+9\rangle
\end{aligned}
\end{equation}
and similar formulas on reverse root position minor triads. The bijections $P$, $L$, and $R$ generate a subgroup of $\text{Sym}(\triads)$ called the {\it $PLR$-group}. This subgroup is dihedral of order 24, is actually generated by just $L$ and $R$, and is the centralizer of the $TI$-group (see Theorems~5.1 and 6.1 of \cite{cransfioresatyendra} for proofs of these classical results; both proofs were improved in Remark~3.10 of \cite{BerryFiore}).

We would next like to extend $P$, $L$, and $R$ via ``conjugation'' to dominant seventh chords and half-diminished seventh chords in a way that is compatible with the unique extension of a consonant triad to such a seventh, and compatible with the foregoing conventions for encodings of consonant triads. For instance, the formula for $P$ on a dominant seventh chord or half-diminished seventh chord should say something like: ``look inside the dominant/half-diminished seventh chord, find the unique consonant triad, then do the contextual inversion corresponding to $P$ on the entire seventh chord.'' We would also like to see that these extended $P$, $L$, and $R$ have parsimonious voice leading: on the underlying pitch-class sets, each fixes two pitch classes and moves the other two pitch classes by only a half step.

In a {\it dominant seventh chord} in root position, the thirds from root upwards are major third, minor third, and minor third, so for instance $C^7$ in root position is $\langle 0,4,7,10\rangle$. The notation for a dominant 7th chord is a superscript \textsuperscript{7}. Another name for a dominant seventh chord is a {\it major-minor seventh} chord because it consists of a major triad together with a minor seventh above its root. When we write dominant seventh chords as pitch-class segments in this paper, we always use root position. Every major triad sits in a unique dominant seventh chord, and every dominant seventh chord contains a unique major triad. In our pitch-class segment conventions, the contained major triad is always in the first three entries.

Any entry-wise inversion (reflection) of a dominant 7th pc-segment in root position is a {\it half-diminished seventh chord} in reverse root position. The descending intervals are major third, minor third, minor third, and the root of the half-diminished seventh chord is in the last entry. For instance, the reflection of $C^7$ over 0 is a $D\textsuperscript{\o{}7}$ in reverse root position, with root $D=2$ last.
\begin{align*}
I_0\langle 0,4,7,10 \rangle = \langle 0, 8, 5, 2\rangle
\end{align*}
The notation for a half-diminished seventh chord is a superscript \textsuperscript{\o{}7}. Another name for a half-diminished seventh chord is a {\it minor 7 flat 5 chord} because it is a minor seventh chord\footnote{Minor seventh chords are briefly reviewed in Section~\ref{sec:bijections_of_triads_and_seventh_chords}.} with the fifth lowered by a half step. In other words, a half-diminished seventh chord consists of a diminished triad together with a minor seventh above its root.
When we write half-diminished seventh chords as pitch-class segments, we always use reverse root position. Every minor triad sits in a unique half-diminished seventh chord, and every half-diminished seventh chord contains a unique minor triad. In our pitch-class segment conventions, the contained minor triad (itself in reverse root position) is always in the first three entries of its containing reverse root position half-diminished seventh chord. Our dualistic root position convention for dominant seventh chords and half-diminished seventh chords is compatible with our dualistic root position convention for consonant triads, with only one small (unavoidable) hiccup: the inclusion of a minor triad in a half-diminished seventh chord changes the root name by a descending minor third. For instance, here in the present example, $f = \langle 0, 8, 5 \rangle$ is contained in $D\textsuperscript{\o{}7} = \langle 0, 8, 5, 2\rangle$. Fortunately, the inclusion of a major triad preserves the root name. For instance, here in the present example, $C = \langle 0, 4, 7 \rangle$ is contained in $C^7 = \langle 0,4,7,10 \rangle$.
See Figure~\ref{fig:I7_commutes_with_f} for an example with $I_7$ instead of $I_0$.

Let $\dom$ be the collection of 24 dominant seventh chords and half-diminished seventh chords (i.e. 12 of each) written as pitch-class segments in {\it dualistic root position}; this means dominant seventh chords are written in root position order and half-diminished seventh chords are written in reverse root position order.
\begin{align*}
\dom &:= \{\;\langle w,\; w+4,\; w+7,\; w+10 \rangle \; \vert \; w \in \mathbb{Z}_{12} \;\}\; \\&\bigcup \;\{ \; \langle w,\; w-4,\; w-7,\; w-10 \rangle \; \vert \; w \in \mathbb{Z}_{12} \;\}
\end{align*}
The root of a half-diminished seventh chord $\langle w,\; w-4,\; w-7,\; w-10 \rangle$ is of course the last entry $w-10$, {\it not} $w$.
The collection of dualistic root position dominant sevenths and half-diminished sevenths as pitch-class segments is equal to the orbit of $C^7$ under the entry-wise action of the $TI$-group.
\begin{align*}
\dom = TI \langle 0, 4, 7, 10 \rangle
\end{align*}

\begin{proposition} \label{prop:TI-equivariance_contextual_extension}
The bijection
\begin{equation} \label{equ:bijection_triads_domhalfdiminished}
\begin{aligned}
\xymatrix{f \colon \; \text{\rm \triads} \ar[r] & \text{\rm \dom}} \\
\xymatrix{\langle w,\; w+4,\; w+7\rangle \ar@{|->}[r] & \langle w,\; w+4,\; w+7,\; w+10 \rangle} \\
\xymatrix{ \langle w,\; w-4,\; w-7\rangle \ar@{|->}[r] &\langle w,\; w-4,\; w-7,\; w-10 \rangle }
\end{aligned}
\end{equation}
is $TI$-equivariant. In fact, this bijection is uniformly defined on both majors and minors with the single ``contextual extension'' formula
\begin{align} \label{equ:2y-x}
f\langle w,\; x,\; y \rangle = \langle w,\; x,\; y,\; 2y-x \rangle,
\end{align}
from which $TI$-equivariance follows.
\end{proposition}
\begin{proof}
The expression $2y-x$ in \eqref{equ:2y-x} is discovered by attempting to write 10 in terms of 4 and 7 modulo 12. This is
\begin{align*}
10 \equiv 2 \times 7 - 4
\end{align*}
so that on a major triad
\begin{align*}
2y-x &= 2(w+7) - (w+4) \\
&=  w+10\phantom{,}
\end{align*}
and on a minor triad
\begin{align*}
2y-x &= 2(w-7) - (w-4) \\
&= w-10,
\end{align*}
so $f$ is equal to the uniform formula in \eqref{equ:2y-x}.

To confirm the $TI$-equivariance of $f$, it is sufficient to confirm $2y-x$ is $TI$-equivariant, because the first three entries of $f$ are identities. For transpositions, we have
\begin{align*}
2(y+n) - (x+n) &= 2y+ 2n - x - n \\
&= 2y-x +n
\end{align*}
The scalar compatibility of the linear expression $2y-x$ implies compatibility with $I_0$, so we also have equivariance with respect to all $I_n = T_n \circ I_0.$
\end{proof}

The $TI$-equivariant bijection $f$ in \eqref{equ:bijection_triads_domhalfdiminished} and \eqref{equ:2y-x}
allows us to define $P$, $L$, and $R$ on $\dom$ as $fPf^{-1}$, $fLf^{-1}$, and $fRf^{-1}$. In fact, the entire $PLR$-group acts on $\dom$ in this way through conjugation. The conjugation of {\it Schritt} transformations $Q_n$ are again {\it Schritt} transformations, but conjugated $P$, $L$, and $R$ require more discussion.  Their $f$-conjugated action is tantamount to: ``look inside the seventh chord, find the unique consonant triad, then do the corresponding contextual inversion on the entire seventh chord.''

\begin{proposition} \label{prop:plr_conjugated_action}
Let $f$ be as in Proposition~\ref{prop:TI-equivariance_contextual_extension}. The $f$-conjugations of $P$, $L$, and $R$ on {\rm $\dom$} are the same contextual inversions as on {\rm $\triads$}. This means for $\langle w, x, y, z \rangle \in \text{\rm \dom}$,
\begin{equation*}
\begin{aligned}
fPf^{-1}\langle w, x, y, z \rangle &= I_{w+y}\langle w, x, y, z \rangle = \langle y,&\;      &-x+w+y,&\;   &w,&\;  &-z+w+y\rangle \phantom{.} \\
fLf^{-1}\langle w, x, y, z \rangle &= I_{x+y}\langle w, x, y, z \rangle = \langle -w+x+y,&\; &y,&\;       &x,&\;   &-z+x+y\rangle \phantom{.} \\
fRf^{-1}\langle w, x, y, z \rangle &= I_{w+x}\langle w, x, y, z \rangle = \langle x,&\;       &w,&\;       &-y+w+x,&\; &-z+w+x\rangle.
\end{aligned}
\end{equation*}

We write $P$, $L$, and $R$ also for these conjugated actions, so the above equations imply

\begin{equation*}
\begin{aligned}
P\langle w,\; \textcolor{red}{w+4},\; w+7,\; \textcolor{blue}{w+10} \rangle &= \langle w+7,&\;      &\textcolor{red}{w+3},&\;   &w,&\;   &\textcolor{blue}{w-3}\rangle \\
L\langle \textcolor{red}{w},\; w+4,\; w+7,\; \textcolor{blue}{w+10} \rangle &= \langle \textcolor{blue}{w+11},&\; &w+7,&\;       &w+4,&\;   &\textcolor{red}{w+1}\rangle \\
R\langle w,\; w+4,\; \textcolor{red}{w+7},\; \textcolor{blue}{w+10} \rangle &= \langle w+4,&\;       &w,&\;       &\textcolor{blue}{w+9},&\;   &\textcolor{red}{w+6}\rangle.
\end{aligned}
\end{equation*}
Similar formulas on reverse root position half-diminished seventh chords are also visible here, using the fact that $P$, $L$, and $R$ are self-inverse.

Moreover, at the level of pitch-class sets, each of $P$, $L$, and $R$ preserves two pitch classes and moves two pitch classes by a half step. For each of the these three transformations, the two $\textcolor{red}{red}$ pitch classes indicate motion by a semitone.
Similarly, for each of the these three transformations, the two  $\textcolor{blue}{blue}$ pitch classes indicate motion by a semitone. The black pitches are stationary (i.e. exchanged by the contextually selected inversion).
\end{proposition}
\begin{proof}
Let $\langle w, x, y, z \rangle$ be a dominant seventh chord or half-diminished seventh chord in dualistic root position, which implies $z = 2y-x$. Then
\begin{align*}
P\langle w,x,y,z\rangle &\overset{\text{def}}{=} fPf^{-1} \langle w,x,y,z\rangle   \\
&= fP\langle w, x, y \rangle \\
&= f I_{w+y} \langle w, x, y \rangle \\
&= I_{w+y} f \langle w, x, y \rangle \\
&= I_{w+y} \langle w,\; x,\; y,\; 2y -x\rangle \\
&= I_{w+y} \langle w,\; x,\; y,\; z\rangle.
\end{align*}
The proofs for $L$ and $R$ are completely analogous.
\end{proof}

\begin{example}
For $C^7$ = $\langle 0,4,7,10 \rangle$, Proposition~\ref{prop:plr_conjugated_action} says
\begin{align*}
fPf^{-1}\langle 0,4,7,10 \rangle &= \langle 7,3,0,9\rangle = A\textsuperscript{\o{}7}\\
fLf^{-1}\langle 0,4,7,10 \rangle &= \langle 11,7,4,1\rangle = C\sh\textsuperscript{\o{}7} \\
fRf^{-1}\langle 0,4,7,10 \rangle &= \langle 4,0,9,6\rangle = F\sh\textsuperscript{\o{}7}
\end{align*}
The roots are 3, 11, and 6 half-steps below the original root $C$. The operation which interchanges dominant seventh and half-diminished seventh chords and keeps the root the same is {\it not} $fPf^{-1}$, but is $K_{1,4}$ defined via equation \eqref{equ:contextualinversion_Kij_definition} and is in the group generated by the $f$-conjugated $P$, $L$, and $R$, as we will see in a moment in Remark~\ref{rem:other_contextual_inversions}.
\end{example}

We next make a few comments about Proposition~\ref{prop:plr_conjugated_action} concerning the isomorphic image of the $PLR$-group under $f$-conjugation, contextual inversions, and the voice leading properties of the $f$-conjugated $P$, $L$, and $R$.

\begin{remark} \label{rem:PLRconjugation_is_gcg}
The subgroup of $\text{Sym}(\dom)$ generated by the conjugates $$fPf^{-1}, \quad fLf^{-1}, \quad \text{and} \quad fRf^{-1}$$ from Proposition~\ref{prop:plr_conjugated_action} is equal to the generalized contextual group associated to the pitch-class segment $\langle 0,4,7,10\rangle$, as we now prove concretely with generators and relations.\footnote{We can also prove the $f$-conjugation of the $PLR$-group is equal to the generalized contextual group associated to $\langle 0, 4, 7, 10\rangle$ theoretically. Namely, $f$ commutes with the $TI$-group, as do $P$, $L$, and $R$, so their $f$-conjugates do as well, and so does the $f$-conjugation of the entire $PLR$-group. Thus the 24-element $f$-conjugation of the $PLR$-group is contained in the 24-element dual group of the $TI$-group acting on its orbit of $\langle 0,4,7,10\rangle$. But the dual group of the $TI$-group acting on its orbit of $\langle 0,4,7,10\rangle$ is equal to the generalized contextual group associated to $\langle 0,4,7,10\rangle$, so the conjugated $PLR$-group is contained in the generalized contextual group associated to $\langle 0,4,7,10\rangle$, and has the same size. Thus the two groups are equal.} The theory of generalized contextual groups was introduced in \cite{fioresatyendra2005}. The generalized contextual group associated to  $\langle 0,4,7,10\rangle$ is the subgroup
\begin{equation*}
\langle Q_1,\; K_{1,2} \rangle \leq \text{Sym}(\dom)
\end{equation*}
with generators $Q_1$ and $K_{1,2}$ defined on a dominant seventh chord or half-diminished seventh chord $Y=\langle y_1,\, y_2,\, y_3,\, y_4\rangle$ in dualistic root position as
\begin{equation*}
Q_1(Y) :=
\begin{cases}
T_1(Y) & \text{if $Y$ is a dominant seventh chord} \\
T_{-1}(Y) & \text{if $Y$ is a half-diminished seventh chord}
\end{cases}
\end{equation*}
and
\begin{equation*}
K_{1,2}(Y) := I_{y_1+y_2}(Y).
\end{equation*}
We know $Q_5 = f LR f^{-1}$ from the proof of Theorem~5.1 in \cite{cransfioresatyendra}, and
\begin{align*}
Q_1 &= (Q_5)^5 \\
&= \left( (fLf^{-1})(fRf^{-1})\right)^5,
\end{align*}
and we also know $K_{1,2}=fRf^{-1}$ from Proposition~\ref{prop:plr_conjugated_action}, so
$$\langle Q_1,\, K_{1,2} \rangle \leq \langle fLf^{-1},\, fRf^{-1} \rangle.$$
On the other hand,
\begin{align*}
fLf^{-1} & =  (fLRf^{-1}) (fRf^{-1}) \\
&= Q_5(fRf^{-1})\\
&= (Q_1)^5K_{1,2},
\end{align*}
so $$\langle Q_1,\, K_{1,2} \rangle \geq \langle fLf^{-1},\, fRf^{-1} \rangle.$$
Lastly recall that the $PLR$-group is generated by just $L$ and $R$, so we now have that the $f$-conjugation of the $PLR$-group is equal to the generalized contextual group associated to the pitch-class segment $\langle 0 , 4, 7 , 10 \rangle$. Proposition~\ref{prop:plr_conjugated_action} and this present remark foreshadow Examples~\ref{examp:triadsandsevenths} and \ref{examp:generalizedcontextualgroup}, as well as the good behavior of the bijections from consonant triads to major seventh chords and minor seventh chords in Section~\ref{sec:bijections_of_triads_and_seventh_chords} in Table~\ref{table:seventhchord_functions}.
\end{remark}

\begin{remark} \label{rem:diminished_sevenths_mystery_solution}
The fact that the generalized contextual groups associated to $\langle 0,4,7\rangle$ and $\langle 0,4,7,10 \rangle$ are generated by contextual inversions without the need for $Q_1$ is quite special amongst generalized contextual groups. The equation
$$m(7-0)+n(7-4) \equiv 1 \text{ mod }12$$
has solution $m = 7$, $n=0$ in $\mathbb{Z}_{12}$, so $Q_1$ on $TI\langle 0,4,7\rangle$ can be expressed in terms of contextual inversions as a consequence of Theorem~3.3.2(5) of \cite{fiorenoll_SIAM} (or by direct computation).  In contrast, the generalized contextual group associated to the diminished triad $\langle 1, 4, 7\rangle$ is not generated by contextual inversions alone: the equation
$$m(7-1)+n(7-4) = 3(2m+n) \equiv 1 \text{ mod }12$$
has no solutions $m,n \in \mathbb{Z}_{12}$, so $Q_1$ on $TI\langle 1,4,7\rangle$ cannot be expressed in terms of contextual inversions (again we are using Theorem~3.3.2(5) of \cite{fiorenoll_SIAM}). Appending 10 to $\langle 1,4,7\rangle$ does not create any new interval classes in the pc-seg $\langle 1,4,7,10\rangle$, so $Q_1$ on $TI\langle 1,4,7,10\rangle$ cannot be expressed in terms of contextual inversions either (we apply Theorem~3.3.2(5) of \cite{fiorenoll_SIAM} to every sub-3-tuple of $\langle 1,4,7,10\rangle$). Thus, Proposition~\ref{prop:plr_conjugated_action} does not hold for the (still $TI$-equivariant) slight modification of $f$ that goes into diminished seventh chords:
\begin{equation*}
\begin{aligned}
\xymatrix{f_\text{mod} \colon \; \text{\rm \triads} \ar[r] & \text{\rm \diminishedsevenths}} \\
\xymatrix{\langle w,\; w+4,\; w+7\rangle \ar@{|->}[r] & \langle w+1,\; w+4,\; w+7,\; w+10 \rangle \phantom{.}} \\
\xymatrix{ \langle w,\; w-4,\; w-7\rangle \ar@{|->}[r] &\langle w-1,\; w-4,\; w-7,\; w-10 \rangle .}
\end{aligned}
\end{equation*}
In more detail: conjugation of the $PLR$-group by this $TI$-equivariant $f_\text{mod}$ does indeed provide an isomorphism onto the dual group of the $TI$-group acting on $TI\langle 0,4,7,10\rangle$ (without use of modified Proposition~\ref{prop:plr_conjugated_action}), and this dual group is equal to the generalized contextual group of $\langle 0 , 4, 7, 10 \rangle$, so if Proposition~\ref{prop:plr_conjugated_action} did hold for this slightly modified $f$, that would mean the generalized contextual group of $\langle 1, 4, 7, 10 \rangle$ is generated by contextual inversions alone, which it is not. We can also directly compute that the conjugation of $P$ by $f_\text{mod}$ is not a contextual inversion (similar computation for the conjugations of other contextual inversions).
\begin{align*}
    f_\text{mod}Pf^{-1}_\text{mod} \langle w+1,\, w+4,\, w+7,\, w+10 \rangle & = f_\text{mod}P \langle w,\, w+4,\, w+7 \rangle \\
    & = f_\text{mod} \langle w+7,\, w+3,\, w \rangle \\
    & = \langle w+6,\, w+3,\, w,\, w-3 \rangle
\end{align*}
We see the input and final output pitch-class segments have not even a single pitch class in common, so $f_\text{mod}Pf^{-1}_\text{mod}$ is not a contextual inversion (though is in the generalized contextual group).

This extended remark about the diminished seventh chord pitch-class segment $\langle 1, 4, 7, 10 \rangle$ foreshadows the special case of the bijection from consonant triads to diminished seventh chords in Section~\ref{sec:bijections_of_triads_and_seventh_chords} in Table~\ref{table:seventhchord_functions}. See also Example~\ref{examp:alteration_diminished_sevenths}.
\end{remark}

We next make a few observations about the moving pitch-classes in Proposition~\ref{prop:plr_conjugated_action}.

\begin{remark}
In Proposition~\ref{prop:plr_conjugated_action}, in the $P$ transformation on dominant seventh chords and half-diminished seventh chords, the two red pitch classes have the same index, as do the two blue pitch classes. This is not the case for $L$ and $R$. Another way $P$ distinguishes itself is that in $P$, one of the motion pairs (\textcolor{red}{$w+4$} and \textcolor{red}{$w+3$}) is the same motion pair that $P$ does on the underlying major and minor triads. This is not the case for $L$ and $R$. In $L$ on consonant triads, \textcolor{red}{$w$} is matched with \textcolor{blue}{$w+11$}, but on sevenths, $L$ matches \textcolor{red}{$w$} with \textcolor{red}{$w+1$} so that \textcolor{blue}{$w+11$} is free to match with \textcolor{blue}{$w+10$}, and we get two pairs of semitonal motion inside $L$. In $R$ on consonant triads, \textcolor{red}{$w+7$} is matched with \textcolor{blue}{$w+9$} via a whole step, but on sevenths, $R$ matches \textcolor{red}{$w+7$} with \textcolor{red}{$w+6$} so that \textcolor{blue}{$w+9$} is free to match with \textcolor{blue}{$w+10$}, and we get two pairs of semitonal motion inside $R$. It will be an interesting research question to incorporate permutations (as in \cite{fiorenollsatyendraMCM2013}) into the theory of the present paper, and to determine whether or not the Cohn group\footnote{The {\it Cohn group} is generated by $P'$, $L'$, and $R'$, which are certain composites of conjugation-extended $P$, $L$, and $R$ with dual permutations. These variants are defined so that preserved pitch-classes stay in the same voice, i.e. $P'\langle 0,4,7 \rangle = \langle 0, 3, 7 \rangle$, and $L'\langle 0,4,7 \rangle = \langle 11, 4, 7 \rangle$, and $R'\langle 0,4,7 \rangle = \langle 0, 4, 9 \rangle$. See Examples~3.3 and 3.4 of \cite{fiorenollsatyendraMCM2013}. The Cohn group is {\it not} dihedral of order 24, as incorrectly claimed in Example~3.4 of that paper. The Cohn group actually has 72 elements. We thank Dmitri Tymoczko for notifying us of the correction to that 2013 paper.} of \cite{fiorenollsatyendraMCM2013} can be extended in a way that resolves the changed voice leading matching when $P$, $L$, and $R$ are extended to dominant seventh chords and half-diminished chords via conjugation with \eqref{equ:bijection_triads_domhalfdiminished} and \eqref{equ:2y-x}.
\end{remark}

\begin{remark}
In all three transformations, the two moving voices move in parallel motion; however, in $P$ and $R$ the motion is downward, while in $L$ the motion is upward.
\end{remark}

\begin{remark} \label{rem:other_contextual_inversions}
In Proposition~\ref{prop:plr_conjugated_action}, we see that $P$, $L$, and $R$ on $\dom$ are contextual inversions that exchange the pitch classes at three pairs of pre-selected indices. But in a tuple with four entries, there are of course $4!/(2! \cdot 2!)=6$ possible pairs that could be selected. Where are the other three contextual inversions $K_{1,4}$, $K_{2,4}$, and $K_{3,4}$? Here $K_{i,j}(Y)$ is the reflection of $Y$ that interchanges $y_i$ and $y_j$, this is
\begin{equation} \label{equ:contextualinversion_Kij_definition}
K_{i,j}\langle y_1,\, y_2,\, y_3,\, y_4\rangle \overset{\text{def}}{=} I_{y_i+y_j}\langle y_1,\, y_2,\, y_3,\, y_4 \rangle.
\end{equation}
For instance, $K_{1,4}$ maps a dominant seventh chord to the half-diminished seventh chord with the same root, and vice versa, so $K_{1,4}(C^7)=C\textsuperscript{\o{}7}$ and $K_{1,4}(C\textsuperscript{\o{}7})=C^7$. \\
Answer: these other three contextual inversions are already in the subgroup of
$$\text{Sym}(\dom)$$ generated by $fPf^{-1}$, $fLf^{-1}$, and $fRf^{-1}$, because this subgroup is the generalized contextual group of $\langle 0 , 4, 7, 10 \rangle$ by Remark~\ref{rem:PLRconjugation_is_gcg}. In particular, we rely here on the fact that every contextual inversion $K_{i,j}$ on the $TI$-class of a pitch-class segment satisfying the tritone condition is automatically in the associated generalized contextual group by Corollary~4.4 of \cite{fioresatyendra2005}. Since the contextual inversions $K_{1,4}$, $K_{2,4}$, and $K_{3,4}$ are in the subgroup generated by $fPf^{-1}$, $fLf^{-1}$, and $fRf^{-1}$,
{\it each} can be expressed in six ways: as $fPf^{-1}Q_i$, or $fLf^{-1}Q_j$, or $fRf^{-1}Q_k$ (for respective $i$, $j$, and $k$) or as the inverses of these. In fact, for $K_{1,4}$ we have
$$K_{1,4}=fPf^{-1}Q_3 \quad \quad \quad K_{1,4}=fLf^{-1}Q_{11} \quad \quad \quad K_{1,4}=fRf^{-1}Q_6.$$

\end{remark}

\begin{notation} \label{not:bar_notation}
For added clarity, sometimes we use an overline to denote the action of $P$, $L$, or $R$ combined with its $f$-conjugate on the disjoint union
$$\triads \bigsqcup \dom.$$
For instance, $\overline{L}$ means $L$ on consonant triads and $\overline{L}$ means $fLf^{-1}$ on dominant seventh chords and half-diminished chords.
\end{notation}

\section{Construction of a simply transitive group action on a disjoint union: an equivariant bijection produces an internal direct product} \label{sec:construction_for_disjoint_union}


Recall the motivating example from the introduction: we would like to extend the simply transitive actions of the $TI$-group
and $PLR$-group on consonant triads to simply transitive actions on the union of consonant triads with dominant seventh chords and half-diminished chords. Simultaneously, we want those extended groups to extend the duality between the $TI$-group and the $PLR$-group. This extended duality allows a unified network for the omnibus progression in Figure~\ref{fig:omnibus}. Another outcome of this extended duality is new flip-flop cycles of transformations and chords.\footnote{The notion of flip-flop cycle was introduced in \cite{cloughFlipFlop}. Earlier flip-flop cycles from the neo-Riemannian $PLR$-group have been a rich source of examples, for instance the $PL$-cycles and hexatonic cycles of Cohn, see \cite{cohn1996} and \cite{cohn1997}.}

The following theorem constructs an extended duality from an equivariant bijection of simply transitive group actions, crafted for the motivating example. However, we would like to be able to apply the construction to a group action that is not simply transitive and not already part of a duality, so parts \ref{thm:construction_of_action_on_disjoint_union:i:equivariance}, \ref{thm:construction_of_action_on_disjoint_union:ii:internal_direct_product}, and \ref{thm:construction_of_action_on_disjoint_union:iii:short_exact_sequence} are formulated to merely require a faithful group action, while the subsequent parts \ref{thm:construction_of_action_on_disjoint_union:iv:simple_transitivity}, \ref{thm:construction_of_action_on_disjoint_union:v:H}, \ref{thm:construction_of_action_on_disjoint_union:vi:hbar} require simple transitivity. All group actions that we know of in mathematical music theory are faithful. Recall that a group action\footnote{All group actions we consider in this paper are {\it left} group actions, so we usually leave off the adjective {\it left}.} of a group $G$ on a set $S$ is {\it faithful} if for any $g_0 \in G$ such that $g_0s = s$ for all $s \in S$ we can conclude $g_0=e_G$. Faithfulness is equivalent to requiring the homomorphism
\begin{align*}
G \to \text{Sym}(S) \\
g \mapsto \left( s \mapsto gs \right)
\end{align*}
induced by the action  to be injective. Thus, any group that acts faithfully sits inside of the relevant permutation group. Every simply transitively group action is faithful, but faithfulness is more general than simply transitivity.

In Theorem~\ref{thm:construction_of_action_on_disjoint_union}, $f$ can be built from a pitch-class segment extension, a pitch-class segment truncation, or a pitch-class segment modification, as we explain in Example~\ref{examp:generalizedcontextualgroup}, Theorem~\ref{thm:conjugation_with_extension}, Theorem~\ref{thm:modifications}, Example~\ref{examp:strides_and_strains}, and Example~\ref{examp:alteration_diminished_sevenths}.

Later in Section~\ref{sec:construction_for_disjoint_union_multiple} we will extend this theorem to more than two sets.

\begin{theorem}[Construction of an Extension from an Equivariant Bijection of Faithful Group Actions, and Extension of Duality in Simply Transitive Cases]
\label{thm:construction_of_action_on_disjoint_union}
Suppose a group $G$ acts faithfully on two disjoint sets $S_1$ and $S_2$ and suppose $f\colon \; S_1 \to S_2$ is a $G$-equivariant bijection. Let $\overline{S}:=S_1 \bigsqcup  S_2$ and define $\overline{f}\colon \; \overline{S} \to \overline{S}$ by
$$\xymatrix@C=3pc{ S_1 \bigsqcup S_2  \ar[r]_{f \bigsqcup f^{-1}} \ar@/^1pc/[rr]^{\overline{f}} & \ar@{=}[r] S_2 \bigsqcup S_1 & S_1 \bigsqcup S_2 }.$$
Let every $g$ in $G$ act on the disjoint union $\overline{S}$ as $g \bigsqcup  g$, so that $G$ also acts on $\overline{S}$ faithfully and we may consider $G$ as a subgroup of $\text{\rm Sym}(\overline{S})$ via the associated embedding $G \hookrightarrow \text{\rm Sym}(\overline{S})$. Let $\overline{G}:=\langle G, \overline{f} \rangle$, generated inside of $\text{\rm Sym}(\overline{S})$. Then the following statements hold.
\begin{enumerate}
    \item \label{thm:construction_of_action_on_disjoint_union:i:equivariance}
    $\overline{f}$ is a $G$-equivariant involution.
    \item \label{thm:construction_of_action_on_disjoint_union:ii:internal_direct_product}
    $G$ and $\langle\overline{f}\rangle$ commute in $\text{\rm Sym}(\overline{S})$, and $\overline{G}$ is an internal direct product of $G$ and $\langle\overline{f}\rangle$. Moreover, $\overline{G}$ is the disjoint union $$\overline{G}=G \cup \overline{f}G=G \cup G\overline{f}$$ where $G$ is exactly the set of those transformations in $\overline{G}$ that preserve $S_1$ and $S_2$, and the coset $\overline{f}G=G\overline{f}$ is exactly the set of those transformations in $\overline{G}$ that exchange $S_1$ and $S_2$.
    \item \label{thm:construction_of_action_on_disjoint_union:iii:short_exact_sequence}
    The groups $G$, $\overline{G}$, and $\langle \overline{f} \rangle$ fit into the following short exact sequence, where the third homomorphism $\overline{G} \to \langle \overline{f} \rangle$ is $g \overline{f}^i \mapsto \overline{f}^i$.
    $$\{Id_{\overline{S}}\}\longrightarrow G\longrightarrow \overline{G}\longrightarrow \langle \overline{f}\rangle\longrightarrow \{Id_{\overline{S}}\}$$
    This short exact sequence is left split, so is isomorphic to the direct product short exact sequence of $G$ and $\langle \overline{f} \rangle$, which we already know from part \ref{thm:construction_of_action_on_disjoint_union:ii:internal_direct_product}.
    \item  \label{thm:construction_of_action_on_disjoint_union:iv:simple_transitivity}
    If the action of $G$ on $S_1$ is simply transitive, then the action of $G$ on $S_2$ is also simply transitive, and moreover the action of $\overline{G}$ on the disjoint union $\overline{S}$ is also simply transitive.
    \item
    \label{thm:construction_of_action_on_disjoint_union:v:H}
    Suppose the action of $G$ on $S_1$ is simply transitive. Let $H$ be the centralizer of $G$ in $\text{\rm Sym}(S_1)$. Then $H$ is the Lewin dual group of $G$. Let $H$ act on $S_2$ via its embedded copy $fHf^{-1} \leqslant \text{\rm Sym}(S_2)$, so that $f$ is $H$-equivariant. Then parts \ref{thm:construction_of_action_on_disjoint_union:i:equivariance}, \ref{thm:construction_of_action_on_disjoint_union:ii:internal_direct_product}, \ref{thm:construction_of_action_on_disjoint_union:iii:short_exact_sequence}, and \ref{thm:construction_of_action_on_disjoint_union:iv:simple_transitivity} apply to $H$ and $f$, and moreover $\overline{H}:=\langle  H,\overline{f}\rangle$ is the centralizer of $\overline{G}$ in $\text{\rm Sym}(\overline{S})$, so $\overline{H}$ and $\overline{G}$ are dual groups in $\text{\rm Sym}(\overline{S})$.
    \begin{enumerate}
    \item[(i)]
    $\overline{f}$ is $H$-equivariant.
    \item[(ii)]
    $H$ and $\langle \overline{f} \rangle$ commute in $\text{\rm Sym}(\overline{S})$, and $\overline{H}$  is an internal direct product of $H$ and $\langle \overline{f} \rangle$. Moreover, $\overline{H}$ is the disjoint union $$\overline{H}=H \cup \overline{f}H=H \cup H\overline{f}$$ where $H$ is exactly the set of those transformations in $\overline{H}$ that preserve $S_1$ and $S_2$, and $\overline{f}H=H\overline{f}$ is exactly the set of those transformations in $\overline{H}$ that exchange $S_1$ and $S_2$.
    \item[(iii)]
    The groups $H$, $\overline{H}$, and $\langle \overline{f} \rangle$ fit into a short exact sequence similar to part \ref{thm:construction_of_action_on_disjoint_union:iii:short_exact_sequence}.
    \item[(iv)]
    The action of $H$ on $S_2$ is also simply transitive, and the action of $\overline{H}$ on the disjoint union $\overline{S}$ is also simply transitive.
    \end{enumerate}
    \item \label{thm:construction_of_action_on_disjoint_union:vi:hbar}
    Suppose again the action of $G$ on $S_1$ is simply transitive. Let $H$ and its action on $S_2$ be as in part \ref{thm:construction_of_action_on_disjoint_union:v:H}. Let $h \in H$ and let $\overline{h}=h \bigsqcup fhf^{-1}$ be the extension of $h$ to $\overline{S}$. In particular, on $S_2$ this is $fhf^{-1}$. Then the function $\overline{f} \circ \overline{h} = \overline{h} \circ \overline{f} $ is $fh$ on $S_1$ and is $hf^{-1}$ on $S_2$.
\end{enumerate}
\end{theorem}

\begin{proof}
\begin{enumerate}
\item\textbf{$\overline{f}$ is $G$-equivariant and has order 2.}
Since the bijection $f$ is $G$-equivariant, its inverse $f^{-1}$ is also $G$-equivariant. The coproduct of $G$-equivariant maps is $G$-equivariant, so $\overline{f}$ is $G$-equivariant. To see the $G$-equivariance of $\overline{f}$ more concretely, we can write this argument with piecewise defined functions.
\begin{align*}
\overline{f}(gs)
& =\begin{cases} f(gs) & \text{if } s \in S_1 \\ f^{-1}(gs) & \text{if } s \in S_2 \end{cases} \\ \\
& =\begin{cases} gf(s) & \text{if } s \in S_1 \\ gf^{-1}(s) & \text{if } s \in S_2 \end{cases} \\ \\
& =g\overline{f}(s).
\end{align*}

Next we see $\overline{f}$ has order 2.
\begin{align*}
\overline{f} \circ \overline{f} &= \left( f^{-1} \bigsqcup f \right) \circ \left( f \bigsqcup f^{-1} \right) \\
&=  (f^{-1} \circ f) \bigsqcup \,(f \circ f^{-1}) \\
&= Id_{S_1} \bigsqcup Id_{S_2} \\
&= Id_{\overline{S}}
\end{align*}

\item\textbf{$\overline{G}$ is an Internal Direct Product.} The groups $G$ and $\langle \overline{f} \rangle$ commute by \ref{thm:construction_of_action_on_disjoint_union:i:equivariance}. Let $\overline{G} = \langle G, \overline{f} \rangle$. Since $G$ and $\langle \overline{f} \rangle$ commute, we have $\overline{G} = G \cdot \langle \overline{f} \rangle$. Since $G$ and $\langle \overline{f} \rangle$ commute, we also have $G$ and $\langle \overline{f} \rangle$ are normal in $\overline{G}$. Lastly, $G \cap \langle \overline{f} \rangle = Id_{\overline{S}}$ because every element of $G$ preserves $S_1$ and $S_2$ while $\overline{f}$ exchanges $S_1$ and $S_2$ and is an involution, so $\overline{G}$ is the internal direct product of the commuting groups $G$ and $\langle \overline{f} \rangle$. Since $\overline{G}$ is an internal direct product of commuting groups, every element of $\overline{G}$ can be written uniquely as $g \overline{f}^i= \overline{f}^i g$ for some $g \in G$ and some $i=0,1$ (recall $\overline{f}$ is an involution), so $$\overline{G}=G \cup \overline{f}G=G \cup G\overline{f}.$$ Since $G$ preserves $S_1$ and $S_2$, and $\overline{f}$ exchanges $S_1$ and $S_2$, we know each element of $\overline{f}G=G\overline{f}$ exchanges $S_1$ and $S_2$.

\item\textbf{Short Exact Sequence.}
The homomorphism $\overline{G} \to \langle \overline{f} \rangle$ given by the formula $g \overline{f}^i \mapsto \overline{f}^i$ is clearly surjective onto $\langle \overline{f} \rangle$, and the kernel is clearly $G$, so we have the short exact sequence. A left inverse homomorphism for the inclusion $G \rightarrow \overline{G}$ is $g\overline{f}^i \mapsto g$, so the short exact sequence is isomorphic to the product short exact sequence, see Theorem~3.2 of the lecture note \cite{conradkeith_splitting}.

The Second Isomorphism Theorem offers another way to see this short exact sequence, though less concretely. The Second Isomorphism Theorem says: if $A$ and $B$ are subgroups of a group, and if $A$ normalizes $B$, then
\begin{equation} \label{equ:second_isomorphism_theorem}
\left(AB\right)/B \cong A/\left( A \cap B \right).
\end{equation}
The isomorphism in equation \eqref{equ:second_isomorphism_theorem} is induced by the surjective homomorphism
\begin{align*}
AB \to A/(A\cap B) \\
ab \mapsto a (A\cap B),
\end{align*}
see Section 3.3 in \cite{DummitFoote} for an explanation of the Second Isomorphism Theorem with this notation and homomorphism. We take $A=\langle \overline{f} \rangle$ and $B=G$ in equation \eqref{equ:second_isomorphism_theorem}, and obtain
\begin{align*}
\left(\langle\overline{f}\rangle G \right)/G & \cong \langle\overline{f}\rangle/\left(\langle\overline{f}\rangle\cap G\right) \\
\overline{G} /G &\cong \langle \overline{f}\rangle / \{ Id_{\overline{S}} \} \\
\overline{G}/G &\cong \langle\overline{f}\rangle.
\end{align*}

\item\textbf{Simple Transitivity of $\overline{G}$.} Suppose $G$ acts simply transitively on $S_1$. We claim the action of $G$ on $S_2$ is also simply transitive. Let $s_2, s_2' \in S_2$. Then there exist $s_1, s_1' \in S_1$ and unique $g \in G$ such that
$$f(s_1) = s_2, \quad \quad f(s_1') = s_2', \quad \quad \text{and} \quad \quad gs_1 = s_1'.$$
An application of $f$ to the last equality yields $gs_2 = s_2'$. If there is a second element of $G$ that moves $s_2$ to $s_2'$, then we can apply $f^{-1}$ to the equation, and use the simple transitivity on $S_1$ to conclude we have the same $g$. Of course we are making use of the $G$-equivariance of $f$ and $f^{-1}$. Thus, the action of $G$ on $S_2$ is also simply transitive.

We next prove transitivity of $\overline{G}$ on $\overline{S}$ by evaluating all group elements on any one $s_1 \in S_1$, using the simple transitivity of $G$ on $S_1$ and $S_2$.
\begin{align*}
\left(G \bigsqcup G \overline{f}\right)s_1 &= Gs_1 \bigsqcup (G\overline{f})s_1 \\
&= Gs_1 \bigsqcup G (\overline{f}s_1) \\
&= S_1 \bigsqcup S_2 \\
&= \overline{S}
\end{align*}
The orbit of any one $s_1 \in S_1$ is all of $\overline{S}$, so the $\overline{G}$-action on $\overline{S}$ is transitive, without any finiteness assumption.

The uniqueness part of simple transitivity on $\overline{S}$ follows quickly from the transitivity in the finite case from multiple applications of the Orbit-Stabilizer Theorem. Suppose in this paragraph that any one of $G$, $S_1$, or $S_2$ is finite. Then so are the other two, because all three have the same cardinality from the simple transitivity of $G$ on $S_1$ and $S_2$. We have $|\overline{G}|=2\cdot|G|$ from part \ref{thm:construction_of_action_on_disjoint_union:ii:internal_direct_product}, which equals $|S_1|+|S_2|=|\overline{S}|$. Plugging into the Orbit-Stabilizer Theorem for any $s \in \overline{S}$, we have
\begin{align*}
|\overline{G}|/|\overline{G}_s| &= |\text{orbit}(s)| \\
|\overline{G}|/|\overline{G}_s| &= |\overline{S}| \\
|\overline{G}|/|\overline{S}| &=  |\overline{G}_s|\\
1 & = |\overline{G}_s|,
\end{align*}
so that $\overline{G}$ acts simply transitively in the finite case.

If $G$, $S_1$, and $S_2$ are not finite, the uniqueness part of simple transitivity on $\overline{S}$ can still be verified. From the expression of $\overline{G}$ in part \ref{thm:construction_of_action_on_disjoint_union:ii:internal_direct_product} as a disjoint union of preserving and exchanging functions on $\overline{S}$, and from the assumed simple transitivity of $G$ on $S_1$, we know for any two $s_1, s_1' \in S_1$, there is a unique element of $\overline{G}$ that carries $s_1$ to $s_1'$, and that element is in $G$. Similarly, if $s_2, s_2' \in S_2$, there is unique element in $\overline{G}$ that carries $s_2$ to $s_2'$, and that element is in $G$.

Suppose $s_1 \in S_1$ and $s_2 \in S_2$, and suppose there are two elements in $\overline{G}$ that carry $s_1$ to $s_2$. From the decomposition in part \ref{thm:construction_of_action_on_disjoint_union:ii:internal_direct_product} these two group elements necessarily have the form $g\overline{f}$ and $g'\overline{f}$. Then
\begin{align*}
g\overline{f}(s_1) = \; &s_2 = g' \overline{f}(s_1) \\
g\overline{f}(s_1) &= g' \overline{f}(s_1) \\
\overline{f}(s_1) &= g^{-1} g' \left( \overline{f}(s_1) \right)
\end{align*}
so $g^{-1}g' = Id_{\overline{S}}$, and $g\overline{f}=g' \overline{f}$. A similar argument shows the uniqueness of a group element in $\overline{G}$ that carries $s_2$ to $s_1$.

\item\textbf{Duality of $\overline{H}$ and $\overline{G}$.}
Suppose $G$ acts simply transitively on $S_1$.
Let $H$ be the centralizer of the simply transitive group $G$ in $\text{\rm Sym}(S_1)$. Then $H$ also acts simply transitively, so $H$ is the dual group of $G$, by Proposition~3.2 of \cite{BerryFiore} (the claim that a centralizer of a simply acting group is its Lewin dual was stated on page 253 of \cite{LewinGMIT} but not proved there).

Let $H$ act on $S_2$ via its embedded copy $fHf^{-1} \leqslant \text{\rm Sym}(S_2)$. Then $f$ is $H$-equivariant because
\begin{equation} \label{equ:f_is_H_equivariant}
\begin{aligned}
   f(hs_1) &= f\left( h(f^{-1}f)s_1 \right) \\
   &= \left( f h f^{-1} \right)f(s_1) \\
   &\overset{\text{def}}{=} hf(s_1). \\
\end{aligned}
\end{equation}
The group $H$ and bijection $f$ now satisfy the hypotheses on $G$ and $f$ in the present theorem, so we can apply parts \ref{thm:construction_of_action_on_disjoint_union:i:equivariance}, \ref{thm:construction_of_action_on_disjoint_union:ii:internal_direct_product}, \ref{thm:construction_of_action_on_disjoint_union:iii:short_exact_sequence}, and \ref{thm:construction_of_action_on_disjoint_union:iv:simple_transitivity} to $H$ and $f$. In particular, the group $\overline{H}:=\langle  H,\overline{f}\rangle$ in $\text{\rm Sym}(\overline{S})$ acts simply transitively on $\overline{S}$.

The groups $\overline{G}$ and $\overline{H}$ commute because
$$(g\overline{f}^i)(h\overline{f}^j) = (h\overline{f}^j)(g\overline{f}^i),$$
as all four of these elements commute and can be rearranged at will. By Proposition~3.1 of \cite{BerryFiore}, we conclude from the commutativity of $\overline{G}$ and $\overline{H}$ that $\overline{H}$ is actually the centralizer of $\overline{G}$, and $\overline{H}$ and $\overline{G}$ are dual groups in $\text{\rm Sym}(\overline{S})$.

\item\textbf{formulas for the Function $\overline{f} \,\overline{h}$.}
Suppose the action of $G$ on $S_1$ is simply transitive. Let $H$ and its action on $S_2$ be as in part \ref{thm:construction_of_action_on_disjoint_union:v:H}.
On $S_1$, the composition $\overline{f} \,\overline{h}$ is clearly $fh$.

Let $s_2 \in S_2$. Then
$$\overline{f} \,\overline{h} (s_2) = f^{-1} \left( fhf^{-1} \right) (s_2) = hf^{-1}(s_2).$$
\end{enumerate}
\end{proof}

We can now apply Theorem~\ref{thm:construction_of_action_on_disjoint_union} to the 24 major and minor triads and the 24 dominant and half-diminished seventh chords to treat the omnibus progression and its transformational network, both pictured in Figure~\ref{fig:omnibus}. We do the algebra in Example~\ref{examp:triadsandsevenths} (building on Section~\ref{sec:notations}), and we do the omnibus progression in Example~\ref{examp:omnibus}.

\begin{example}[Simply Transitive Group Action on Major, Minor, Dominant, and Half-Diminished Chords] \label{examp:triadsandsevenths}
We now use the ingredients from Section~\ref{sec:notations} in Theorem~\ref{thm:construction_of_action_on_disjoint_union}. We take $G$ to be the $TI$-group, we take $S_1$ and $S_2$ to be the following $G$-orbits with 24 elements each,
\begin{align*}
S_1 = \triads & = TI\langle 0,4,7 \rangle \\
S_2 = \dom & = TI\langle 0, 4, 7, 10 \rangle,
\end{align*}
and we take the $G$-equivariant bijection\footnote{In Table~\ref{table:seventhchord_functions} in Section~\ref{sec:bijections_of_triads_and_seventh_chords} we will introduce bijections with other seventh chords as well. The bijection $f$ from equations \eqref{equ:bijection_triads_domhalfdiminished} and \eqref{equ:2y-x} is denoted $f_{Dom^7Tr}$ in Table~\ref{table:seventhchord_functions}. These bijections will be used to build a larger group that acts simply transitively on several types of seventh chords in Section~\ref{sec:construction_for_disjoint_union_multiple}.} to be
\begin{align*}
f \colon \; \triads &\to \dom \\
f\langle w, x, y \rangle &:= \langle w,\; x,\; y,\; 2y-x \rangle.
\end{align*}
from Figure~\ref{fig:I7_commutes_with_f} in the Introduction and equations \eqref{equ:bijection_triads_domhalfdiminished} and \eqref{equ:2y-x}. In Theorem~\ref{thm:construction_of_action_on_disjoint_union}, $H$ is now necessarily the neo-Riemannian $PLR$-group on consonant triads in dualistic root position (this particular instantiation of the $PLR$-group was reviewed in \cite{cransfioresatyendra}, and is briefly recalled in Section~\ref{sec:notations}). The $PLR$-group acts on the dominant and half-diminished seventh chords via conjugation with $f$ from equations \eqref{equ:bijection_triads_domhalfdiminished} and \eqref{equ:2y-x}. With $\overline{f}\overset{\text{def}}{=}f \sqcup f^{-1}$,
Theorem~\ref{thm:construction_of_action_on_disjoint_union} now guarantees that the groups
$$\overline{G}=\langle TI\text{-group},\; \overline{f} \rangle \quad \quad \quad \text{and} \quad \quad \quad \overline{H}=\langle PLR\text{-group},\; \overline{f} \rangle$$
act simply transitively on
$$\overline{S} = \triads \bigsqcup \dom$$
and are dual to one another in $\text{Sym}(\overline{S})$, and moreover $\overline{f}$ commutes with all elements of both groups.

The transformation $\overline{f} \circ \overline{L} = \overline{L} \circ \overline{f}$ will be essential for our transformational re-interpretation of the omnibus progression. Recall $\overline{L}$ from Notation~\ref{not:bar_notation}.
Part \ref{thm:construction_of_action_on_disjoint_union:vi:hbar} of Theorem~\ref{thm:construction_of_action_on_disjoint_union} allows us to rewrite $\overline{f} \circ \overline{L} = \overline{L} \circ \overline{f}$ as $f L$ on consonant triads and as $Lf^{-1}$ on dominant and half-diminished seventh chords. In Figure~\ref{fig:fbar_Lbar_squares} we compute the transformation $\overline{f} \circ \overline{L} = \overline{L} \circ \overline{f}$ on both $a$ and $A$ using both composites, and see they are the same as $fL$ on $a$ and $A$ respectively (the upper right paths in each diagram).

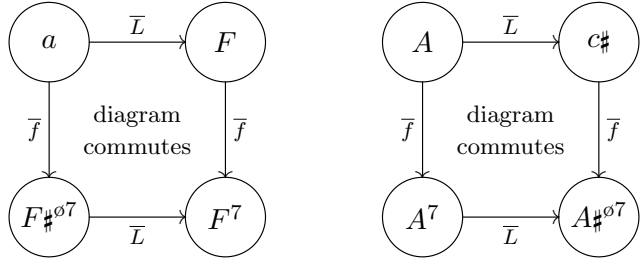
\begin{figure}[H]
    $$\entrymodifiers={=<2.5pc>[o][F-]}
\xymatrix@C=3pc@R=3pc{
a \ar[r]^{\overline{L}} \ar[d]_{\overline{f}} \ar@{}[dr]|{\txt{\footnotesize diagram \\ \footnotesize commutes}} & F \ar[d]^{\overline{f}} \\
F\sh\textsuperscript{\o{}7} \ar[r]_{\overline{L}} & F^7
}
\quad \quad \quad \quad
\entrymodifiers={=<2.5pc>[o][F-]}
\xymatrix@C=3pc@R=3pc{
A \ar[r]^{\overline{L}} \ar[d]_{\overline{f}} \ar@{}[dr]|{\txt{\footnotesize diagram \\ \footnotesize commutes}} & c\sh \ar[d]^{\overline{f}} \\
A^7 \ar[r]_{\overline{L}} & A \sh \textsuperscript{\o{}7}
}$$
\caption{We compute the transformation $\overline{f} \circ \overline{L} = \overline{L} \circ \overline{f}$ on both $a$ and $A$ using both composites.  Here
$\overline{f} \overset{\text{def}}{=} f \sqcup f^{-1}$ where $f$ is as in equations \eqref{equ:bijection_triads_domhalfdiminished} and \eqref{equ:2y-x}, and $\overline{L}$ means $L$ on consonant triads and $\overline{L}$ means $fLf^{-1}$ on dominant seventh and half-diminished seventh chords as in Notation~\ref{not:bar_notation}. These two composites $\overline{f} \circ \overline{L} = \overline{L} \circ \overline{f}$ on $a$ and $A$ are the first step in the omnibus progression of Figure~\ref{fig:omnibus_network} and the inverted omnibus progression of Figure~\ref{fig:omnibus_network_inverted}. The intermediate chords of these two squares are {\it not} in the progression; only $a$ and $F^7$ are in the omnibus progression of Figure~\ref{fig:omnibus_network}, and $A$ and $A\sh\textsuperscript{\o{}7}$ are in the inverted omnibus progression of Figure~\ref{fig:omnibus_network_inverted}.} \label{fig:fbar_Lbar_squares}
\end{figure}
\end{example}

We can proceed to discuss the omnibus progression from Figure~\ref{fig:omnibus} using our construction from Theorem~\ref{thm:construction_of_action_on_disjoint_union} in the special case of consonant triads combined with dominant seventh chords and half-diminished seventh chords. The omnibus progression moves between minor triads and dominant seventh chords, while the inverted omnibus progression moves between major triads and half-diminished seventh chords. Both omnibus progressions incorporate parsimonious voice leading.

\begin{example}[Omnibus Network] \label{examp:omnibus}
We continue the notations $\overline{f}$, $\overline{G}$, $\overline{H}$, and $\overline{S}$ from the preceding Example~\ref{examp:triadsandsevenths}.
The length 7 omnibus progression in Figure~\ref{fig:omnibus} of the Introduction consists of two transformationally palindromic  progressions of length 4 that are structurally similar, overlap in one chord, and have another chord in common. The pattern can be repeated to make a length 12 sequence consisting of 4 overlapping length 4 sequences, ultimately ending with the same initial chord $a$. This is done in Figure~\ref{fig:omnibus_network}, where the first two rows make up the length 7 progression in Figure~\ref{fig:omnibus}. The first two rows overlap in $f\sh$ and both have $D^7$. The structural similarity of all the rows in Figure~\ref{fig:omnibus_network} is visible in the same  horizontal $\overline{H}$ transformations of each row. The duality of Theorem~\ref{thm:construction_of_action_on_disjoint_union}~\ref{thm:construction_of_action_on_disjoint_union:v:H} guarantees that the element $T_{-3}$ of $\overline{G}$ on the vertical arrows commutes with all the horizontal $\overline{H}$ arrows. We can also start the entire network on $A$ instead of $a$ to obtain a kind of inverted omnibus in Figure~\ref{fig:omnibus_network_inverted}. The inverted omnibus appeared on page 114 of \cite{ziehn} in the section called ``Canons in the small Third and large Sixth.''

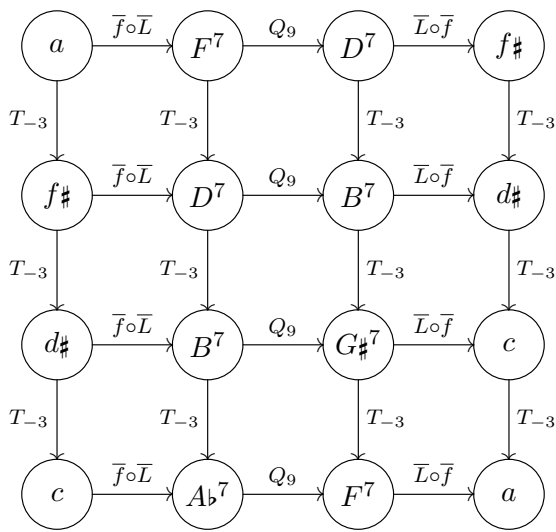
\begin{figure}[H]
$$\entrymodifiers={=<2.2pc>[o][F-]}
\xymatrix@C=2.5pc@R=2.5pc
{a   \ar[r]^{\overline{f}\circ {\overline{L}}} \ar[d]_{T_{-3}} & F^7 \ar[r]^{Q_9} \ar[d]_{T_{-3}} &  D^7 \ar[r]^{ {\overline{L}}\circ \overline{f}} \ar[d]^{T_{-3}} &  f\sh \ar[d]^{T_{-3}}  \\
 f\sh \ar[r]^{\overline{f}\circ  {\overline{L}}} \ar[d]_{T_{-3}} & D^7 \ar[r]^{Q_9} \ar[d]_{T_{-3}}  & B^7 \ar[r]^{ {\overline{L}}\circ \overline{f}} \ar[d]^{T_{-3}} &  d\sh \ar[d]^{T_{-3}} \\
 d\sh \ar[r]^{\overline{f}\circ {\overline{L}}} \ar[d]_{T_{-3}} & B^7 \ar[r]^{Q_9} \ar[d]_{T_{-3}}  & G \sh ^7 \ar[r]^{ {\overline{L}}\circ \overline{f}} \ar[d]^{T_{-3}} & c\ar[d]^{T_{-3}} \\
 c    \ar[r]^{\overline{f}\circ  {\overline{L}}}                 & A \fl^7 \ar[r]^{Q_9}             &  F^7 \ar[r]^{ {\overline{L}}\circ \overline{f}}  & a
}$$
\caption{The first two rows of this network are the two short omnibus progression excerpts in Figure~\ref{fig:omnibus}, ignoring voicings. Notice the last chord $f \sh$ of the first line is the first chord of the second line. This network continues that enchained pattern, and shows successive short omnibus progression excerpts as rows. Each row is a minor third below its preceding row; the duality in Theorem~\ref{thm:construction_of_action_on_disjoint_union}~\ref{thm:construction_of_action_on_disjoint_union:v:H} applied in Example~\ref{examp:triadsandsevenths} guarantees the commutativity of the vertical transformations with the horizontal transformations. Overall, this long omnibus has 12 chords, and begins and ends on $a$. The short omnibus progression actually has 5 chords. The enchaining of two successive 5-chord short omnibus progressions overlaps in two chords. Each row here only shows four of the five chords in order to be consistent with Figure~\ref{fig:omnibus}; the enchaining of two four-chord sequences overlaps in just one chord. This is why we speak of short omnibus progression {\it excerpts} (only 4 of the 5 chords).} \label{fig:omnibus_network}
\end{figure}
\begin{figure}
$$\entrymodifiers={=<2pc>[o][F-]}
\xymatrix@C=2.5pc@R=2.5pc
{A   \ar[r]^{\overline{f}\circ  {\overline{L}}} \ar[d]_{T_{-3}} & A\sh\textsuperscript{\o{}7} \ar[r]^{Q_9} \ar[d]_{T_{-3}} &  G\textsuperscript{\o{}7} \ar[r]^{ {\overline{L}}\circ \overline{f}} \ar[d]^{T_{-3}} &  F\sh \ar[d]^{T_{-3}}  \\
 F\sh \ar[r]^{\overline{f}\circ  {\overline{L}}} \ar[d]_{T_{-3}} & G\textsuperscript{\o{}7} \ar[r]^{Q_9} \ar[d]_{T_{-3}}  & E\textsuperscript{\o{}7} \ar[r]^{ {\overline{L}}\circ \overline{f}} \ar[d]^{T_{-3}} &  D\sh \ar[d]^{T_{-3}} \\
 D\sh \ar[r]^{\overline{f}\circ  {\overline{L}}} \ar[d]_{T_{-3}} & E\textsuperscript{\o{}7} \ar[r]^{Q_9} \ar[d]_{T_{-3}}  & C\sh\textsuperscript{\o{}7} \ar[r]^{ {\overline{L}}\circ \overline{f}} \ar[d]^{T_{-3}} & C\ar[d]^{T_{-3}} \\
 C    \ar[r]^{\overline{f}\circ  {\overline{L}}}                 & C\sh\textsuperscript{\o{}7} \ar[r]^{Q_9}             &  A\sh\textsuperscript{\o{}7} \ar[r]^{ {\overline{L}}\circ \overline{f}}  & A
}$$
\caption{If we begin the omnibus network in Figure~\ref{fig:omnibus_network} with $A$ instead of $a$, then we obtain an inverted omnibus progression involving major triads and half-diminished seventh chords instead of minor triads and dominant seventh chords (the {\it Schritt} $Q_9$ is important for the inverted omnibus to work). The inverted omnibus appeared on page 114 of \cite{ziehn}. } \label{fig:omnibus_network_inverted}
\end{figure}
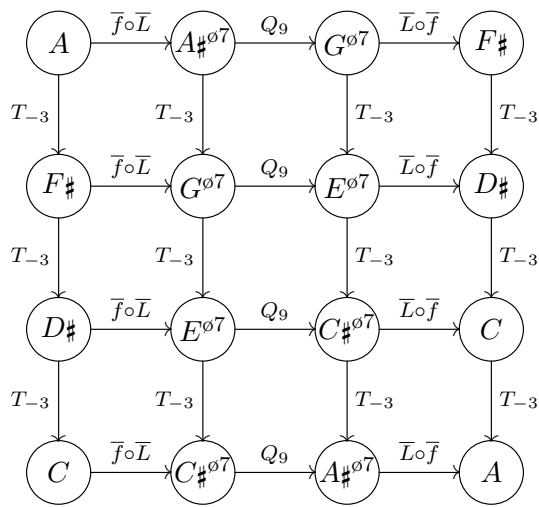
\end{example}

\begin{example}[Omnibus with Voicing Transformation Matrices]
The voice leading in Figure~\ref{fig:omnibus} is exactly traced via the following voicing transformation matrices:
$$U_{sopr} = \left(
\begin{array}{cccc}
8 & 7 & 3 & 7\\
0 & 1 & 0 & 0\\
0 & 0 & 1 & 0\\
5 & 5 & 9 & 6
\end{array}
\right), \,
U_{alto} = \left(
\begin{array}{cccc}
{11} & 1 & {10} & 3\\
0 & 1 & 0 & 0\\
2 & 7 & 3 & 1\\
0 & 0 & 0 & 1
\end{array}
\right), \,
U_{tenor} = \left(
\begin{array}{cccc}
2 & 1 & 7 & 3\\
{11} & 0 & 5 & 9\\
0 & 0 & 1 & 0\\
0 & 0 & 0 & 1
\end{array}
\right) \in SL(4, \mathbb{Z}_{12})
$$
In order to realize the voice leading in Figure~\ref{fig:omnibus}, one starts with the voicing $(4, 9, 0, 4) \in \mathbb{Z}_{12}^{\times 4}$ and applies (here reading from left to right) one after the other the transformations
$$U_{sopr},\, U_{sopr},\, U_{alto},\, U_{alto},\, U_{alto},\, U_{tenor},\, U_{tenor},\, U_{tenor},\, U_{sopr},\, U_{sopr},\, U_{sopr},\, U_{alto}, \dots $$
All three matrices $U_{sopr}$, $U_{alto}$, and $U_{tenor}$ are of order $12$ in $SL(4,\mathbb{Z}_{12})$.
There are precisely 36 elements in $SL(4,\mathbb{Z}_{12})$ of the form $$\left(
\begin{array}{cccc}
{a_1} & {a_2} & {a_3} & {a_4}\\
0 & 1 & 0 & 0\\
0 & 0 & 1 & 0\\
{d_1} & {d_2} & {d_3} & {d_4}\\
\end{array}
\right),$$ which realize the voice-leading with moving bass and soprano.
\end{example}

\begin{example}[``I'm Coming Virginia'' and ``Stella by Starlight'']
We have two more examples of Theorem~\ref{thm:construction_of_action_on_disjoint_union} and the group in Example~\ref{examp:triadsandsevenths} with ``I'm Coming Virginia'' by Donald Heywood and ``Stella by Starlight'' by Victor Young. They are both flip-flop cycles of $K_{3,4}$ and $Q_7K_{1,4}$ with a consonant triad connected at the end via $Q_5\overline{f}$, see Figure~\ref{fig:Virginia_and_Stella}. The $T_{-7}$ transposition between the two examples should not be overemphasized, as keys are commonly changed in Jazz and we could simply use a different transposition for different keys. The two passages here are in their ``original'' keys - as in the {\it Real Book} or the {\it Anthologie les Grilles du Jazz}. The big band arrangement of ``I’m coming Virginia'' by Artie Shaw starts in the ``historic'' key of $F$, but then it modulates to $A\fl$. The original film score of ``Stella by Starlight'' is in $E\fl$, but jazz musicians now prefer $B\fl$.
\begin{figure}[H]
\begin{center}
\includegraphics[scale = 0.375]{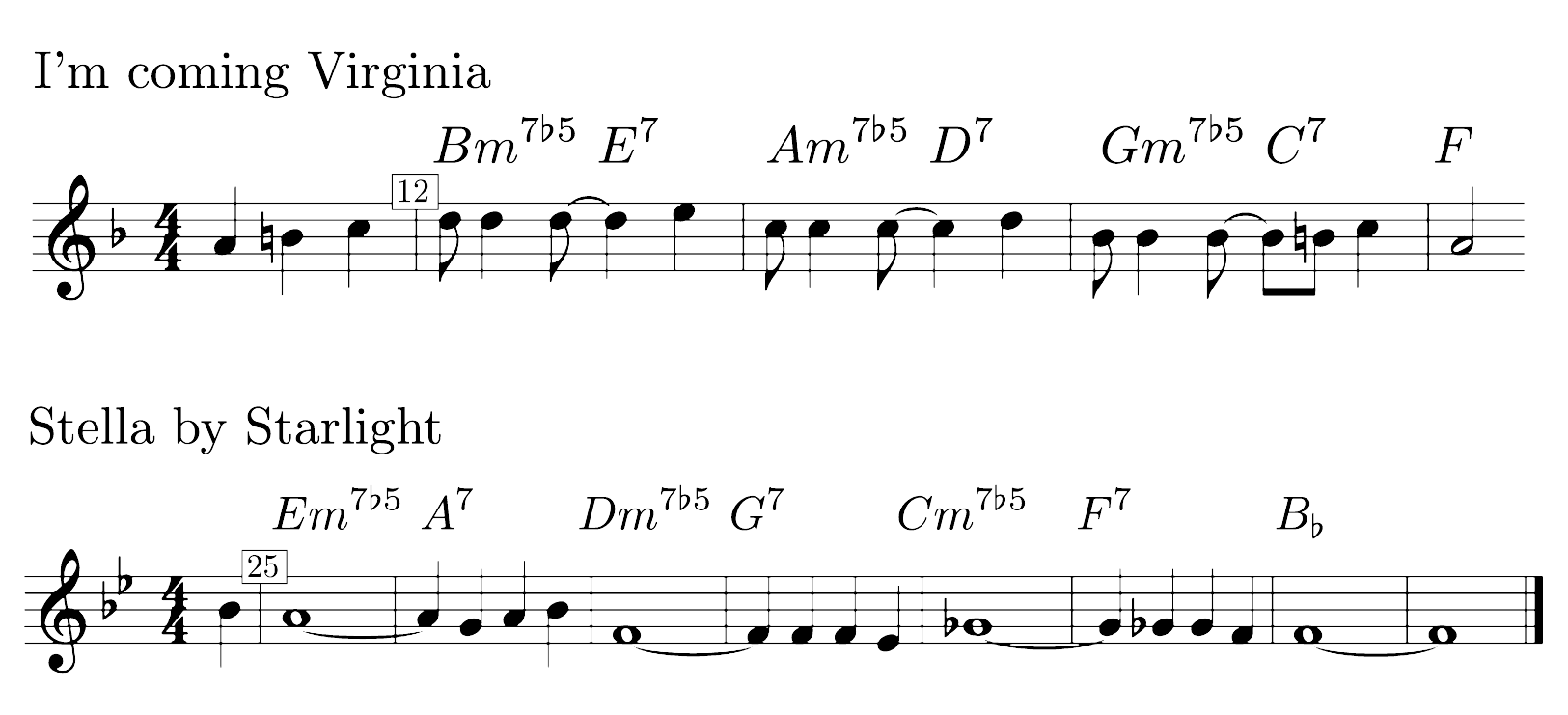}
$$\entrymodifiers={=<2.2pc>[o][F-]}
\xymatrix@C=2.5pc@R=2.5pc{B^{\o{}7}   \ar[r]^{K_{3,4}} \ar[d]_{T_{-7}} & E^{7} \ar[r]^{Q_7K_{1,4}} \ar[d]_{T_{-7}} &  A^{\o{} 7}
\ar[r]^{K_{3,4}} \ar[d]^{T_{-7}} &  D^7 \ar[d]^{T_{-7}} \ar[r]^{Q_7K_{1,4}} & G^{\o{}7} \ar[r]^{K_{3,4}} \ar[d]^{T_{-7}} & C^7 \ar[d]^{T_{-7}} \ar[r]^{Q_5\overline{f}} & F \ar[d]^{T_{-7}} \\
E^{\o{}7} \ar[r]^{K_{3,4}} & A^{7} \ar[r]^{Q_7K_{1,4}} & D^{\o{} 7} \ar[r]^{K_{3,4}} &  G^7 \ar[r]^{Q_7K_{1,4}} & C^{\o{}7} \ar[r]^{K_{3,4}} & F^7 \ar[r]^{Q_5\overline{f}} &B\fl} $$
$$\entrymodifiers={+<2mm>[F-]}
\xymatrix@C=2.25pc@R=2.5pc{\langle 9,5,2,11 \rangle   \ar[r]^{K_{3,4}} \ar[d]_{T_{-7}} & \langle 4,8,11,2 \rangle \ar[r]^{\;Q_7K_{1,4}} \ar[d]_{T_{-7}} &  \langle 7,3,0,9 \rangle
\ar[r]^{K_{3,4}} \ar[d]^{T_{-7}} &  \langle 2,6,9,0 \rangle\ar[d]^{T_{-7}}  \ar[r]^{Q_7K_{1,4}} &
\langle 5,1,10,7 \rangle \ar[r]^{K_{3,4}} \ar[d]^{T_{-7}} & \langle 0,4,7,10 \rangle \ar[d]^{T_{-7}} \\
\langle 2,10,7,4 \rangle \ar[r]^{K_{3,4}} & \langle 9,1,4,7 \rangle \ar[r]^{Q_7K_{1,4}} & \langle 0,8,5,2 \rangle \ar[r]^{K_{3,4}} &  \langle 7,11,2,5 \rangle \ar[r]^{Q_7K_{1,4}} & \langle 10,6,3,0 \rangle \ar[r]^{K_{3,4}} & \langle 5,9,0,3 \rangle } $$
\end{center}
\caption{``I'm Coming Virginia'' by Donald Heywood with lyrics by Will Marion Cook, and ``Stella by Starlight'' by Victor Young. Chords denoted with minor 7 flat 5 are the same as half-diminished seventh chords, so $Bm^{7\fl 5}$ is the same as $B^{\o{} 7}$, etc. Both songs are flip-flop cycles of $K_{3,4}$ and $Q_7K_{1,4}$ with a consonant triad connected at the end via $Q_5\overline{f}$. That last triad is left off of the bottom pc-seg network due to space limitations.} \label{fig:Virginia_and_Stella}
\end{figure}
\end{example}

\begin{example}[`` `Round Midnight '' by Thelonious Monk and Cootie Williams]
Another example of Theorem~\ref{thm:construction_of_action_on_disjoint_union} and the group in Example~\ref{examp:triadsandsevenths} is the flip-flop cycle of $K_{3,4}$ and $Q_7K_{1,4}$ in the introduction to `` `Round Midnight '' by Thelonious Monk and Cootie Williams, although the dominant seventh chords here are not literally dominant seventh chords because the fifth of each dominant seventh chord is flattened. See Figure~\ref{fig:RoundMidnight}.
\begin{figure}
\begin{center}
\includegraphics[scale = 0.6]{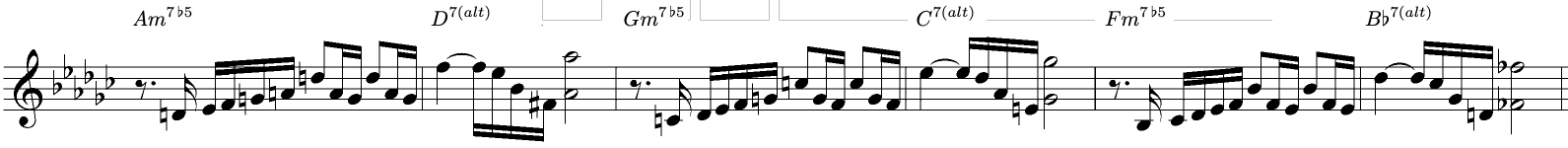}
$$\entrymodifiers={=<2.2pc>[o][F-]}
\xymatrix@C=5pc@R=2.5pc{A^{\o{} 7} \ar[r]^{K_{3,4}} & D^7 \ar[r]^{Q_7K_{1,4}} & G^{\o{}7} \ar[r]^{K_{3,4}} & C^7 \ar[r]^{Q_7K_{1,4}} & F^{\o{}7} \ar[r] & B\flat ^7 }$$
$$\entrymodifiers={+<2mm>[F-]}
\xymatrix@C=2.5pc@R=2.5pc{\langle 7,3,0,9 \rangle \ar[r]^{K_{3,4}} & \langle 2,6,9,0 \rangle \ar[r]^{Q_7K_{1,4}\;} &
\langle 5,1,10,7 \rangle \ar[r]^{K_{3,4}} & \langle 0,4,7,10 \rangle \ar[r]^{Q_7K_{1,4}} & \langle 3,11,8,5 \rangle \ar[r]^{K_{3,4}} & \langle 10,2,5,8 \rangle }$$
\end{center}
\caption{Introduction to `` `Round Midnight '' by Thelonious Monk and Cootie Williams with lyrics by Bernie Hanighen. This is a flip-flop cycle of $K_{3,4}$ and $Q_7K_{1,4}$, though the dominant seventh chords are not literally dominant 7th chords because they have a flattened fifth $\flat \hat{5}$. The networks drawn show the {\it unaltered} dominant 7th chords.
On the downbeats with the dominant chords, the melody has a $\sharp \hat{9}$ in tension with the simultaneous $\hat{3}$ in the harmony, and only later in the measure does the $\hat{3}$ appear in the melody. For instance in the second measure with $D^{7(alt)}$, the melody starts with $E\sh =F$, and later in the measure we see the $F\sh$ in the melody. The $A\fl$ on beat 3 in that measure is a $\flat \hat{5}$ (the altered chord tone), the fourth and sixth measures are analogous. The melody in the second measure is based on the \emph{altered scale} $D$ - $E\fl$ - $E\sh$ - $F\sh$ - $G\sh \slash A\fl$ - $B\fl$ - $C$ - $D$, which is a split-and-fuse-alteration of the $D$-mixolydian $D$ - $E$ - $F\sh$ - $G$ - $A$ - $B$ - $C$ - $D$ in which $E$ splits into $E\flat$ and $E\sh$ and $G$ and $A$ merge to $G\sh \slash A\fl$. Here $E\fl$ and $E\sh$ are considered as representatives of $\hat{2}$ so that $F\sh$ is still $\hat{3}$. The $G\sh \slash A\fl$ can be interpreted as $\hat{4}$ or $\hat{5}$, we consider it as $\hat{5}$ in order to have the flattened fifth $A\fl$ of our seventh chord $D^7$. The $D^7$ is the $V$ chord of $G$-major, so it is natural to consider it inside the fifth mode of $G$-major, which is $D$-mixolydian.}
\label{fig:RoundMidnight}
\end{figure}
\end{example}

Next we develop several other options for $f$ in Theorem~\ref{thm:construction_of_action_on_disjoint_union} to greatly broaden its applicability. These other options for $f$ arise from pitch-class segment extensions, pitch-class segment truncations, and pitch-class segment modifications.

\begin{example}[Pitch-Class Segment Extensions and Dualities between $TI$-Groups and Generalized Contextual Groups] \label{examp:generalizedcontextualgroup}
In Example~\ref{examp:triadsandsevenths} we saw how the extension of the pitch-class segment $\langle 0,4,7 \rangle$ to the pitch-class segment  $\langle 0, 4, 7, 10 \rangle$ together with the known $TI$-$PLR$ duality generated an entire example of Theorem~\ref{thm:construction_of_action_on_disjoint_union}. We can in fact do this for the extension of any pitch-class segment $X_1$  to a longer pitch-class segment $X_2$ using \cite{fioresatyendra2005} (the initial pitch-class segment $X_1$ must contain an interval besides a unison and a tritone). We will prove in Theorem~\ref{thm:conjugation_with_extension} that conjugation with pc-segment extension commutes with contextual inversion and $Q_i$, so the generalized contextual group of the longer pitch-class segment $X_2$ is equal to the conjugation of the generalized contextual group of the shorter pitch-class segment $X_1$, where we conjugate by the map of orbits $$TIX_1 \to TIX_2$$
in part \ref{thm:construction_of_action_on_disjoint_union:v:H} of Theorem~\ref{thm:construction_of_action_on_disjoint_union}.

Let $X_1=\langle x_1, \dots, x_{n_1}\rangle$ be a selected, fixed pitch-class segment that contains two distinct pitch-classes $x_q$ and $x_r$ which span an interval other than a tritone. Then the $TI$-group acts simply transitively on the orbit
$$S_1 := TI\langle x_1, \dots, x_{n_1}\rangle.$$
By Sections~3 and 4 of \cite{fioresatyendra2005}, its dual group in $\text{Sym}(S_1)$ is the {\it generalized contextual group}
\begin{equation} \label{equ:qk}
\langle Q_1,\; K_{1,2} \rangle.
\end{equation}
The generators are defined on a pitch-class segment $Y \in S_1$ as
\begin{equation} \label{equ:Qi}
Q_1(Y) :=
\begin{cases}
T_1(Y) & \text{if $Y$ is a $T$-form of $X_1$} \\
T_{-1}(Y) & \text{if $Y$ is an $I$-form of $X_1$}
\end{cases}
\end{equation}
and
\begin{equation} \label{equ:K}
K_{1,2}(Y) := I_{y_1+y_2}(Y).
\end{equation}
By Corollary~4.4 of \cite{fioresatyendra2005}, $K_{1,2}$ can be replaced by any $K_{i,j}$ and the resulting group $\langle Q_1, K_{i,j} \rangle$ is equal to the generalized contextual group in \eqref{equ:qk}, even though a generator is different.

Now we can consider a pitch-class segment $X_2$ of length $n_2$ which extends $X_1$:
$$X_2 = \langle x_1, \dots, x_{n_1}, x_{n_1+1}, \dots, x_{n_2}\rangle.$$
We repeat the aforementioned process for $X_2$ and we obtain
$$S_2 := TI\langle x_1, \dots, x_{n_1}, x_{n_1+1}, \dots, x_{n_2}\rangle$$
with simple transitive $TI$-group action and simply transitive contextual group action (generators have same formulas, but act on the extended pitch-class segments). The one-element assignment $X_1 \mapsto X_2$ then expands to a $TI$-equivariant bijection $f_\text{ext}$ via the $TI$-action. Read the statement of Theorem~\ref{thm:conjugation_with_extension}.

Theorem~\ref{thm:construction_of_action_on_disjoint_union} now applies and obtain dual groups on $\overline{S}=S_1 \bigsqcup S_2$. This encompasses Example~\ref{examp:triadsandsevenths} via the extension $\langle 0,4,7 \rangle \mapsto \langle 0,4,7,10 \rangle $ and any other extension of a pitch-class segment that initially contains two distinct pitch classes that span an interval other than a tritone. For instance, Theorem~\ref{thm:construction_of_action_on_disjoint_union} now applies to the other two extensions of $\langle 0,4,7 \rangle$ in Table~\ref{table:seventhchord_functions}: the extension to a major seventh chord with same root arising from the single-element assignment
$$\langle 0,4,7 \rangle \mapsto \langle 0,4,7,11 \rangle,$$
 and the extension to a minor seventh with root a minor third lower arising from the single-element assignment
$$\langle 0,4,7 \rangle \mapsto \langle 0,4,7,9 \rangle.$$
The last function in Table~\ref{table:seventhchord_functions}, from consonant triads to diminished sevenths, is {\it not} an example of Theorem~\ref{thm:conjugation_with_extension} because it is an alteration combined with an extension, namely $\langle 0,4,7 \rangle \mapsto \langle 1, 4, 7, 10\rangle$ as discussed in Example~\ref{examp:alteration_diminished_sevenths}.

Another interesting example for Theorem~\ref{thm:conjugation_with_extension} in combination with Theorem~\ref{thm:construction_of_action_on_disjoint_union} is the tetractys including into the pentatonic as generated scales, or in other words extending the tetractys to the pentatonic:
$$\langle F,C,G \rangle \mapsto \langle F,C,G,D,A \rangle.$$
We obtain a 48-element group acting on the disjoint union of tetractys pitch-class segments and pentatonic pitch-class segments. The pentatonic further includes into the diatonic, which further includes into the chromatic, so we can further extend the group to have a meta-rotation example of Theorem~\ref{thm:generalized} in Example~\ref{examp:generated_scales}.
\end{example}

\begin{theorem}[Extension and Truncation of a Pitch-Class Segment and Conjugation of Contextual Groups] \label{thm:conjugation_with_extension}
Let $X_1=\langle x_1, \dots, x_{n_1} \rangle$ be a selected, fixed pitch-class segment that contains two distinct pitch-classes $x_q$ and $x_r$ which span an interval other than a tritone. Let $X_2$ be a selected, fixed pitch-class segment of length $n_2$ which extends $X_1$:
$$X_2 = \langle x_1, \dots, x_{n_1}, x_{n_1+1}, \dots, x_{n_2}\rangle.$$
Consider the $TI$-orbits $S_1 := TIX_1$ and $S_2 := TIX_2$, and let $f_\text{\rm ext}$ be the unique $TI$-equivariant map
$$f_\text{\rm ext}\colon \; TIX_1 \to TIX_2$$
that extends the extension $X_1 \mapsto X_2$. Then for $1 \leq i,j \leq n_1$, the $f_\text{\rm ext}$-conjugation of $K_{i,j}$ on $TIX_1$ is $K_{i,j}$ on $TIX_2$. This means for $Y \in TIX_2$,
\begin{equation} \label{equ:contextual_conjugation}
\begin{aligned}
f_\text{\rm ext}K_{i,j}f^{-1}_\text{\rm ext} (Y) &= I_{y_i+y_j}(Y) \\
& \overset{\text{\rm def}}{=} K_{i,j}(Y).
\end{aligned}
\end{equation}
Moreover, the $f_{\text{\rm ext}}$-conjugation of $Q_i$ on $TIX_1$ is $Q_i$ on $TIX_2$:
\begin{equation} \label{equ:Qi_conjugation}
f_\text{\rm ext}Q_if^{-1}_\text{\rm ext}=Q_i,
\end{equation}
and the $f_\text{\rm ext}$-conjugation of the generalized contextual group associated to $X_1$ is the generalized contextual group associated $X_2$. Consequently, either one of these generalized contextual groups is generated by contextual inversions if and only if the other is.

Let $f_\text{\rm trunc}$ be the unique $TI$-equivariant map
$$f_\text{\rm trunc }\colon \; TIX_2 \to TIX_1$$
that extends the truncation $X_2 \mapsto X_1$. Then for $1 \leq i,j \leq n_1$, the $f_\text{\rm trunc}$-conjugation of $K_{i,j}$ on $TIX_2$ is $K_{i,j}$ on $TIX_1$. This means for $Y \in TIX_1$,
\begin{equation} \label{equ:contextual_conjugation_truncation}
\begin{aligned}
f_\text{\rm trunc}K_{i,j}f^{-1}_\text{\rm trunc} (Y) &= I_{y_i+y_j}(Y) \\
& \overset{\text{\rm def}}{=} K_{i,j}(Y).
\end{aligned}
\end{equation}
Moreover, the $f_{\text{\rm trunc}}$-conjugation of $Q_i$ on $TIX_2$ is $Q_i$ on $TIX_1$:
\begin{equation} \label{equ:Qi_conjugation_truncation}
f_\text{\rm trunc}Q_if^{-1}_\text{\rm trunc}=Q_i,
\end{equation}
and the $f_\text{\rm trunc}$-conjugation of the generalized contextual group associated to $X_2$ is the generalized contextual group associated $X_1$. Consequently, either one of these generalized contextual groups is generated by contextual inversions if and only if the other is.
\end{theorem}
\begin{proof}
Because of the $TI$-equivariance of $f_\text{ext}$, we have $f_\text{ext} I_{y_i+y_j} = I_{y_i+y_j} f_\text{ext}$ as functions, so
\begin{align*}
f_\text{ext}K_{i,j}f^{-1}_\text{ext} \langle y_1, \dots, y_{n_1}, y_{n_1+1}, \dots, y_{n_2} \rangle  &= f_\text{ext}K_{i,j}\langle y_1, \dots, y_{n_1}\rangle \\
&= f_\text{ext} I_{y_i+y_j} \langle y_1, \dots, y_{n_1} \rangle \\
&= I_{y_i+y_j} f_\text{ext} \langle y_1, \dots, y_{n_1} \rangle \\
&= I_{y_i+y_j} \langle y_1, \dots, y_{n_1}, y_{n_1+1}, \dots, y_{n_2} \rangle.
\end{align*}
The claim in equation \eqref{equ:Qi_conjugation} is fairly clear because $f_\text{ext}$ preserves the property of being a $T$-form or $I$-form. From \eqref{equ:contextual_conjugation} and \eqref{equ:Qi_conjugation} we see that conjugation with $f_\text{ext}$ is an isomorphism of generalized contextual groups.

All of the theorem statements about truncation follow from the already proved statements about extension because $f_\text{trunc} = f_\text{ext}^{-1}$.
\end{proof}

Another useful kind of $TI$-equivariant bijection that can be used in Theorem~\ref{thm:construction_of_action_on_disjoint_union} for $f$ is alterations (change one pitch class by a half-step), or more generally, modifications (replace a pitch class by another), although modifications do not conjugate $K_{i,j}$ to $K_{i,j}$ if $i$ or $j$ is the index of the changed pitch class.

\begin{theorem}[Modification of a Pitch-Class Segment and Conjugation of Contextual Groups] \label{thm:modifications}
Let $X_1=\langle x_1, \dots, x_{n} \rangle$ be a selected, fixed pitch-class segment that contains two distinct pitch-classes $x_q$ and $x_r$ which span an interval other than a tritone. Let $X_2$ be the modification of $X_1$ in which just the single pitch class $x_{j_0}$ in position $j_0$ is replaced by the pitch class $x_{j_0}'$, and suppose $X_2$ also satisfies the aforementioned tritone condition. Consider the $TI$-orbits $S_1 := TIX_1$ and $S_2 := TIX_2$, and let $f_\text{\rm mod}$ be the unique $TI$-equivariant map
$$f_\text{\rm mod}\colon \; TIX_1 \to TIX_2$$
that extends the modification $X_1 \mapsto X_2$. Then for $1 \leq i,j \leq n_1$ such that $i \neq j_0$ and $j \neq j_0$, the $f_\text{\rm mod}$-conjugation of $K_{i,j}$ on $TIX_1$ is $K_{i,j}$ on $TIX_2$. This means for $Y \in TIX_2$,
\begin{equation} \label{equ:contextual_conjugation_modification}
\begin{aligned}
f_\text{\rm mod}K_{i,j}f^{-1}_\text{\rm mod} (Y) &= I_{y_i+y_j}(Y) \\
& \overset{\text{\rm def}}{=} K_{i,j}(Y).
\end{aligned}
\end{equation}
Moreover, the $f_{\text{\rm mod}}$-conjugation of $Q_i$ on $TIX_1$ is $Q_i$ on $TIX_2$:
\begin{equation} \label{equ:Qi_conjugation_modification}
f_\text{\rm mod}Q_if^{-1}_\text{\rm mod}=Q_i.
\end{equation}
Although the $f_{\text{\rm mod}}$-conjugation of $K_{i,j_0}$ involving $j_0$ may not even be a contextual inversion, the $f_\text{\rm mod}$-conjugation of the generalized contextual group associated to $X_1$ is still nevertheless the generalized contextual group associated to $X_2$.
\end{theorem}
\begin{proof}
Equation~\eqref{equ:contextual_conjugation_modification} holds because $i$ and $j$ avoid $j_0$, so both $Y$ and $f_\text{mod}^{-1}(Y)$ have the same respective entries in positions $i$ and $j$, namely $y_i$ and $y_j$. Equation~\eqref{equ:Qi_conjugation_modification} follows from the $TI$-equivariance of $f_\text{mod}$. The $f_\text{mod}$-conjugation of the $X_1$ generalized contextual group is equal to the $X_2$ generalized contextual group because it commutes with the $TI$-group and has 24 elements.
\end{proof}

The following modification of majors to strides and minors to strains via multiplication by 10 in \cite{fiorenollsatyendraSchoenberg} illustrates what happens when we conjugate $K_{i,j_0}$.

\begin{example}[Modification of Majors to Strides, Minors to Strains] \label{examp:strides_and_strains} The modification
$$\langle 0, 4, 7 \rangle \mapsto \langle 0, 4, 10 \rangle$$
arises from (non-invertible) multiplication with 10, and causes also the modification
$$\langle 0, -4, -7 \rangle \mapsto \langle 0, -4, -10 \rangle.$$ The resulting $TI$-equivariant bijection
\begin{equation} \label{equ:stride_bijection}
f_\text{mod}\colon\; TI\langle0,4,7 \rangle \rightarrow TI\langle 0,4,10\rangle
\end{equation}
does not conjugate $K_{2,3}$ to $K_{2,3}$:
\begin{equation} \label{equ:stride_bijection_K23}
f_\text{mod} K_{2,3} f_\text{mod}^{-1} \neq K_{2,3}.
\end{equation}
In fact we have
\begin{align*}
f_\text{mod} K_{2,3} f_\text{mod}^{-1} \langle w, w+4, w+10 \rangle & = f_\text{mod} K_{2,3} \langle w, w+4, w+7 \rangle \\
&= f_\text{mod} \langle w+11, w+7, w+4 \rangle \\
&= \langle w+11, w+7, w+1 \rangle.
\end{align*}
We see the input and final output pitch-class segments have not even a single pitch class in common, so $f_\text{mod}K_{2,3}f_\text{mod}^{-1}$ is not any contextual inversion at all, especially not $K_{2,3}$. Recall that the generalized contextual group of $\langle 0,4,7 \rangle$ (which is the $PLR$-group) is generated by just $L=K_{2,3}$ and $R=K_{1,2}$ because $LR=Q_5$. However, this is not the case for generalized contextual group of $\langle 0,4,10\rangle$ because on $TI\langle 0, 4, 10 \rangle$
\begin{equation} \label{equ:LR_for_strides}
L^{\langle 0,4,10 \rangle } R^{\langle 0, 4, 10\rangle } = K_{2,3} K_{1,2} =Q_2,
\end{equation}
where the superscript notation means ``on the $TI$-class of $\langle 0,4,10 \rangle$''.
Nevertheless, the conjugation of the $PLR$-group by \eqref{equ:stride_bijection} is equal to the generalized contextual group of $\langle 0 , 4, 10 \rangle$ by Theorem~\ref{thm:modifications}. This conjugation isomorphism is compatible with \eqref{equ:LR_for_strides} because of \eqref{equ:stride_bijection_K23}, or in other words
$$f_\text{mod} L^{\langle 0, 4, 7 \rangle} f_\text{mod}^{-1} \neq L^{\langle 0 , 4, 10 \rangle}.$$
Strides and strains were considered in Section~4 of \cite{fiorenollsatyendraSchoenberg}. In Remark~1 on page 14 of that paper concerning multiplication by 7 and jets and sharks, it is important to realize that the $f$ there is multiplication by 7 and is not the same as the $TI$-equivariant bijection arising from the modification
$$\langle 0,4,7 \rangle \mapsto \langle 0,4,1 \rangle,$$
for instance, the resulting $TI$-equivariant bijection does $\langle 1, 5, 8 \rangle \mapsto \langle 1, 5, 2 \rangle$, which is clearly quite different from multiplication by 7. Conjugation of $Q_5$ by the modification bijection is still $Q_5$, not $Q_{11}$ as a misreading of Remark 1 on page 14 of that paper would suggest.
\end{example}

\begin{example}[Consonant Triads, Diminished Seventh Chords, and Dominant/Half-Diminished Seventh Chords] \label{examp:alteration_diminished_sevenths}
Modifications and extensions can be combined and composed. The single-element assignment
$$\langle 0,4,7 \rangle \mapsto \langle 1,4,7,10 \rangle$$
is a combination of a modification and an extension, and creates a $TI$-equivariant bijection from consonant triads in dualistic root position to diminished seventh chord pitch-class segments in Table~\ref{table:seventhchord_functions}, but it does not conjugate contextual inversions to contextual inversions, as illustrated for $P$ in Remark~\ref{rem:diminished_sevenths_mystery_solution}. We can combine this with truncation to obtain the musical alteration of a dominant seventh chord to a diminished seventh chord
$$\langle 0,4,7,10 \rangle \mapsto \langle 0,4,7 \rangle \mapsto  \langle 1,4,7,10 \rangle,$$
which gives rise to the composite of an inverted bijection and bijection from Table~\ref{table:seventhchord_functions} that systematically does a musical alteration
$$\xymatrix{\dom \ar[r] & \triads \ar[r] & \diminishedsevenths.}$$
\end{example}

Moving to a different kind of example, we can also consider smaller group actions on smaller sets, such as group actions on major triads and minor triads separately, which are then unioned. We do this now for a mixed group, but later reconstruct the $TI$-group and the $PLR$-group in Examples~\ref{examp:TI_via_antiequivariantconstruction} and \ref{examp:PLR_via_antiequivariantconstruction} in a similar way after we prove an ``anti-equivariant'' version of Theorem~\ref{thm:construction_of_action_on_disjoint_union} in Theorem~\ref{thm:anti-equivariant}.
\cite{popoff2013} proposed understanding certain musical groups as group extensions and short exact sequences, such as the $TI$-group and $PLR$-group, though did not contain the following example.

\begin{example}[$T$-Group as $G$ with {\it Wechsel} $P$ as Bijection Makes a Self-Dual Group]
The group of transpositions $\langle T_1 \rangle$ acts simply transitively on the major triads in root position, and also acts simply transitively on the minor triads in reverse root position, so in Theorem~\ref{thm:construction_of_action_on_disjoint_union} we may take $G=\langle T_1 \rangle$ and take the orbits
$$S_1 := \langle T_1 \rangle \langle 0,4,7 \rangle \quad \quad \quad S_2 := \langle T_1 \rangle \langle 7,3,0 \rangle,$$
together with $P\colon \; S_1 \to S_2$ as the equivariant bijection $f$. Then Theorem~\ref{thm:construction_of_action_on_disjoint_union} constructs a simply transitive commutative group acting on \triads.

Theorem~\ref{thm:construction_of_action_on_disjoint_union} implies the group structure of this group extension is an internal direct product, so the extended group is commutative as $\langle T_1 \rangle$ was already commutative. The dual group to $\langle T_1 \rangle$ acting on major triads is of course itself, since it is commutative, so the extension of the dual group to $\langle T_1 \rangle$ (which is the dual group of the extended group of by part \ref{thm:construction_of_action_on_disjoint_union:v:H} Theorem~\ref{thm:construction_of_action_on_disjoint_union}) is the same as the extended group of $\langle T_1 \rangle$. This is expected because the extended group is already commutative, as we observed a moment ago.

Clearly, the extended group and its dual are both the same as the mixed group $\langle T_1, P\rangle$ where $T_1$ and $P$ both act on the union, so the theorem just confirms what we already know about this group. This example merely serves as a consistency check on Theorem~\ref{thm:construction_of_action_on_disjoint_union} in a familiar setting.
\end{example}

\section{Construction of a simply transitive group action on a disjoint union: anti-equivariant bijection produces an internal semi-direct product} \label{sec:anti-equivariant_disjoint_union}

Theorem~\ref{thm:construction_of_action_on_disjoint_union} worked well for equivariant bijections and direct products, but it does not apply to {\it anti-equivariant} bijections and {\it semi-direct} products. For instance, the $TI$-group and the $PLR$-group cannot be reconstructed with Theorem~\ref{thm:construction_of_action_on_disjoint_union}. Thus, we provide in Theorem~\ref{thm:anti-equivariant} a modified version of Theorem~\ref{thm:construction_of_action_on_disjoint_union} to construct an internal semi-direct product from an anti-equivariant bijection. This modified version {\it will} reconstruct the $TI$-group and the $PLR$-group, see Examples~\ref{examp:TI_via_antiequivariantconstruction} and \ref{examp:PLR_via_antiequivariantconstruction}. In Example~\ref{examp:PLR_via_antiequivariantconstruction} we have another new proof that the $PLR$-group acts simply transitively and has order 24.

The assumption of a $G$-anti-equivariant bijection $f$ in Theorem~\ref{thm:anti-equivariant} instead of a $G$-equivariant bijection as in Theorem~\ref{thm:construction_of_action_on_disjoint_union} causes some big differences. The group $\langle \overline{f} \rangle$ is no longer normal in $\overline{G}$, and the function $\overline{G} \to G$ given by $g\overline{f}^i \mapsto g$ is no longer a homomorphism. Consequently, the short exact sequence in part \ref{thm:anti-equivariant:iii:short_exact_sequence} is no longer left split and no longer isomorphic to a direct product sequence. But it is right split, and is isomorphic to a semi-direct product sequence. Another big difference is the formulation of duality for the extended groups in part \ref{thm:anti-equivariant:v:H}. The extended groups cannot both contain $f$ due to anti-equivariance, so instead we begin with two pairs $(G,f)$ and $(H,k)$, each a group with its own anti-equivariant bijection $S_1 \to S_2$, and these two pairs are required to suitably commute with each other. From this conglomeration we obtain duality between the extended groups in the anti-equivariant case. The classical duality of the $TI$-group and $PLR$-group, and its pc-seg generalization to the $TI$-group and generalized contextual group, are prime examples of this extended duality in the anti-equivariant case.

We state Theorem~\ref{thm:anti-equivariant} in its entirety because it is not a one-to-one rephrasing of Theorem~\ref{thm:construction_of_action_on_disjoint_union}. However, we skip the proofs that are straightforward modifications of proofs in Theorem~\ref{thm:construction_of_action_on_disjoint_union}. There is no part (6) in Theorem~\ref{thm:anti-equivariant} because the $H$-action on $S_2$ is {\it not} the $f$-conjugation of the $H$-action on $S_1$.

\begin{theorem}[Construction of an Extension from an Anti-Equivariant Bijection of Faithful Group Actions, and Extension of Duality in Simply Transitive Cases]
\label{thm:anti-equivariant}
Suppose a group $G$ acts faithfully on two disjoint sets $S_1$ and $S_2$ and suppose $f\colon \; S_1 \to S_2$ is a bijection that is $G$-anti-equivariant in the sense that
$f( g \cdot s)=g^{-1} \cdot f(s)$ for all $g \in G$ and all $s \in S_1$. Let $\overline{S}:=S_1 \bigsqcup  S_2$ and define $\overline{f}\colon \; \overline{S} \to \overline{S}$ by
$$\xymatrix@C=3pc{ S_1 \bigsqcup S_2  \ar[r]_{f \bigsqcup f^{-1}} \ar@/^1pc/[rr]^{\overline{f}} & \ar@{=}[r] S_2 \bigsqcup S_1 & S_1 \bigsqcup S_2 }.$$
Let every $g$ in $G$ act on the disjoint union $\overline{S}$ as $g \bigsqcup  g$, so that $G$ also acts on $\overline{S}$ faithfully and we may consider $G$ as a subgroup of $\text{\rm Sym}(\overline{S})$ via the associated embedding $G \hookrightarrow \text{\rm Sym}(\overline{S})$. Let $\overline{G}:=\langle G, \overline{f} \rangle$, generated inside of $\text{\rm Sym}(\overline{S})$. Then the following statements hold.
\begin{enumerate}
    \item \label{thm:anti-equivariant:i:anti-equivariance}
    $\overline{f}$ is a $G$-anti-equivariant involution.
    \item \label{thm:anti-equivariant:ii:internal_direct_product}
    $G$ is normal in $\overline{G}$, and $\overline{G}$ is the internal semi-direct product of $G \rtimes \langle\overline{f}\rangle$. The conjugation of elements of $G$ by $\overline{f}$ is inversion, that is $\overline{f} g \overline{f}^{-1} = g^{-1}$. Moreover, $\overline{G}$ is the disjoint union $$\overline{G}=G \cup \overline{f}G=G \cup G\overline{f}$$ where $G$ is exactly the set of those transformations in $\overline{G}$ that preserve $S_1$ and $S_2$, and the coset $\overline{f}G=G\overline{f}$ is exactly the set of those transformations in $\overline{G}$ that exchange $S_1$ and $S_2$.
    \item \label{thm:anti-equivariant:iii:short_exact_sequence}
    The groups $G$, $\overline{G}$, and $\langle \overline{f} \rangle$ fit into the following short exact sequence, where the third homomorphism $\overline{G} \to \langle \overline{f} \rangle$ is $g \overline{f}^i \mapsto \overline{f}^i$.
    $$\{Id_{\overline{S}}\}\longrightarrow G\longrightarrow \overline{G}\longrightarrow \langle \overline{f}\rangle\longrightarrow \{Id_{\overline{S}}\}$$
    This short exact sequence is right split, so is isomorphic to the semi-direct product short exact sequence for $G \rtimes \langle \overline{f} \rangle$, which we already know from part \ref{thm:anti-equivariant:ii:internal_direct_product}.
    \item  \label{thm:anti-equivariant:iv:simple_transitivity}
    If the action of $G$ on $S_1$ is simply transitive, then the action of $G$ on $S_2$ is also simply transitive, and moreover the action of $\overline{G}$ on the disjoint union $\overline{S}$ is also simply transitive.
    \item
    \label{thm:anti-equivariant:v:H}
    Suppose the action of $G$ on $S_1$ is simply transitive. Suppose $H$ is another group that acts on $S_1$ and $S_2$, and suppose that $k\colon \; S_1 \to S_2$ is an $H$-anti-equivariant bijection. Define $\overline{k} = k \bigsqcup k^{-1}$. Suppose $H$ acts simply transitively on $S_1$, so that parts \ref{thm:anti-equivariant:i:anti-equivariance}, \ref{thm:anti-equivariant:ii:internal_direct_product}, \ref{thm:anti-equivariant:iii:short_exact_sequence}, and \ref{thm:anti-equivariant:iv:simple_transitivity} apply to $H$ and $k$. Suppose
    \begin{itemize}
    \item
    $k^{-1}f = f^{-1}k$,
    \item
    the actions of $G$ and $H$ on $S_1$ commute,
    \item
    $f$ is $H$-equivariant, and
    \item
    $k$ is $G$-equivariant.
    \end{itemize}
    Then the actions of $G$ and $H$ on $S_2$ commute, $\overline{f}$ and $\overline{k}$ commute, and $\overline{G}\overset{\text{\rm def}}{=}\langle G, \overline{f} \rangle$ and $\overline{H}\overset{\text{\rm def}}{=} \langle H, \overline{k}\rangle $ are dual groups in $\text{\rm Sym}(\overline{S})$.
    \end{enumerate}
\end{theorem}

\begin{proof}
\begin{enumerate}
\item\textbf{$\overline{f}$ is $G$-anti-equivariant and has order 2.}
This proof is completely analogous to the proof of Theorem~\ref{thm:construction_of_action_on_disjoint_union}~\ref{thm:construction_of_action_on_disjoint_union:i:equivariance}, we merely replace every external $g$ by $g^{-1}$.

\item\textbf{$\overline{G}$ is an Internal Semi-Direct Product.} Let $\overline{G} = \langle G, \overline{f} \rangle$.
Since $\overline{f}$ is $G$-anti-equivariant by part \ref{thm:anti-equivariant:i:anti-equivariance}, $G$ is normal in $\overline{G}$
because
$$\overline{f} g \overline{f}^{-1} = g^{-1} \overline{f} \, \overline{f}^{-1} = g^{-1}.$$
Also from the $G$-anti-equivariance of $\overline{f}$,
any product of elements of $G$ and powers of $\overline{f}$ can be moved into an element of the form $g\overline{f}^i$, so $\overline{G}=G \cdot \langle \overline{f} \rangle$. Lastly, $G \cap \langle \overline{f} \rangle = Id_{\overline{S}}$ because every element of $G$ preserves $S_1$ and $S_2$ while $\overline{f}$ exchanges $S_1$ and $S_2$ and is an involution, so $\overline{G}$ is the internal semi-direct product $G \rtimes \langle \overline{f} \rangle$.

Since $\overline{G} = G \rtimes \langle \overline{f} \rangle$, every element of $\overline{G}$ can be written uniquely as $g \overline{f}^i$ for some $g \in G$ and some $i=0,1$ (recall $\overline{f}$ is an involution), so there are only two right $G$-cosets, $G$ and $G\overline{f}$. This implies there are only two left $G$-cosets, which must be $G$ and $\overline{f}G$, so $\overline{f}G=G\overline{f}$. We now have
$$\overline{G}=G \cup \overline{f}G=G \cup G\overline{f}.$$ Since $G$ preserves $S_1$ and $S_2$, and $\overline{f}$ exchanges $S_1$ and $S_2$, we know each element of $\overline{f}G=G\overline{f}$ exchanges $S_1$ and $S_2$.

\item\textbf{Short Exact Sequence.}
The homomorphism $\overline{G} \to \langle \overline{f} \rangle$ given by the formula $g \overline{f}^i \mapsto \overline{f}^i$ is clearly surjective onto $\langle \overline{f} \rangle$, and the kernel is clearly $G$, so we have the short exact sequence.

A right inverse homomorphism for the surjection $\overline{G} \to \langle \overline{f} \rangle$ is the inclusion $\overline{f}^i \mapsto \overline{f}^i$, so the short exact sequence is isomorphic to the semi-direct product short exact sequence, see Theorem~3.3 of the lecture note \cite{conradkeith_splitting}. Notice the set-theoretic left inverse $g\overline{f}^i \mapsto g$
for the inclusion $G \rightarrow \overline{G}$ is {\it not} a homomorphism because $\overline{f}$ is $G$-anti-equivariant.

\item\textbf{Simple Transitivity of $\overline{G}$.}
This proof is completely analogous to the proof of Theorem~\ref{thm:construction_of_action_on_disjoint_union}~\ref{thm:construction_of_action_on_disjoint_union:iv:simple_transitivity}, anti-equivariance equalities merely replace the equivariance equalities.

\item\textbf{Duality of $\overline{H}$ and $\overline{G}$.}
We now have two pairs $(G,f)$ and $(H, k)$ that act simple transitively on $S_1$ and suitably commute with each other, as enumerated in the hypotheses.

The actions of $G$ and $H$ on $S_2$ commute because
\begin{align*}
ghs_2 & = ghf(s_1) \\
      &=  f(g^{-1}hs_1) \\
      &=  f(hg^{-1}s_1) \\
      &=  hgf(s_1) \\
      &=  hgs_2.
\end{align*}

The bijections $\overline{f}$ and $\overline{k}$ commute because $f$ conjugation of the assumed equation $k^{-1} f = f^{-1} k$ yields the other half of the equation $\overline{f} \, \overline{k} = \overline{k} \, \overline{f}$.

From Theorem~\ref{thm:construction_of_action_on_disjoint_union}~\ref{thm:construction_of_action_on_disjoint_union:i:equivariance}, $\overline{f}$ is $H$-equivariant and $\overline{k}$ is $G$-equivariant, so now $\overline{G}$ and $\overline{H}$ commute inside $\text{Sym}(\overline{S})$. Each acts simply transitively on $\overline{S}$ from part \ref{thm:anti-equivariant:iv:simple_transitivity}. From commutativity and simply transitivity of both groups, Theorem~3.1 of \cite{BerryFiore} finally allows us to conclude that $\overline{G}$ and $\overline{H}$ are dual groups.
\end{enumerate}
\end{proof}

Now we can reconstruct the $TI$-group and $PLR$-group, and also their duality, and as a bonus show that the $PLR$-group is the same as the {\it Schritt-Wechsel} group.

\begin{example}[Reconstruction of $TI$-Group from $\langle T_1 \rangle$ and Anti-Equivariant $I_0$] \label{examp:TI_via_antiequivariantconstruction}
We would like to obtain the following expected isomorphism of short exact sequences from Theorem~\ref{thm:anti-equivariant}~\ref{thm:anti-equivariant:iii:short_exact_sequence}.
\begin{equation} \label{equ:ses_for_TI}
\vcenter{
\xymatrix{\{Id_{\overline{S}}\} \ar[r] & \langle T_1 \rangle \ar[r] & TI\text{-group} \ar[r] & \langle I_0 \rangle \ar[r] & \{Id_{\overline{S}}\} \\
\{0\} \ar[r] \ar[u]^\cong & \mathbb{Z}_{12} \ar[r] \ar[u]^\cong & \mathbb{Z}_{12} \rtimes \{\pm 1\} \ar[r] \ar[u]^\cong  & \{\pm 1\} \ar[r] \ar[u]^\cong  & \{1\} \ar[u]^\cong  }
}
\end{equation}
We take $G=\langle T_1\rangle$, $S_1 = \langle T_1 \rangle \langle 0, 4, 7\rangle$, and $S_2 = \langle T_1 \rangle \langle 7, 3, 0\rangle$. Then $I_0\colon \; S_1 \to S_2$ is $G$-anti-equivariant
$$I_0 T_n(s_1) = (T_{n})^{-1} I_0(s_1).$$
Theorem~\ref{thm:anti-equivariant} applies, $\overline{I_0}$ is $I_0$, and $\overline{G}=\langle T_1, I_0 \rangle$ is the $TI$-group acting simply transitively on $\overline{S} = \triads$. Thus we obtain the desired isomorphism of short exact sequences in \eqref{equ:ses_for_TI}. Conjugation by $I_0$ is group element inversion
$$I_0^{-1} T_n I_0 = (T_n)^{-1},$$
which corresponds to additive inversion in $\mathbb{Z}_{12}$ in the bottom short exact sequence. Our theorem smoothly embeds the $TI$-action on consonant triads into the familiar expression of dihedral groups as semi-direct products via short exact sequences.
\end{example}

Theorem~\ref{thm:anti-equivariant} applies to the neo-Riemannian $PLR$-group in two ways, corresponding to the two ways of writing the $PLR$-group: in terms of $L$ and $R$, or in terms of {\it Schritte} and {\it Wechsel}.

\begin{example}[Reconstruction of $PLR$-Group from $\langle LR \rangle$ and Anti-Equivariant $R$, and also from {\it Schritt} group $\langle Q_1 \rangle$ and Anti-Equivariant {\it Wechsel} $W_0=P$] \label{examp:PLR_via_antiequivariantconstruction}
We would like to obtain the following expected isomorphisms of short exact sequences from Theorem~\ref{thm:anti-equivariant}~\ref{thm:anti-equivariant:iii:short_exact_sequence}.
\begin{equation} \label{equ:ses_for_plr}
\vcenter{
\xymatrix{\{Id_{\overline{S}}\} \ar[r] & \langle LR \rangle \ar[r] & PLR\text{-group} \ar[r] & \langle R \rangle \ar[r] & \{Id_{\overline{S}}\} \\
\{0\} \ar[r] \ar[u]^\cong \ar[d]_\cong & \mathbb{Z}_{12} \ar[r] \ar[u]^\cong \ar[d]_\cong & \mathbb{Z}_{12} \rtimes \{\pm 1\} \ar[r] \ar[u]^\cong \ar[d]_\cong & \{\pm 1\} \ar[r] \ar[u]^\cong \ar[d]_\cong & \{1\} \ar[u]^\cong \ar[d]_\cong \\
\{Id_{\overline{S}}\} \ar[r] & \langle Q_1 \rangle \ar[r] & PLR\text{-group} \ar[r] & \langle P \rangle \ar[r] & \{Id_{\overline{S}}\} }
}
\end{equation}
For the top row, we take $G=\langle LR\rangle$, $S_1 = \langle LR \rangle \langle 0, 4, 7\rangle$, and $S_2 = \langle LR \rangle \langle 7, 3, 0\rangle$.  The composite $LR$ means first do $R$, then do $L$, and the notation $\langle LR \rangle$ means the group generated by the sole generator $LR$. The composite $LR$ (on \triads) transposes major chords up by a perfect fourth, and minor chords down by a perfect fourth,\footnote{The alternating application of $R$ then $L$, starting with $C$, is a famous chord progression in Beethoven's {\it Ninth Symphony} studied in \cite{cohn1997}. See page 487 of \cite{cransfioresatyendra} for an application of this chord progression to prove that the $PLR$-group is generated by $L$ and $R$ and is dihedral of order 24. In the present example of the present paper, we are considering the composite function $LR$ rather than the alternating sequence $R$ then $L$.} so has order 12, as 5 is relatively prime to 12. The restrictions of $LR$ to $S_1$ and $S_2$ also have order 12 for the same reason. Note that $LR$ is equal to the {\it Schritt} $Q_5$.

For the top row we take the $G$-anti-equivariant bijection $f$ to be $R$. This bijection is anti-equivariant with respect to the group $G=\langle LR \rangle$, as we confirm with the generator $LR$:
\begin{align*}
f \circ LR &= R \circ LR \\
&= R(LR) \\
&= (RL)R \\
&= (LR)^{-1} R \\
&= (LR)^{-1} \circ f.
\end{align*}
Theorem~\ref{thm:anti-equivariant} applies, $\overline{R}$ is $R$, and $\overline{G}=\langle LR, R\rangle$ acts simply transitively on $\overline{S} = \triads$, and we obtain the desired isomorphism of short exact sequences in the top half of \eqref{equ:ses_for_plr}. Conjugation by $R$ is group element inversion
$$R^{-1} (LR) R = RL = (LR)^{-1}$$
which corresponds to additive inversion in $\mathbb{Z}_{12}$ in the middle short exact sequence.

What the demonstration so far has {\it not} shown is that $\overline{G}\overset{\text{def}}{=}\langle LR, R\rangle$ is the same as $\langle P, L, R \rangle$. However, that follows from the two equalities:
\begin{equation} \label{equ:P_formula}
P=R(LR)^3
\end{equation}
\begin{equation} \label{equ:same_generated_group}
\langle L, R\rangle = \langle LR, R \rangle.
\end{equation}
Equality \eqref{equ:P_formula} follows from the argument on page 487 of \cite{cransfioresatyendra}.
Equality \eqref{equ:same_generated_group} is basic algebra (show each set of generators can be written in terms of the other set of generators). We now have yet another proof that the $PLR$-group, defined as $\langle P,L,R \rangle$, is dihedral of order 24 and acts simply transitively on \triads.

In summary we have
$$ \langle L,R \rangle = \langle LR, R \rangle = \langle LR \rangle \rtimes \langle R \rangle \cong \mathbb{Z}_{12} \rtimes \{\pm 1\}.$$ 

For the bottom row of \eqref{equ:ses_for_plr}, we use the {\it Schritt} $Q_1$ and the {\it Wechsel} $W_0=P$. We take $G=\langle Q_1 \rangle$, $S_1 = \langle Q_1 \rangle \langle 0, 4, 7\rangle$, and $S_2 = \langle Q_1 \rangle \langle 7, 3, 0\rangle$. We select the $G$-anti-equivariant bijection to be $P$.

Theorem~\ref{thm:anti-equivariant} applies, $\overline{P}$ is $P$, and $\overline{G}=\langle Q_1, P \rangle$ acts simply transitively on $\overline{S} = \triads$, and we obtain the desired isomorphism of short exact sequences in the bottom half of \eqref{equ:ses_for_plr}. Conjugation by $P$ is group element inversion
$$P^{-1} Q_n P = (Q_n)^{-1}$$
which corresponds to additive inversion in $\mathbb{Z}_{12}$ in the middle short exact sequence. The group $\overline{G}=\langle Q_1, P \rangle$ will be equal to $\langle LR, R\rangle$ by the uniqueness of the dual group to the $TI$-group acting on \triads, as soon as we know that both of the constructions are dual to the $TI$-group in the next example.
\end{example}

\begin{example}[Duality of $TI$-Group and $PLR$-Group] \label{examp:anti-equivariant_duality_TI-PLR}
We would now like to apply part \ref{thm:anti-equivariant:v:H} of Theorem~\ref{thm:anti-equivariant} to Examples~\ref{examp:TI_via_antiequivariantconstruction} and \ref{examp:PLR_via_antiequivariantconstruction}.
We take $$G=\langle T_1 \rangle, \quad f=I_0, \quad H=\langle LR \rangle, \quad \text{and} \quad k=R.$$
The hypotheses of part \ref{thm:anti-equivariant:v:H} are known to be true, so $\langle T_1, I_0\rangle$ and $\langle LR, R\rangle$ are dual groups.

Alternatively, we can use the second part of Example~\ref{examp:PLR_via_antiequivariantconstruction}. We take $$G=\langle T_1 \rangle, \quad f=I_0, \quad H=\langle Q_1 \rangle, \quad \text{and} \quad k=P.$$
The hypotheses of part \ref{thm:anti-equivariant:v:H} are known to be true, so $\langle T_1, I_0\rangle$ and $\langle Q_1, P\rangle$ are dual groups.

Finally,  $\langle LR, R\rangle= \langle Q_1, P\rangle$, and we obtain another proof of the equality of the $PLR$-group and the {\it Schritt-Wechsel} group (using $\langle P,L,R \rangle= \langle L, R \rangle = \langle LR, R\rangle$ from before).
\end{example}

\begin{example}[Reconstruction of the Generalized Contextual Group from $\langle Q_1 \rangle$ and Anti-Equivariant $K_{1,2}$] \label{examp:generalizedcontextualgroupagain}
The $PLR$-group is one particular example of a generalized contextual group, as defined in \cite{fioresatyendra2005}. It is easy to modify the (second) application of Theorem~\ref{thm:anti-equivariant} in Example~\ref{examp:PLR_via_antiequivariantconstruction} to reconstruct the generalized contextual group of a pitch-class segment and its duality. We would like to obtain the following isomorphism of short exact sequences.
\begin{equation} \label{equ:ses_for_generalizedcontextualgroup}
\vcenter{
\xymatrix{\{Id_{\overline{S}}\} \ar[r] & \langle Q_1 \rangle \ar[r] & \text{Gen. Contextual Group} \ar[r] & \langle K_{1,2} \rangle \ar[r] & \{Id_{\overline{S}}\} \\
\{0\} \ar[r] \ar[u]^\cong & \mathbb{Z}_{12} \ar[r] \ar[u]^\cong & \mathbb{Z}_{12} \rtimes \{\pm 1\} \ar[r] \ar[u]^\cong  & \{\pm 1\} \ar[r] \ar[u]^\cong  & \{1\} \ar[u]^\cong  }
}
\end{equation}

Recall the review of the generalized contextual group of a pitch-class segment in Example~\ref{examp:generalizedcontextualgroup}. Let $X=\langle x_1, \dots, x_{n}\rangle$ be a selected, fixed pitch-class segment that contains two distinct pitch-classes $x_q$ and $x_r$ which span an interval other than a tritone (in the present example we are {\it not} doing a pitch-class segment extension as in Example~\ref{examp:generalizedcontextualgroup}).

We take $$G=\langle Q_1 \rangle,\quad S_1 =\langle Q_1 \rangle X, \quad \text{and} \quad S_2 = K_{1,2}(S_1)=\langle Q_1 \rangle K_{1,2}X.$$
This means $S_1$ is the set of transposed forms of $X$ and $S_2$ is the set of inverted forms of $X$.
We take $K_{1,2} \colon \; S_1 \to S_2$ to be the $G$-anti-equivariant bijection $f$. We write $\overline{f} = K_{1,2} \bigsqcup K_{2,1}$ as in the Theorem~\ref{thm:anti-equivariant} construction, but since $K_{1,2} = K_{2,1}$, we simply have $\overline{f} = K_{1,2}$.

It is clear that $G$ acts simply transitively on both sets $S_1$ and $S_2$.  Therefore, by Theorem~\ref{thm:anti-equivariant}, $\langle Q_1, K_{1,2}\rangle$ acts simply transitively on $\overline{S}=$ the set of both transposed and inverted forms of $X$, and $\langle Q_1, K_{1,2}\rangle$ is the internal semi-direct product of $\langle Q_1 \rangle$ and $\langle K_{1,2}\rangle$, written as $\langle Q_1,K_{1,2}\rangle=\langle Q_1\rangle\rtimes \langle K_{1,2}\rangle $.

The hypotheses of part \ref{thm:anti-equivariant:v:H} of Theorem~\ref{thm:anti-equivariant} are known to be true from \cite{fioresatyendra2005}, so $\langle T_1, I_0\rangle$ and $\langle Q_1, K_{1,2}\rangle$ are dual groups.
\end{example}

\section{$TI$-Equivariant bijections from consonant triads to other seventh chords}\label{sec:bijections_of_triads_and_seventh_chords}

In Section~\ref{sec:notations} we introduced a bijection $f$ between consonant triads and dominant/half-diminished seventh chords in equations \eqref{equ:bijection_triads_domhalfdiminished} and \eqref{equ:2y-x}. More precisely, this is a bijection from the set class of $\langle 0,4,7 \rangle$ to the set class of $\langle 0,4,7,10\rangle$, in symbols:
\begin{equation}
f\colon \; TI\langle 0,4,7 \rangle \to TI\langle 0,4,7,10 \rangle.
\end{equation}
We indicated the function in three different ways:
\begin{enumerate}
\item
the verbal formula:
``map a consonant triad to the unique dominant/half-diminished seventh chord containing it,'' which means \\ $\langle w, w+4, w+7 \rangle \mapsto \langle w, w+4, w+7, w+10 \rangle$ \\and \\$\langle w, w-4, w-7 \rangle \mapsto \langle w, w-4, w-7, w-10 \rangle$,
\item
the linear formula: $\langle w, x, y \rangle \mapsto \langle w, x, y, 2y-x \rangle$, and
\item
the unique $TI$-equivariant extension of the single-element assignment
\begin{equation} \label{equ:single-element_dominant}
\langle 0, 4, 7 \rangle \mapsto \langle 0, 4, 7, 10 \rangle
\end{equation}
to the $TI$-orbits.
\end{enumerate}
The third method has a distinct advantage, as it is a systematic framework for three things: the required $TI$-equivariance of Proposition~\ref{prop:TI-equivariance_contextual_extension}, the $f$-conjugation consistency of $P$, $L$, and $R$ in Proposition~\ref{prop:plr_conjugated_action}, and the extended duality between $TI$-group actions and generalized contextual groups, as explained in Example~\ref{examp:generalizedcontextualgroup} and Theorem~\ref{thm:conjugation_with_extension}. This third method and its benefits also apply to the equivariant bijections of consonant triads with major sevenths and minor sevenths via the single-element assignments
\begin{equation} \label{equ:single-element_majorsevenths}
\langle 0,4,7 \rangle \mapsto \langle 0,4,7,11 \rangle
\end{equation}
\begin{equation} \label{equ:single-element_minorsevenths}
\langle 0,4,7 \rangle \mapsto \langle 0,4,7,9 \rangle
\end{equation}
by Example~\ref{examp:generalizedcontextualgroup} and Theorem~\ref{thm:conjugation_with_extension}. Moreover, conjugation of any contextual inversion with either of these bijections is also a contextual inversion. This third method also works for the single-element assignment
\begin{equation} \label{equ:single-element_diminishedsevenths}
\langle 0,4,7 \rangle \mapsto \langle 1,4,7,10 \rangle
\end{equation}
by Theorem~\ref{thm:modifications} and Example~\ref{examp:alteration_diminished_sevenths}, although conjugation of $P$ and $R$ by this bijection are not contextual inversions.

In this section we remind the reader about these three other seventh chords, make a few remarks about the formalization of their collections as $TI$-orbits of certain pitch-class segments, and explain the bijections in more detail. In particular, we give linear formulas for the bijections, and prove these formulas are $TI$-equivariant, as we will need to apply Theorem~\ref{thm:generalized}.

\begin{figure}
\begin{center}
\includegraphics[scale = 0.4]{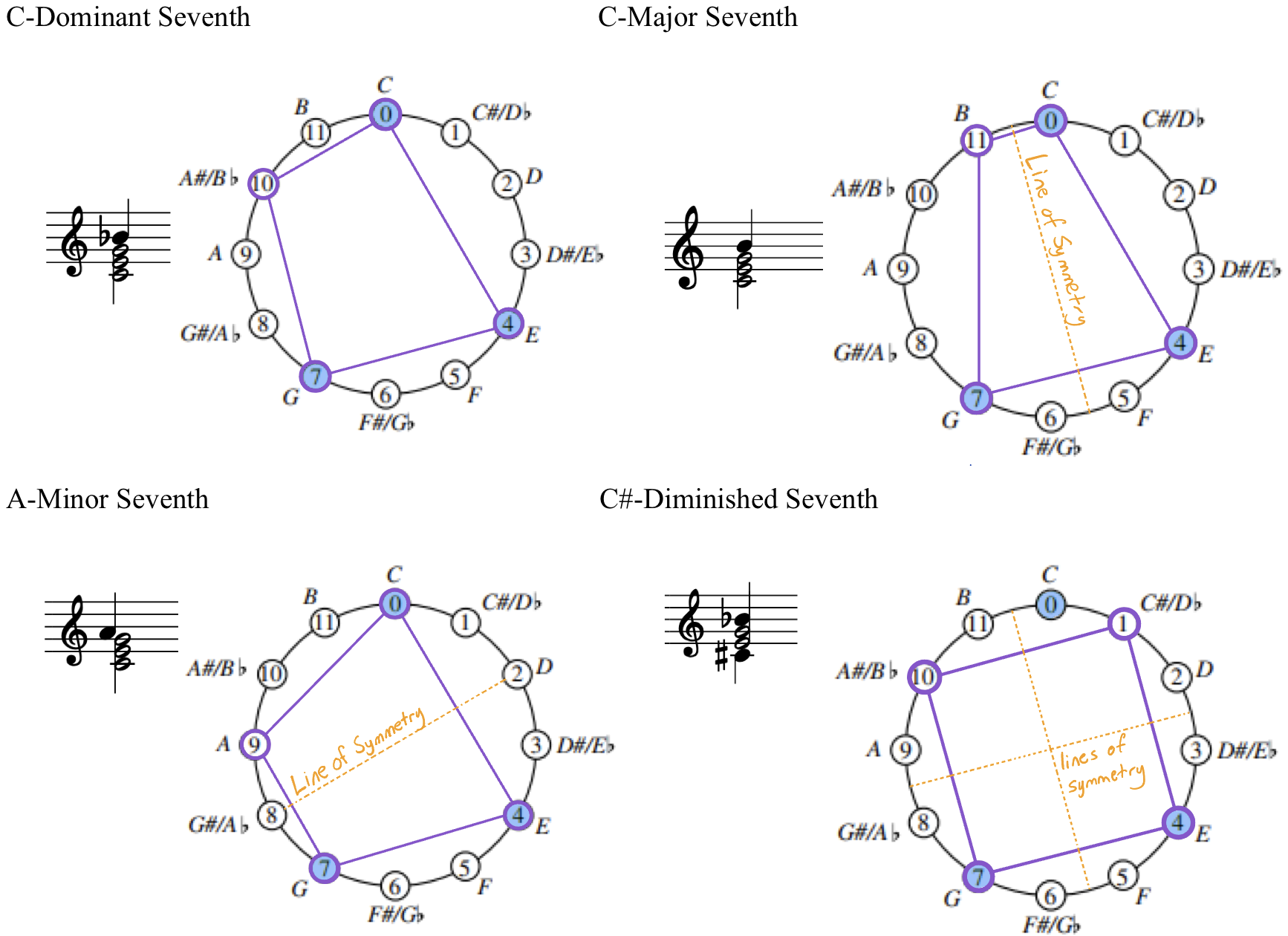}
\end{center}
\caption{Here we see representatives of the 4 different $TI$-orbits of seventh chord pitch-class segments from Table~\ref{table:seventhchord_sets}, as well as how the $C$-major chord is contained (or only partially contained) in these representative seventh chords.  The lines of symmetry of each representative seventh chord are also shown. The reflection of an ascending dominant seventh pc-seg is a descending half-diminished pc-seg (not pictured), so the $TI$-class of the dominant seventh pc-seg $C^7$ contains also the half-diminished seventh pc-segs. Any reflection of an ascending major seventh pc-seg is a descending major seventh pc-seg. Any reflection of a minor seventh chord is another minor seventh chord in opposite direction (note: we choose to write the generating minor seventh pc-seg $A\textsuperscript{-7}$ as $\langle 0,4,7,9 \rangle$ with the root 9 moved last so that the $C$-major chord is its first three entries). Major sevenths and minor sevenths are {\it not} related to each other via reflection (not even as pitch-class sets), against our intuition from consonant triads. Any reflection of an ascending diminished seventh pc-seg is a descending diminished seventh chord. The root entry for our pc-seg selections is indicated in the fourth column of Table~\ref{table:seventhchord_sets}. In the present figure, the shading on the note heads and the musical clock points shows how the $C$-major triad is embedded in $C^7$, $C^{\triangle 7}$, and $A^{-7}$. The $C$-major triad is not embedded in any diminished seventh, but two of its notes are in $C\sh\textsuperscript{o7}$. \label{fig:seventh_images} }
\end{figure}

\begin{table}
    \renewcommand{\arraystretch}{1.2}
    \centering
    \begin{tabular}{|p{9em}|p{5.5em}|p{9.5em}|p{2.5em}|p{6em}|}
    \hline
    \multirow{2}{*}{Set Name} & \multirow{2}{4em}{Set as \\ $TI$-Orbit} & \multirow{2}{*}{Elements} & \multirow{2}{*}{Root} & \multirow{2}{*}{Chord Type} \\
    & & & & \\
    \hline
    \multirow{3}{9em}{\small{\triads}} & \multirow{3}{4em}{\small $TI\langle 0,4,7 \rangle$} &  \multirow{3}{9em}{\small $\langle w,w+4,w+7 \rangle$ \\
    $\langle w,w-4,w-7 \rangle$} & \multirow{3}{2.5em}{\small $w$\\$w-7$} & \multirow{3}{10em}{\small Major Triad \\ Minor Triad} \\
     & & & & \\
     & & & & \\
    \hline
    \multirow{3}{9em}{\small{\dom}} & \multirow{3}{4em}{\small $TI\langle 0,4,7,10 \rangle$} & \multirow{3}{9em}{\small $\langle w,w+4,w+7,w+10 \rangle$ \\ $\langle w,w-4,w-7,w-10 \rangle$} & \multirow{3}{2.5em}{\small $w$\\$w-10$} & \multirow{3}{10em}{\small Dom. 7th \\ Half-Dim. 7th} \\
     & & & & \\
     & & & & \\
    \hline
   \multirow{3}{9em}{\small{\majorsevenths}} & \multirow{3}{4em}{\small $TI\langle 0,4,7,11 \rangle$} &  \multirow{3}{9em}{\small $\langle w,w+4,w+7,w+11 \rangle$ \\ $\langle w,w-4,w-7,w-11 \rangle$} & \multirow{3}{2.5em}{\small $w$\\$w-11$} & \multirow{3}{10em}{\small Major 7th \\ Major 7th } \\
     & & & & \\
     & & & & \\
    \hline
    \multirow{3}{9em}{\small{\minorsevenths}} & \multirow{3}{4em}{\small $TI\langle 0,4,7,9 \rangle$} & \multirow{3}{9em}{\small $\langle w,w+4,w+7,w+9 \rangle$ \\ $\langle w,w-4,w-7,w-9 \rangle$} & \multirow{3}{2.5em}{\small $w+9$\\$w-7$} & \multirow{3}{10em}{\small Minor 7th \\ Minor 7th} \\
     & & & & \\
     & & & & \\
    \hline
    \multirow{3}{9em}{\small{\diminishedsevenths}} & \multirow{3}{4em}{\small $TI\langle 1,4,7,10 \rangle$} & \multirow{3}{9.5em}{\footnotesize $\langle w+1,w+4,w+7,w+10 \rangle$ \\ $\langle w-1,w-4,w-7,w-10 \rangle$} & \multirow{3}{2.5em}{\small $w+1$\\$w-10$} & \multirow{3}{10em}{\small Dim. 7th \\ Dim. 7th} \\
     & & & & \\
     & & & & \\
    \hline
    \end{tabular}
    \caption{This table shows the set of consonant triads and sets of various seventh chords, as $TI$-orbits of selected pitch-class segments. Note that the $TI$-classes of all these seventh chord pc-segments have 24 elements (despite symmetries) because we are using pc-segments instead of pitch-class sets (the major seventh chord pc-set and the minor seventh chord pc-set each have a reflectional symmetry, and the diminished seventh chord pc-set even has two reflectional symmetries and 4 transpositional symmetries, but the corresponding pc-segments do not have any symmetries). Each major seventh chord has two copies in \majorsevenths, one in ascending order and one in descending order. Similarly, each minor seventh chord has two copies in \minorsevenths, although neither is in root position. Each of the three diminished seventh chords even has 8 copies in \diminishedsevenths, its four rotations and the descending version of its four rotations. \label{table:seventhchord_sets}}
    \end{table}

Dominant seventh chords were reviewed in Section~\ref{sec:notations}. The thirds in a dominant seventh chord from root upwards are major third, minor third, minor third, as we see in Figure~\ref{fig:seventh_images}. Half-diminished seventh chords were also reviewed in Section~\ref{sec:notations}. The entry-wise reflection of a dominant seventh chord is a half-diminished seventh chord, so the thirds in a half-diminished seventh chord from root upwards are the opposite: minor third, minor third, major third ({\it not} pictured in Figure~\ref{fig:seventh_images}). We write dominant seventh chords as pitch-class segments in root position and half-diminished seventh chords as pitch-class segments in reverse root position in order to match with entry-wise inversion, see Table~\ref{table:seventhchord_sets}.

In a {\it major seventh chord}, the thirds from root upwards are major third, minor third, major third as we see in Figure~\ref{fig:seventh_images}. Because of this symmetry, any reflection of a major seventh chord is also a major seventh chord, which implies every major seventh chord is represented twice in the $TI$-class of $\langle 0 , 4, 7, 11 \rangle$, once in root position and once again in reverse root position. For instance, $\langle 0 , 4, 7, 11 \rangle$ and $\langle 11 , 7, 4, 0 \rangle$ are both in the set \majorsevenths, which denotes the $TI$-class of $\langle 0 , 4, 7, 11 \rangle$. See Table~\ref{table:seventhchord_sets}. Nevertheless, this set has 24 elements since we use pitch-class segments instead of pitch-class sets. The notation for a major seventh chord is a superscript \textsuperscript{$\triangle$7}.

In a {\it minor seventh chord}, the thirds from root upwards are minor third, major third, minor third as we see in Figure~\ref{fig:seventh_images} (starting with pitch-class $A$ this time instead of pitch-class $C$). Because of this symmetry, any reflection of a minor seventh chord is also a minor seventh chord. We do {\it not} write the pitch-class segments for minor seventh chords in root position, nor in reverse root position! Notice in the $A$-minor seventh chord in Figure~\ref{fig:seventh_images} the $C$-major chord is (uniquely) embedded {\it above} the root $A$. For notational convenience, we write $\langle 0, 4, 7 \rangle$ to be first in the pitch-class segment $\langle 0, 4, 7, 9 \rangle$ in order to be consistent with dominant seventh chords and major seventh chords, see Table~\ref{table:seventhchord_sets} and Table~\ref{table:seventhchord_functions}. Although this extension on $C$ major
$$\langle 0,4,7 \rangle \mapsto \langle 0,4,7,9 \rangle, \quad \quad C \mapsto A^{-7}$$
looks awkward because it changes the root, it actually has as a consequence of $TI$-equivariance the desired behavior on minors: every minor triad maps to its minor seventh chord with the same root, and the root is in the same entry! For instance, for root $A=9$ in the third entry of the $a$ minor pc-seg we have
$$\langle 4,0,9 \rangle \mapsto \langle 4,0,9,7 \rangle, \quad \quad a \mapsto A^{-7}.$$
A different analyst might choose a different bijection between consonant triads and minor seventh pitch-class segments. Every minor seventh chord is represented twice in the $TI$-class of $\langle 0 , 4, 7, 9 \rangle$, denoted \minorsevenths. For instance, $\langle 0 , 4, 7, 9 \rangle$ and $\langle 4 , 0, 9, 7 \rangle$ are both in \minorsevenths. The notation for a minor seventh chord is a superscript \textsuperscript{-7}. The minor seventh chords are not reflectionally related to the major seventh chords, as we can see in the distributions of notes in the two chords in Figure~\ref{fig:seventh_images}.

Lastly, the {\it diminished seventh chords} consist of minor third, minor third, and minor third upwards from the root. Because of this symmetry, there are only three diminished seventh chord pitch-class sets in $\mathbb{Z}_{12}$ (visualize the rotations of the square set in Figure~\ref{fig:seventh_images}). However, we work with the $TI$-class of the pitch class-segment $\langle 1,4,7,10\rangle$, denoted \diminishedsevenths. This set contains 24 elements. Because the diminished seventh chord pitch-class set is evenly distributed in $\mathbb{Z}_{12}$, any of the 4 pitch classes can be considered a root, depending on the musical context. For consistency with all other collections (except minor sevenths), we arbitrarily declare in Table~\ref{table:seventhchord_sets} the root of a $T$-form to be the first entry, and the root of an $I$-form to be the last entry. The notation for a diminished seventh chord is a superscript \textsuperscript{o7}.

The redundancies in \majorsevenths, \minorsevenths, and \\ \diminishedsevenths \ will allow us to construct a simply transitive group action on the union of all five sets in Table~\ref{table:seventhchord_sets}.

Having finished this review of seventh chords, our selected formalization with pitch-class segments, notations, and consequences of our choices, we can now state and prove the results we need to apply Section~\ref{sec:construction_for_disjoint_union_multiple}. We claim the formulas in Table~\ref{table:seventhchord_functions} are the maps induced by the single-element assignments in \eqref{equ:single-element_dominant}, \eqref{equ:single-element_majorsevenths}, \eqref{equ:single-element_minorsevenths}, and \eqref{equ:single-element_diminishedsevenths}. To confirm the formulas coincide with the maps induced by the single-element assignments, it suffices to check that the formulas do the corresponding single-element assignment and to confirm that the formulas are $TI$-equivariant. The correct output of each formula on $\langle 0,4,7 \rangle$ is easily seen with the formulas. We prove the $TI$-equivariance in Proposition~\ref{prop:TI-equivariance_of_the_seventh_bijections}.


\begin{table}
    \renewcommand{\arraystretch}{2} 
    \centering
    \begin{tabular}{|p{18em}|p{18em}|} 
    \hline
    Function & Formula on $T$-Forms and $I$-Forms \vspace{.03in} \\ 
    \hline
    \multirow{2}{18em}{\small $f_{Dom7Tr}\colon\;TI\langle0,4,7\rangle \to TI\langle0,4,7,10 \rangle$ \\ \vspace{.05in}
    $\phantom{f_{Dom7Tr}\colon\;} \quad\langle w, x, y  \rangle \mapsto  \langle w,\;x,\;y,\;2y-x \rangle$ }
    & \multirow{2}{18em}{\small \\$\langle w,w+4,w+7 \rangle \mapsto \langle w,w+4,w+7,w+10 \rangle$ \\
      $\langle w,w-4,w-7 \rangle \mapsto \langle w,w-4,w-7,w-10 \rangle$} \\
     &  \\
    \hline
    \multirow{2}{18em}{\small $f_{Maj7Tr}\colon\; TI\langle0,4,7\rangle \to TI\langle0,4,7,11 \rangle$ \\ \vspace{.05in}
    $\phantom{f_{Maj7Tr}\colon\;} \quad\langle w, x, y  \rangle \mapsto  \langle w,\;x,\;y,\;x+y-w \rangle$}
    & \multirow{2}{18em}{\small \\$\langle w,w+4,w+7 \rangle
      \mapsto \langle w,w+4,w+7,w+11 \rangle $ \\
      $\langle w,w-4,w-7 \rangle
      \mapsto \langle w,w-4,w-7,w-11 \rangle$} \\
     &  \\
    \hline
    \multirow{2}{18em}{\small $f_{Min7Tr}\colon \;TI\langle0,4,7\rangle \to TI\langle0,4,7,9 \rangle$ \\ \vspace{.05in}
    $\phantom{f_{Min7Tr}\colon\;} \quad\langle w, x, y  \rangle \mapsto  \langle w,\;x,\;y,\;w+x-y \rangle$}
    & \multirow{2}{18em}{\small \\$\langle w,w+4,w+7 \rangle
      \mapsto \langle w,w+4,w+7,w+9 \rangle $ \\
      $\langle w,w-4,w-7 \rangle
      \mapsto \langle w,w-4,w-7,w-9 \rangle$}
     \cr &   \\
    \hline
    \multirow{2}{20em}{\small $f_{Dim7Tr}\colon \;TI\langle0,4,7\rangle \to TI\langle 1,4,7,10 \rangle$ \\ \vspace{.05in}
    \hspace{-.1in}$\phantom{f_{Dim7Tr}\colon\;} \quad\langle w, x, y  \rangle \mapsto  \langle 2x-y,\;x,\;y,\;2y-x \rangle$}
    & \multirow{2}{18em}{\small \\$\langle w,w+4,w+7 \rangle
      \mapsto \langle w+1,w+4,w+7,w+10 \rangle $ \\
      $\langle w,w-4,w-7 \rangle
      \mapsto \langle w-1,w-4,w-7,w-10 \rangle$} \\
     & \\
    \hline
    \end{tabular}
    \caption{We will use these $TI$-equivariant bijections in an application of Theorem~\ref{thm:generalized}. These $TI$-equivariant bijections go from consonant triads to various seventh chords, all formalized as the $TI$-classes of selected pitch-class segments in Table~\ref{table:seventhchord_sets}. The subscripts have the domain indicated on the right because right is closer to the argument, in other words for a function with domain set $U$ and codomain set $V$ we write $f_{VU}\colon\; U \rightarrow V$ so that we can write $f_{VU}(u)$ for $u \in U$. Notice that the third bijection, the one from consonant triads to minor seventh chords, is quite intuitive on minor triads: it sends a minor triad to the minor seventh chord with the same root name in the same entry, for instance $a$ minor maps to $A^{-7}$, that is $\langle 4,0,9 \rangle \mapsto \langle 4,0,9,7\rangle$ with root $9$.}
    \label{table:seventhchord_functions}
\end{table}

\begin{proposition} \label{prop:TI-equivariance_of_the_seventh_bijections}
The four functions in Table~\ref{table:seventhchord_functions} are $TI$-equivariant bijections. The top three bijections conjugate contextual inversions to contextual inversions, the last does not. All four bijections conjugate $Q_i$ to $Q_i$.
\end{proposition}
\begin{proof}
The discovery of the first function in Table~\ref{table:seventhchord_functions} and the proof of its $TI$-equivariance were already completed in the proof of Proposition~\ref{prop:TI-equivariance_contextual_extension}. The formula discovery and equivariance proofs for the remaining functions are very similar, so we abbreviate.

For the discovery, we attempt to write the new interval in terms of the existing intervals $4$ and $7$. We do the discovery only on the major triads, since the minor triads are very similar just with negative signs inserted. For major seventh chords: $11 \equiv 4+7$, so
\begin{align*}
    x+y-w &= (w+4) + (w+7) - w \\
    &= w + 11.
\end{align*}
For minor seventh chords: $9 \equiv 4- 7$, so
\begin{align*}
    w+x-y &= w + (w+4) - (w+7) \\
    &= w + 9.
\end{align*}
For the first entry of diminished seventh chords: $1 \equiv 2\times 4 - 7$, so
\begin{align*}
    2x-y &= 2(w+4) - (w+7) \\
    &= w + 1.
\end{align*}
The last entry $2y-x$ of diminished seventh chords was already treated in the proof of Proposition~\ref{prop:TI-equivariance_contextual_extension}.

Since identity entries are clearly $TI$-equivariant, to confirm the $TI$-equivariance of each bijection it suffices to confirm the non-identity entries are $TI$-equivariant, as we did in Proposition~\ref{prop:TI-equivariance_contextual_extension} for the first bijection. A glance at all the non-identity entries of the functions in Table~\ref{table:seventhchord_functions} allows us to see that they are all linear and have two positives and one negative, so confirmation of equivariance with respect to $T_n$ reduces to $2n-n=n$. The linearity implies compatibility with $I_0$, so we also have equivariance with respect to all $I_n=T_n \circ I_0$.

By Example~\ref{examp:generalizedcontextualgroup} and Theorem~\ref{thm:conjugation_with_extension}, the first three functions in Table~\ref{table:seventhchord_functions} conjugate contextual inversions to contextual inversions, and conjugate $Q_i$ to $Q_i$. By Theorem~\ref{thm:modifications} and Example~\ref{examp:alteration_diminished_sevenths}, the last function in Table~\ref{table:seventhchord_functions} conjugates $Q_i$ to $Q_i$.
\end{proof}

We now have all the ingredients to apply Theorem~\ref{thm:generalized} to our main example of consonant triads together with several kinds of seventh chords.

\section{Construction of a simply transitive group action on a disjoint union of multiple sets}\label{sec:construction_for_disjoint_union_multiple}

While Theorem~\ref{thm:construction_of_action_on_disjoint_union} is useful for building
simply transitive group actions on two $TI$-classes of pitch-class segments, some musical passages include representatives from more than two $TI$-classes. For instance, {\it Canonic Passacaglia} by Clare Fischer in Figure~\ref{fig:canonic_passacaglia} contains three classes: major/minor triads, minor seventh chords, and half-diminished/dominant seventh chords. The progression repeats a pattern of minor triad, minor seventh chord, half-diminished seventh chord, then dominant seventh chord. Thus we generalize Theorem~\ref{thm:construction_of_action_on_disjoint_union} to Theorem~\ref{thm:generalized} to construct a simply transitive group action on multiple sets. The point of departure for the theorem is a {\it star-shaped} collection of faithful $G$-sets and $G$-equivariant bijections, meaning we have one $G$-set in the middle and then $G$-equivariant bijections emanating outwards. {\it Additionally}, we label the $G$-set in the middle as $S_1$ and we also have $S_1$ on the outside, connected with the central $G$-set by the identity map, see Figure~\ref{fig:canonic_passacaglia} for an illustration of a star shape with three sets. We propose a ``meta rotation" $\overline{f}$ to traverse around the sets, where $\overline{f}$ rotates through chord type sets, from one set to the next, as pictured in Figure~\ref{fig:canonic_passacaglia}.

\begin{theorem}[Construction of an Extension from a Star-Shaped Diagram of Faithful $G$-Sets and $G$-Equivariant Bijections, and Extension of Duality in Simply Transitive Cases]
\label{thm:generalized}
Suppose a group $G$ acts faithfully on $n$ disjoint sets $S_1$, $S_2$, \dots, $S_n$ and suppose for each $2 \leq j \leq n$ we have a $G$-equivariant bijection $f_j \colon \; S_1 \to S_j$. Define $f_1\colon \; S_1 \to S_1$ to be the identity function on $S_1$. See Figure~\ref{fig:canonic_passacaglia} for an example with $n=3$. Let $$\overline{S}:=S_1 \bigsqcup  \cdots \bigsqcup S_n$$ and define the ``meta-rotation'' $\overline{f}\colon \; \overline{S} \to \overline{S}$ by
$$\xymatrix@C=3pc{ \overline{f} := \bigsqcup\limits_{j=1}^{n} \left( f_{j+1} \circ f_j^{-1} \right) \colon \; \bigsqcup\limits_{j=1}^{n} S_j \ar[r] & \bigsqcup\limits_{j=1}^{n} S_{j+1}}$$
where the indices are read modulo $n$, so that
$$\xymatrix@C=3pc{\overline{f} = \left( f_2 \circ f_1^{-1} \right) \bigsqcup \cdots \bigsqcup  \left(f_1 \circ f_n^{-1}\right) \colon \; S_1 \bigsqcup  \cdots \bigsqcup S_{n-1} \bigsqcup S_n \ar[r] & S_2 \bigsqcup  \cdots \bigsqcup S_n \bigsqcup S_1}.$$
Let every $g$ in $G$ act on the disjoint union $\overline{S}$ as the union of its actions on $S_1$, \dots, $S_n$, so that $G$ also acts on $\overline{S}$ faithfully and we may consider $G$ as a subgroup of $\text{\rm Sym}(\overline{S})$ via the associated embedding $G \hookrightarrow \text{\rm Sym}(\overline{S})$. Let $\overline{G}:=\langle G, \overline{f} \rangle$, generated inside of $\text{\rm Sym}(\overline{S})$. Then the following statements hold.
\begin{enumerate}
    \item \label{thm:generalized:i:equivariance}
    $\overline{f}$ is $G$-equivariant and has order $n$.
    \item \label{thm:generalized:ii:internal_direct_product}
    $G$ and $\langle\overline{f}\rangle$ commute in $\text{\rm Sym}(\overline{S})$, and $\overline{G}$ is an internal direct product of $G$ and $\langle\overline{f}\rangle$. Moreover, $\overline{G}$ is the disjoint union
    \begin{align*}
    \overline{G} &= G \; \bigcup \; \overline{f}G \; \bigcup \; \cdots \; \bigcup \;\overline{f}^{n-1} G \\
    &= G \; \bigcup \; G\overline{f} \; \bigcup \; \cdots \; \bigcup \; G\overline{f}^{n-1}
    \end{align*}
     where $G$ is exactly the set of those transformations in $\overline{G}$ that preserve each $S_j$ for $1 \leq j \leq n$ and the coset $\overline{f}^iG=G\overline{f}^i$ is exactly the set of those transformations in $\overline{G}$ that send $S_j$ to $S_{j+i}$ for all $1 \leq j \leq n$ (indices are read modulo $n$).
    \item \label{thm:generalized:iii:short_exact_sequence}
    The groups $G$, $\overline{G}$, and $\langle \overline{f} \rangle$ fit into the following short exact sequence, where the third homomorphism $\overline{G} \to \langle \overline{f} \rangle$ is $g \overline{f}^i \mapsto \overline{f}^i$.
    $$\{Id_{\overline{S}}\}\longrightarrow G\longrightarrow \overline{G}\longrightarrow \langle \overline{f}\rangle\longrightarrow \{Id_{\overline{S}}\}$$
    This short exact sequence is left split, so is isomorphic to the direct product short exact sequence of $G$ and $\langle \overline{f} \rangle$, which we already know from part \ref{thm:generalized:ii:internal_direct_product}.
    \item  \label{thm:generalized:iv:simple_transitivity}
    If the action of $G$ on any one $S_{j_0}$ is simply transitive, then the action of $G$ on every $S_j$ is simply transitive, and moreover the action of $\overline{G}$ on the disjoint union $\overline{S}$ is also simply transitive.
    \item
    \label{thm:generalized:v:H}
    Suppose the action of $G$ on $S_1$ is simply transitive. Let $H$ be the centralizer of $G$ in $\text{\rm Sym}(S_1)$. Then $H$ is the Lewin dual group of $G$. Let $H$ act on $S_j$ via its embedded copy $f_jHf^{-1}_j \leqslant \text{\rm Sym}(S_j)$, so that $f_j$ is $H$-equivariant. Then parts \ref{thm:generalized:i:equivariance}, \ref{thm:generalized:ii:internal_direct_product}, \ref{thm:generalized:iii:short_exact_sequence}, and \ref{thm:generalized:iv:simple_transitivity} apply to $H$ and the $f_j$'s, and moreover $\overline{H}:=\langle  H,\overline{f}\rangle$ is the centralizer of $\overline{G}$ in $\text{\rm Sym}(\overline{S})$, so $\overline{H}$ and $\overline{G}$ are dual groups in $\text{\rm Sym}(\overline{S})$.
    \begin{enumerate}
    \item[(i)]
    $\overline{f}$ is $H$-equivariant.
    \item[(ii)]
    $H$ and $\langle \overline{f} \rangle$ commute in $\text{\rm Sym}(\overline{S})$, and $\overline{H}$  is an internal direct product of $H$ and $\langle \overline{f} \rangle$. Moreover, $\overline{H}$ is the disjoint union
    \begin{align*}
    \overline{H} &= H \; \bigcup \; \overline{f}H \; \bigcup \; \cdots \; \bigcup \;\overline{f}^{n-1} H \\
    &= H \; \bigcup \; H\overline{f} \; \bigcup \; \cdots \; \bigcup \; H\overline{f}^{n-1}
    \end{align*}
    where $H$ is exactly the set of those transformations in $\overline{H}$ that preserve each $S_j$ for $1 \leq j \leq n$ and the coset $\overline{f}^iH=H\overline{f}^i$ is exactly the set of those transformations in $\overline{H}$ that send $S_j$ to $S_{j+i}$ for all $1 \leq j \leq n$ (indices are read modulo $n$).
    \item[(iii)]
    The groups $H$, $\overline{H}$, and $\langle \overline{f} \rangle$ fit into a short exact sequence similar to part \ref{thm:generalized:iii:short_exact_sequence}.
    \item[(iv)]
    The action of $H$ on every $S_j$ is simply transitive, and the action of $\overline{H}$ on the disjoint union $\overline{S}$ is also simply transitive.
    \end{enumerate}
    \item \label{thm:generalized:vi:hbar}
    Suppose again the action of $G$ on $S_1$ is simply transitive. Let $H$ and its action on each $S_j$ be as in part \ref{thm:generalized:v:H}. Let $h \in H$ and let
    $$\overline{h}=\bigsqcup_{j=1}^n f_jhf^{-1}_j$$ be the extension of $h$ to $\overline{S}$. In particular, on $S_j$ this is $f_jhf_j^{-1}$. Then the function $\overline{f} \circ \overline{h} = \overline{h} \circ \overline{f}$ is $f_{j+1}hf^{-1}_j$ on $S_j$.
\end{enumerate}
\end{theorem}

\begin{proof}
\begin{enumerate}
\item\textbf{$\overline{f}$ is $G$-equivariant and has order $n$.}
Since each bijection $f_j$ is $G$-equivariant, each inverse $f^{-1}_j$ is also equivariant, and so each composite $f_{j+1} \circ f_j^{-1}$ is $G$-equivariant. The coproduct of $G$-equivariant maps is $G$-equivariant, so $\overline{f}$ is $G$-equivariant. The $G$-equivariance of $\overline{f}$ can be more concretely written on each $S_j$, as we did for $S_1$ and $S_2$ in the proof of Theorem~\ref{thm:construction_of_action_on_disjoint_union}~\ref{thm:construction_of_action_on_disjoint_union:i:equivariance}.

Next we see $\overline{f}$ has order $n$. Let $s_1 \in S_1$. The case $s_j \in S_j$ is similar but will start with $\left(f_{j+1} \circ f_j^{-1}\right)$ in place of $\left(f_2 \circ f_1^{-1}\right)$.
\begin{align*}
\overline{f}^n(s_1) &= \left(f_1 \circ f_n^{-1} \right) \circ \left( f_n \circ f_{n-1}^{-1} \right) \circ \cdots \circ
\left(f_2 \circ f_1^{-1}\right)(s_1)\\
&=  \left(f_1 \circ f^{-1}_1\right)(s_1) \\
&= s_1
\end{align*}

\item\textbf{$\overline{G}$ is an Internal Direct Product.} The groups $G$ and $\langle \overline{f} \rangle$ commute by \ref{thm:generalized:i:equivariance}. Let $\overline{G} = \langle G, \overline{f} \rangle$. Since $G$ and $\langle \overline{f} \rangle$ commute, we have $\overline{G} = G \cdot \langle \overline{f} \rangle$. Since $G$ and $\langle \overline{f} \rangle$ commute, we also have $G$ and $\langle \overline{f} \rangle$ are normal in $\overline{G}$. Lastly, $G \cap \langle \overline{f} \rangle = Id_{\overline{S}}$ because every element of $G$ preserves each $S_j$ while $\overline{f}^i$ sends $S_j$ to $S_{j+i}$ where indices are read modulo $n$ (by construction $\overline{f}$ sends every $S_j$ to $S_{j+1}$, so the $i$-th power of $\overline{f}$ will jump by $i$ sets). We have confirmed the three conditions for $\overline{G}$ to be the internal direct product of the commuting groups $G$ and $\langle \overline{f} \rangle$. Since $\overline{G}$ is an internal direct product of commuting groups, every element of $\overline{G}$ can be written uniquely as $g \overline{f}^i= \overline{f}^i g$ for some $g \in G$ and some $0 \leq i \leq n-1$, so
\begin{align*}
    \overline{G} &= G \; \bigcup \; \overline{f}G \; \bigcup \; \cdots \; \bigcup \;\overline{f}^{n-1} G \\
    &= G \; \bigcup \; G\overline{f} \; \bigcup \; \cdots \; \bigcup \; G\overline{f}^{n-1}
    \end{align*}
Since $G$ sends each $S_j$ to $S_j$, and $\overline{f}^i$ sends each $S_j$ to $S_{j+i}$, we know each element of $\overline{f}^iG=G\overline{f}^i$ sends $S_j$ to $S_{j+i}$ (indices are read modulo $n$).

\item\textbf{Short Exact Sequence.} This proof is completely analogous to the proof of Theorem~\ref{thm:construction_of_action_on_disjoint_union}~\ref{thm:construction_of_action_on_disjoint_union:iii:short_exact_sequence}, so we skip it.

\item\textbf{Simple Transitivity of $\overline{G}$.} Suppose $G$ acts simply transitively on $S_{j_0}$. We claim the action of $G$ on $S_j$ is also simply transitive. Let $s_j, s_j' \in S_j$. Then there exist $s_{j_0}, s_{j_0}' \in S_{j_0}$ and (unique) $g \in G$ such that
$$\left( f_j \circ f_{j_0}^{-1} \right)(s_{j_0}) = s_j, \quad \quad \left( f_j \circ f_{j_0}^{-1} \right)(s_{j_0}')= s_j', \quad \quad \text{and} \quad \quad gs_{j_0} = s_{j_0}'.$$
An application of $\left( f_j \circ f_{j_0}^{-1} \right)$ to the last equality yields $gs_j = s_j'$. If there is a second element of $G$ that moves $s_j$ to $s_j'$, then we can apply $\left( f_j \circ f_{j_0}^{-1} \right)^{-1}$ to the equation, and use the simple transitivity on $S_{j_0}$ to conclude we have the same $g$. Of course we are making use of the $G$-equivariance of $f_{j_0}$ and $f_j$ throughout. Thus, the action of $G$ on $S_j$ is also simply transitive.

We next prove transitivity of $\overline{G}$ on $\overline{S}$ by evaluating all group elements on any one $s_1 \in S_1$, using the simple transitivity of $G$ on every $S_j$.
\begin{align*}
\left(G \; \bigsqcup \; G\overline{f} \; \bigsqcup \; \cdots \; \bigsqcup \; G\overline{f}^{n-1}\right)s_1 &=
Gs_1 \; \bigsqcup \; (G\overline{f})s_1 \; \bigsqcup \; \cdots \; \bigsqcup \; (G\overline{f}^{n-1})s_1 \\
&= Gs_1 \; \bigsqcup \; G(\overline{f}s_1) \; \bigsqcup \; \cdots \; \bigsqcup \; G(\overline{f}^{n-1}s_1)  \\
&= S_1 \bigsqcup S_2 \bigsqcup \cdots \bigsqcup S_n\\
&= \overline{S}
\end{align*}
The orbit of any one $s_1 \in S_1$ is all of $\overline{S}$, so the $\overline{G}$-action on $\overline{S}$ is transitive, without any finiteness assumption.

The uniqueness part of simple transitivity on $\overline{S}$ follows quickly from the transitivity in the finite case from multiple applications of the Orbit-Stabilizer Theorem. Namely, if any one of $G$, $S_1$, $S_2$, \dots, or $S_n$ is finite, then so are all the others, and
$$|\overline{G}| = n \cdot |G| = n \cdot |S_1| = |\overline{S}|$$
for use in the Orbit-Stabilizer Theorem as in the proof of Theorem~\ref{thm:construction_of_action_on_disjoint_union}~\ref{thm:construction_of_action_on_disjoint_union:iv:simple_transitivity}. We skip the Orbit-Stabilizer equations because they are analogous.

If $G$, $S_1$, $S_2$, \dots, $S_n$ are not finite, the uniqueness part of simple transitivity on $\overline{S}$ can still be verified by an argument very similar to the non-finite proof in Theorem~\ref{thm:construction_of_action_on_disjoint_union}~\ref{thm:construction_of_action_on_disjoint_union:iv:simple_transitivity}. We skip the details for brevity.

\item\textbf{Duality of $\overline{H}$ and $\overline{G}$.}
Suppose $G$ acts simply transitively on $S_1$.
Let $H$ be the centralizer of the simply transitive group $G$ in $\text{\rm Sym}(S_1)$. Then $H$ also acts simply transitively, so $H$ is the dual group of $G$, by Proposition~3.2 of \cite{BerryFiore} (the claim that a centralizer of a simply acting group is its Lewin dual was stated on page 253 of \cite{LewinGMIT} but not proved there).

Let $H$ act on each $S_j$ via its embedded copy $f_jHf_j^{-1} \leqslant \text{\rm Sym}(S_j)$. Then each $f_j$ is $H$-equivariant by the equations \eqref{equ:f_is_H_equivariant}.

The group $H$ and bijections $f_j$ now satisfy the hypotheses on $G$ and bijections $f_j$ in the present theorem, so we can apply parts \ref{thm:generalized:i:equivariance}, \ref{thm:generalized:ii:internal_direct_product}, \ref{thm:generalized:iii:short_exact_sequence}, and \ref{thm:generalized:iv:simple_transitivity} to $H$ and the $f_j$'s. In particular, the group $\overline{H}:=\langle  H,\overline{f}\rangle$ in $\text{\rm Sym}(\overline{S})$ acts simply transitively on $\overline{S}$.

The groups $\overline{G}$ and $\overline{H}$ commute because
$$(g\overline{f}^i)(h\overline{f}^j) = (h\overline{f}^j)(g\overline{f}^i),$$
as all four of these elements commute and can be rearranged at will. By Proposition~3.1 of \cite{BerryFiore}, we conclude from the commutativity of $\overline{G}$ and $\overline{H}$ that $\overline{H}$ is actually the centralizer of $\overline{G}$, and $\overline{H}$ and $\overline{G}$ are dual groups in $\text{\rm Sym}(\overline{S})$.

\item\textbf{Formula for the Function $\overline{f} \,\overline{h}$.}
Suppose the action of $G$ on $S_1$ is simply transitive. Let $H$ and its action on each $S_j$ be as in part \ref{thm:generalized:v:H}. On $S_j$ we have
$$\overline{f} \, \overline{h} = \left( f_{j+1} f_j^{-1} \right) \left( f_j h f_j^{-1}\right)
= f_{j+1} h f_j^{-1}$$
and also
$$\overline{h} \, \overline{f} = \left( f_{j+} h f_{j+1}^{-1}\right)\left( f_{j+1} f_j^{-1} \right).$$
\end{enumerate}
\end{proof}

\begin{remark}
In Theorem~\ref{thm:generalized}~\ref{thm:generalized:v:H}, we actually have $n$ copies of the dual pair $G$ and $H$ that fit together via $\overline{f}$ into one big dual pair $\overline{G}$ and $\overline{H}$. In particular inside of each $\text{Sym}(S_j)$, we have the natural copy of $G$ and $f_j H f_{j}^{-1}$ in a dual relationship, and all of these dual pairs fit together with $\overline{f}$ to make a dual pair $\overline{G}$ and $\overline{H}$ on $\overline{S}=\bigsqcup_{j=1}^n S_j$.
\end{remark}

\begin{example}[Rotation is Meta-Rotation] \label{examp:tetractys}
Rotation of a pitch-class segment produces a meta-rotation. We illustrate this concretely with the modes of the tetractys, but the idea can be formulated on general pitch-class segments to construct rotation sub-dual groups of the dual permutation groups $\lambda(\Sigma_nTI\text{-group})$ and $\rho(\Sigma_nTI\text{-group})$ in Theorem~3.2 of \cite{fiorenollsatyendraMCM2013}. If we move the first letter to the end iteratively, we have the tetractys rotation sequence
$$FCG \quad \mapsto \quad CGF \quad  \mapsto \quad GFC \quad \mapsto \quad FCG.$$
The meta-rotation $\overline{f}$ constructed in Figure~\ref{fig:rotation_is_metarotation} does this to all transpositions and inversions of $FCG$, and the resulting simply transitive group $$\overline{G}=\langle TI\text{-group}, \, \overline{f} \rangle$$ from Theorem~\ref{thm:generalized} is the internal direct product of the $TI$-group and $\overline{f}$, and has order $24\times3=72$. The dual group $\overline{H}$ can be understood in terms of \cite{fiorenollsatyendraMCM2013}.
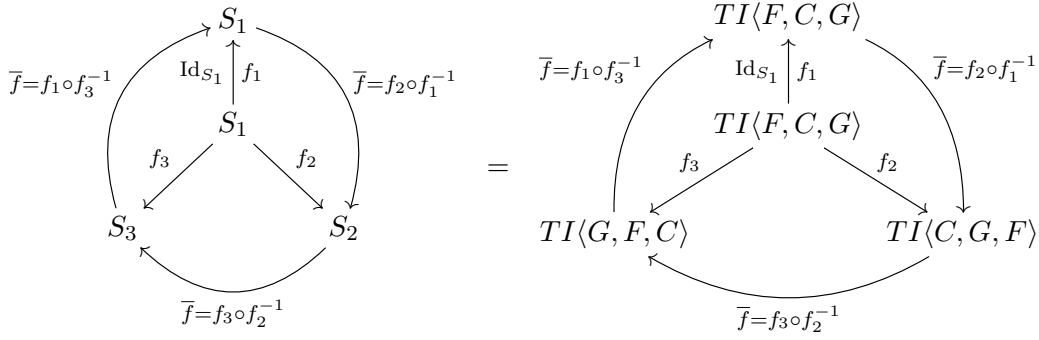
\begin{figure}
$$
\begin{array}{c}
\xymatrix{ & S_1 \ar@/^2.0pc/[ddr]^{\overline{f} = f_2 \circ f_1^{-1}} & \\
& S_1 \ar[dr]^{f_2} \ar[u]^{\text{Id}_{S_1}}_{f_1} \ar[dl]_{f_3} &  \\
S_3 \ar@/^2.0pc/[uur]^{\overline{f} = f_1 \circ f_3^{-1}} & & S_2 \ar@/^2.0pc/[ll]^{\overline{f} = f_3 \circ f_2^{-1}}}
\end{array}
=
\begin{array}{c}
\xymatrix@C=.1cm{ & TI\langle F,C,G \rangle \ar@/^2.0pc/[ddr]^{\overline{f} = f_2 \circ f_1^{-1}} & \\
& TI\langle F,C,G \rangle  \ar[dr]^{f_2} \ar[u]^{\text{Id}_{S_1}}_{f_1} \ar[dl]_{f_3} &  \\
TI\langle G,F,C \rangle  \ar@/^2.0pc/[uur]^{\overline{f} = f_1 \circ f_3^{-1}} & & TI\langle C,G,F \rangle  \ar@/^2.0pc/[ll]^{\overline{f} = f_3 \circ f_2^{-1}}}
\end{array}$$
\caption{Rotation produces a meta-rotation, illustrated in the special case of the modes of the tetractys. The function $f_2$ moves the first pitch class to the end, the function $f_3$ moves the first pitch class to the end and then the second pitch class to the end of that. Consequently, the meta-rotation $\overline{f}$ moves the first letter of any tetractys mode pitch-class segment to the end of the pitch-class segment.} \label{fig:rotation_is_metarotation}
\end{figure}
\end{example}

\begin{example}[{\it Canonic Passacaglia} with Three Classes]
We first apply Theorem~\ref{thm:generalized} to construct a pair of dual groups $\overline{G}$ and $\overline{H}$ for Clare Fischer's {\it Canonic Passacaglia}, acting on three classes: consonant triads, minor seventh chords, and dominant/half-diminished seventh chords. The musical staff at the top of Figure~\ref{fig:canonic_passacaglia} shows two successive chord progressions from the long progression of {\it Canonic Passacaglia}. We see $d$, $D^{-7}$, $B^{\o{} 7}$, $E^7$, and then its transposition down a perfect fourth: $a$, $A^{-7}$, $F\sh^{\o{} 7}$, $B^7$. In Theorem~\ref{thm:generalized} we take $G$ to be the $TI$-group, and take the three sets $S_1$, $S_2$, and $S_3$ to be
$$\triads, \quad \minorsevenths, \quad \text{and} \quad \dom.$$
The two function $S_1 \to S_2$ and $S_1 \to S_3$ are
$$f_{Min7Tr} \colon \;\triads \to \minorsevenths$$
$$f_{Dom7Tr} \colon \; \triads \to \dom$$
given by the formulas in the third row and first row of Table~\ref{table:seventhchord_functions}. These are pictured in the two equal star shaped diagrams at the top of Figure~\ref{fig:canonic_passacaglia}, with the meta-rotation $\overline{f}\colon \;\overline{S} \to \overline{S}$ superimposed. The middle star-shaped diagram shows what all the functions do elementwise on a minor triad in reverse root position.

Since $G$ is the $TI$-group and $S_1$ is $\triads$, $H$ must be the $PLR$-group. Theorem~\ref{thm:generalized} tells us that $\overline{f}$ has order 3, and
$$\overline{G} \cong TI\text{-group} \times \langle \overline{f} \rangle$$
$$\overline{H} \cong PLR\text{-group} \times \langle \overline{f} \rangle,$$
and both groups have order $24 \times 3=72$.

The bottom two networks of Figure~\ref{fig:canonic_passacaglia} show the transformational relationships between the sub-progressions of {\it Canonic Passacaglia}, the horizontal arrows are in $\overline{H}$, and the vertical arrows are in $\overline{G}$. The networks would extend downward to have 13 rows total, with the bottom row the same as the top row.
\begin{figure}
\begin{center}
\includegraphics[scale = 0.3]{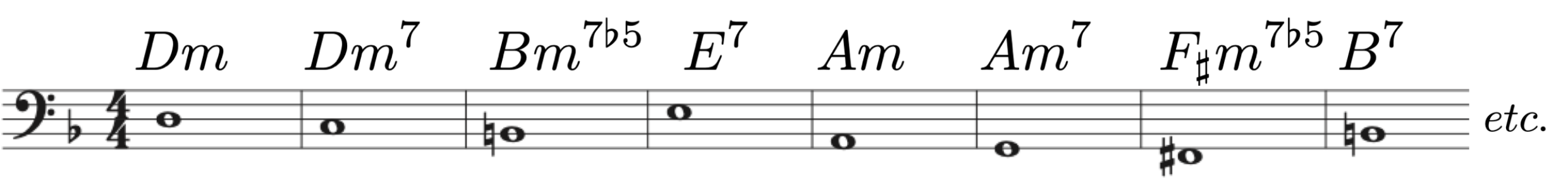}
\end{center}
$$
\begin{array}{c}
\xymatrix{ & S_1 \ar@/^2.0pc/[ddr]^{\overline{f} = f_2 \circ f_1^{-1}} & \\
& S_1 \ar[dr]^{f_2} \ar[u]^{\text{Id}_{S_1}}_{f_1} \ar[dl]_{f_3} &  \\
S_3 \ar@/^2.0pc/[uur]^{\overline{f} = f_1 \circ f_3^{-1}} & & S_2 \ar@/^2.0pc/[ll]^{\overline{f} = f_3 \circ f_2^{-1}}}
\end{array}
=
\begin{array}{c}
\xymatrix@C=.1cm{ & \triads \ar@/^2.0pc/[ddr]^{\overline{f} = f_2 \circ f_1^{-1}} & \\
& \triads \ar[dr]^{f_2} \ar[u]^{\text{Id}_{S_1}}_{f_1} \ar[dl]_{f_3} &  \\
\dom \ar@/^2.0pc/[uur]^{\overline{f} = f_1 \circ f_3^{-1}} & & \minorsevenths \ar@/^2.0pc/[ll]^{\overline{f} = f_3 \circ f_2^{-1}}}
\end{array}$$
$$\entrymodifiers={+<2mm>[F-]}
\xymatrix{ *=<0pt,0pt>{} & \langle w, w-4, w-7 \rangle \ar@/^2.0pc/[ddr]^{\overline{f} = f_2 \circ f_1^{-1}} & *=<0pt,0pt>{} \\
*=<0pt,0pt>{} & \langle w, w-4, w-7 \rangle \ar[dr]^{f_2} \ar[u]^{\text{Id}_{S_1}}_{f_1} \ar[dl]_{f_3} & *=<0pt,0pt>{} \\
\langle w, w-4, w-7, w-10 \rangle \ar@/^2.0pc/[uur]^{\overline{f} = f_1 \circ f_3^{-1}} & *=<0pt,0pt>{} & \langle w, w-4, w-7, w-9 \rangle \ar@/^2.0pc/[ll]^{\overline{f} = f_3 \circ f_2^{-1}}}
$$
$$\entrymodifiers={=<2.2pc>[o][F-]}
\xymatrix@C=2.5pc@R=2.5pc{d   \ar[r]^{\overline{f}} \ar[d]_{T_{-5}} & D^{-7} \ar[r]^{\overline{f}} \ar[d]_{T_{-5}} &  B^{\o{} 7}
\ar[r]^{K_{3,4}} \ar[d]^{T_{-5}} &  E^7 \ar[d]^{T_{-5}}  \\
a \ar[r]^{\overline{f}} & A^{-7} \ar[r]^{\overline{f}} & F\sh^{\o{} 7} \ar[r]^{K_{3,4}} &  B^7 } $$
$$\entrymodifiers={+<2mm>[F-]}
\xymatrix@C=2.5pc@R=2.5pc{\langle 9,5,2 \rangle   \ar[r]^{\overline{f}} \ar[d]_{T_{-5}} & \langle 9,5,2,0 \rangle \ar[r]^{\overline{f}} \ar[d]_{T_{-5}} &  \langle 9,5,2,11 \rangle
\ar[r]^{K_{3,4}} \ar[d]^{T_{-5}} &  \langle 4,8,11,2 \rangle\ar[d]^{T_{-5}}  \\
\langle 4,0,9 \rangle \ar[r]^{\overline{f}} & \langle 4,0,9,7\rangle \ar[r]^{\overline{f}} & \langle 4,0,9,6 \rangle \ar[r]^{K_{3,4}} &  \langle 11, 3,6, 9 \rangle } $$
\caption{{\it Canonic Passacaglia} by Clare Fischer. The notation $Dm^7$ is another notation for the $D$ minor seventh chord, $D^{-7}$. Recall that a minor 7 flat 5 chord is another name for a half-diminished seventh chord, so $Bm^{7\fl 5}$ is the same as $B^{\o{} 7}$. This is an example of Theorem~\ref{thm:generalized} with 3 sets $S_1$, $S_2$, and $S_3$ which are \triads, \minorsevenths, and \dom. In Theorem~\ref{thm:generalized} we took $G$ to be $TI$-group  and we took $H$ to be the $PLR$-group. The top two star-shaped diagrams of $G$-sets are equal, and show also the meta-rotation $\overline{f}$. The middle star-shaped diagram shows the meta-rotation $\overline{f}$ element-wise on a minor triad in reverse root position with root $w-7$, a minor seventh chord in reverse root position with same root $w-7$, and a half-diminished seventh chord in reverse root position with root $w-10$. The bottom two networks both illustrate the first two sequences in {\it Canonic Passacaglia}. An audio excerpt is available at \url{https://www.youtube.com/watch?v=irYWgD4vs-4}.} \label{fig:canonic_passacaglia}
\end{figure}
\end{example}

\begin{example}[Simply Transitive Group Action on 5 Classes] \label{examp:5classes}
We next apply Theorem~\ref{thm:generalized} to construct a pair of dual groups $\overline{G}$ and $\overline{H}$ that act on five classes: consonant triads, major seventh chords, dominant/half-diminished seventh chords, minor seventh chords, and diminished seventh chords.
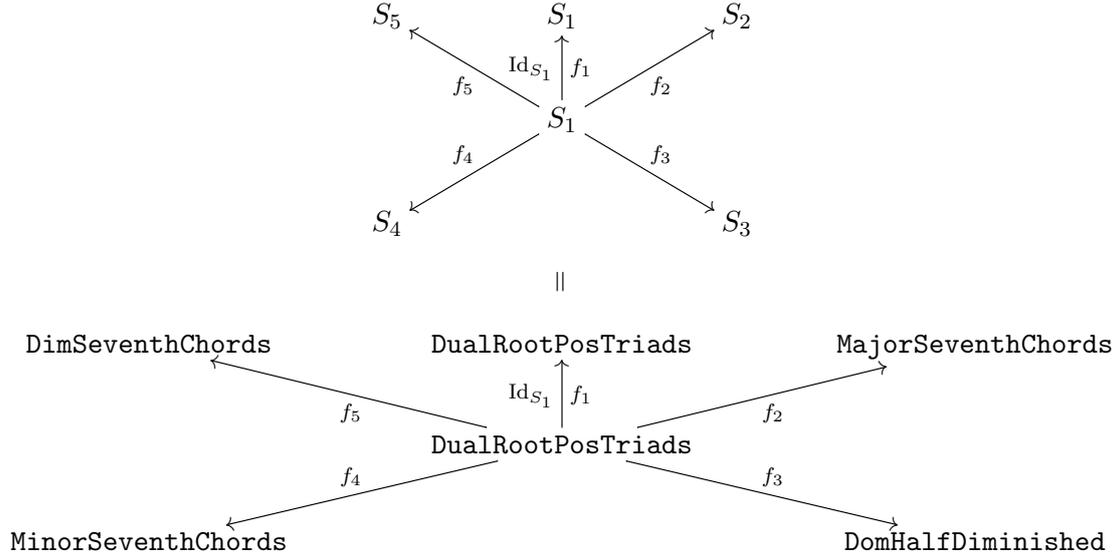
\begin{figure}
$$\xymatrix@C=4pc{ S_5 & S_1 & S_2 \\
& S_1 \ar[u]^{\text{Id}_{S_1}}_{f_1} \ar[ul]^{f_5} \ar[ur]_{f_2} \ar[dl]_{f_4} \ar[dr]^{f_3} & \\
S_4 & & S_3}$$
\begin{center}
\makebox[0pt][c]{\rotatebox{90}{$=$}}
\end{center}
$$\xymatrix@C=4pc{ \diminishedsevenths & \triads & \majorsevenths \\
& \triads \ar[u]^{\text{Id}_{S_1}}_{f_1} \ar[ul]^{f_5} \ar[ur]_{f_2} \ar[dl]_{f_4} \ar[dr]^{f_3} & \\
\minorsevenths & & \dom}$$
\caption{Star-shaped diagram to produce the meta-rotation in Figure~\ref{fig:5classes_smoothvoiceleading} to construct a pair of dual groups that act on five classes. The maps in this diagram are the identity and the four $TI$-equivariant bijections from Table~\ref{table:seventhchord_functions}. \label{fig:5classes_starshape}}
\end{figure}
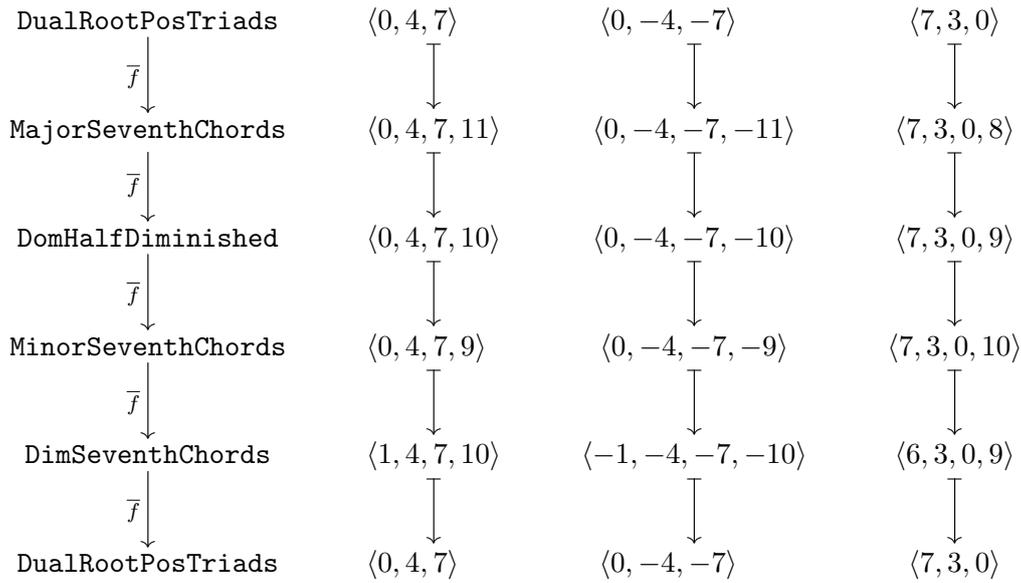
\begin{figure}
$$\xymatrix{\triads \ar[d]_{\overline{f}} & \langle 0,4,7 \rangle \phantom{,11} \ar@{|->}[d] & \langle 0,-4,-7 \rangle \phantom{-11} \ar@{|->}[d] & \langle 7,3,0 \rangle \ar@{|->}[d] \\
\majorsevenths \ar[d]_{\overline{f}} & \langle 0,4,7,11 \rangle \ar@{|->}[d] & \langle 0,-4,-7,-11 \rangle \ar@{|->}[d] & \langle 7,3,0,8 \rangle \ar@{|->}[d] \\
\dom \ar[d]_{\overline{f}} & \langle 0,4,7,10 \rangle \ar@{|->}[d] & \langle 0,-4,-7,-10 \rangle \ar@{|->}[d]
& \langle 7,3,0,9 \rangle \ar@{|->}[d] \\
\minorsevenths \ar[d]_{\overline{f}} & \langle 0,4,7,9 \rangle \phantom{0} \ar@{|->}[d] & \langle 0,-4,-7,-9 \rangle \ar@{|->}[d] & \langle 7,3,0,10 \rangle \ar@{|->}[d]\\
\diminishedsevenths \ar[d]_{\overline{f}} & \langle 1,4,7,10 \rangle \ar@{|->}[d] & \langle -1,-4,-7,-10 \rangle \ar@{|->}[d] & \langle 6,3,0,9 \rangle \ar@{|->}[d]\\
\triads & \langle 0,4,7 \rangle \phantom{,11} & \langle 0,-4,-7 \rangle \phantom{-11} & \langle 7,3,0 \rangle
}$$
\caption{The star-shaped diagram in Figure~\ref{fig:5classes_starshape} produces this sequence of maps. These fit together to form a meta-rotation $\overline{f}$ in Theorem~\ref{thm:generalized} for a dual pair of groups acting on the union of triads with four seventh-chord $TI$-classes of pcsegs. For readability, the second and third columns show the assignments only on the elements $\langle 0,4,7 \rangle$ and $\langle 0,-4, -7 \rangle$, to see the assignments on all elements of Table~\ref{table:seventhchord_sets} simply add $w$ everywhere in the second and third columns. This abbreviation works because each function in this $\overline{f}$ is of the form $f_{j+1} \circ f_j^{-1}$ where $f_{j+1}$ and $f_j$ are $TI$-equivariant bijections from Table~\ref{table:seventhchord_functions}. The fourth column shows the assignments on the $c$ chord for concreteness. This choice of $\overline{f}$ has all voice leading motions by half steps, except in the implicit root doubling in the last function, where $10$ moves to $0$ in the second column, and $-10$ moves to $0$ in the third column. {\it Two} voices move by a half step when going into diminished seventh chords, but everywhere else only {\it one} voice moves. \label{fig:5classes_smoothvoiceleading}}
\end{figure}
The star-shaped diagram of $TI$-equivariant bijections in Figure~\ref{fig:5classes_starshape} produces an especially nice meta-rotation $\overline{f}$ in Figure~\ref{fig:5classes_smoothvoiceleading}, it has smooth voice leading with half-step motion and preservation of 3 tones in all but one transformation. {\it Note} that the class ordering for $\overline{f}$ in Figure~\ref{fig:5classes_smoothvoiceleading} is different from the class ordering in Tables~\ref{table:seventhchord_sets} and \ref{table:seventhchord_functions} because major sevenths and dominant/half-diminished are interchanged. For this example, in Theorem~\ref{thm:generalized} we use the star-shaped diagram in Figure~\ref{fig:5classes_starshape}, and take $G$ to be the $TI$-group. Since $G$ is the $TI$-group and $S_1$ is $\triads$, $H$ must be the $PLR$-group. Theorem~\ref{thm:generalized} tells us that $\overline{f}$ has order 5, and
$$\overline{G} \cong TI\text{-group} \times \langle \overline{f} \rangle$$
$$\overline{H} \cong PLR\text{-group} \times \langle \overline{f} \rangle,$$
and both groups have order $24 \times 5=120$.

This group $\overline{H}$ is particularly interesting: it is a simply transitive group acting on triads and sevenths, the generators $P$, $L$, $R$, and $\overline{f}$ have parsimonious voice leading, and $\overline{f}$ maps between the different classes. Moreover, the algebra inside of $\overline{H}$ is computationally nice, since for $h \in H = PLR\text{-group}$ the extension of $h$ to all of $\overline{S}$,
$$\overline{h}=\bigsqcup_{j=1}^n f_jhf^{-1}_j,$$
is quite intuitive because of Proposition~\ref{prop:TI-equivariance_of_the_seventh_bijections}. Namely, for $h=K_{i,j}$ with $1 \leq i,j \leq 3$ in the $PLR$-group (different $j$ from disjoint union above), the extension $\overline{h}$ on major seventh chords, dominant/half-diminished seventh chords, and minor seventh chords is also $K_{i,j}$, but not on the diminished seventh chords. For $h=Q_i$, the extension $\overline{h}$ on all the seventh chords is also $Q_i$. And $\overline{K_{i,j}}$ and $\overline{Q_i}$ commute with the five powers of $\overline{f}$.

To be explicit, we work out in this group these conjugations $f_j h f_j^{-1}$ for $h=K_{1,3}=P$ on half of the elements (the elements containing a major triad). These computations confirm again quickly that the conjugations of $K_{1,3}$ to major sevenths, dominant seventh/half-diminished seventh chords, and minor seventh chords is again $K_{1,3}$: on the first three entries we see $P$, and on the fourth input entry we evaluate $I_{w+w+7}$ and see the fourth output entry. On the diminished seventh chords, the conjugation of $K_{1,3}$ is {\it not} $K_{1,3}$.
\begin{align*}
    P\langle w,\, w+4,\, w+7,\, w+11\rangle&=f_{Maj7Tr}Pf_{Maj7Tr}^{-1}\langle w,\,w+4,\,w+7,\,w+11\rangle \\
    &= f_{Maj7Tr}P\langle w,\,w+4,\,w+7\rangle  \\
    &=f_{Maj7Tr}\langle w+7,\, w+3,\, w\rangle \\
    &=\langle w+7,\,w+3,\,w,\,w+8\rangle
\end{align*}
\begin{align*}
    P\langle w,\,w+4,\,w+7,\,w+10\rangle&=f_{Dom7Tr}Pf_{Dom7Tr}^{-1}\langle w,\, w+4,\, w+7,\, w+10\rangle \\
    &=f_{Dom7Tr}P\langle w,\,w+4,\,w+7\rangle \\
    &=f_{Dom7Tr}\langle w+7,\, w+3,\, w \rangle \\
    &= \langle w+7,\, w+3,\, w,\, w+9 \rangle
\end{align*}
\begin{align*}
    P\langle w,\,w+4,\,w+7,\,w+9\rangle&=f_{Min7Tr}Pf_{Min7Tr}^{-1}\langle w,\, w+4,\,w+7,\,w+9\rangle \\
    &=f_{Min7Tr}P\langle w,\,w+4,\,w+7\rangle \\
    &=f_{Min7Tr}\langle w+7,\, w+3,\, w \rangle \\
    &= \langle w+7,\, w+3,\, w,\, w+10 \rangle
\end{align*}
\begin{align*}
    P\langle w+1,\,w+4,\,w+7,\,w+10\rangle&= f_{Dim7Tr}Pf_{Dim7Tr}^{-1}\langle w,\, w+4,\,w+7,\,w+10\rangle\\
    &=f_{Dim7Tr}P\langle w,\,w+4,\,w+7\rangle \\
    &=f_{Dim7Tr}\langle w+7,\, w+3,\, w \rangle \\
    &=\langle w+6,\, w+3,\, w,\, w+9 \rangle
\end{align*}
\end{example}

\begin{remark}
In Example~\ref{examp:5classes}, as in other examples with a meta-rotation, the analyst may make other choices for $\overline{f}$ to obtain an isomorphic group but with a different $\overline{f}$ action. In Figures~\ref{fig:alternativefbar1} and \ref{fig:alternativefbar2} we show two other choices for $\overline{f}$ in Example~\ref{examp:5classes} on 5 classes. For readability, we only indicate the maps on the $C$ triad. The salient features of these choices are in their captions. For the consonant triads and 4 seventh chord types in our $\overline{S}$ in Example~\ref{examp:5classes}, it is not possible to have only one voice move between each pitch-class segment. If we expand the $\overline{S}$ to include other chord types (augmented and diminished triads, mM7, and AugM7), then it is possible to have exactly one voice move by a semitone between sets. We do not work out that example in this paper.
\begin{figure}[htbp]
    \centering
    \begin{minipage}[t]{0.45\textwidth}
        \centering
        \begin{tikzpicture}[font = \scriptsize]
  \foreach \i/\label in {
  1/{\(\langle\mathbf{0},4,7\rangle\)},
  2/{\(\langle\mathbf{0},4,7,10\rangle\)},
  3/{\(\langle\mathbf{1},4,7,10\rangle\)},
  4/{\(\langle0,4,7,\mathbf{9}\rangle\)},
  5/{\(\langle\mathbf{0},4,7,11\rangle\)}
} {
    \node[minimum size=1cm] (N\i) at ({90 - 72*(\i-1)}:2cm) {\label};
  }
  \node at (0,0) {\text{root moves}};

   \draw[->, bend left=15] (N1) to node[auto] {\text{add voice}} (N2);
  \draw[->, bend left=15] (N2) to node[auto] {\shortstack{{\text{1 voice,}}\\{\text{1 semitone}}}} (N3);
  \draw[->, bend left=15] (N3) to node[auto] {\shortstack{{\text{2 voices,}}\\{\text{1 semitone each}}}} (N4);
  \draw[->, bend left=15] (N4) to node[auto] {\shortstack{{\text{1 voice,}}\\{\text{2 semitones}}}} (N5);
  \draw[->, bend left=15] (N5) to node[auto] {\text{drop voice}} (N1);
\end{tikzpicture}
        \caption{Another choice for $\overline{f}$ for consonant triads and the four types of seventh chords in Example~\ref{examp:5classes}, we indicate only the $\overline{f}$-orbit of $C$ for readability. Voice motion is always by a semi-tone, except in one transformation. The {\bf bolded} sequence of roots is $C$, $C$, $C\sh$, $A$, $C$, $C$, so the root changes only 3 times, one less than the choice in Figure~\ref{fig:5classes_smoothvoiceleading}. On inverted forms the opposite is true: the root changes 4 times for this $\overline{f}$, but only 3 times in Figure~\ref{fig:5classes_smoothvoiceleading}.}
        \label{fig:alternativefbar1}
    \end{minipage}%
    \hfill
    \begin{minipage}[t]{0.45\textwidth}
        \centering
        \begin{tikzpicture}[font = \scriptsize]
  \foreach \i/\label in {
  1/{\(\langle\mathbf{0},4,7\rangle\)},
  2/{\(\langle\mathbf{0},4,7,11\rangle\)},
  3/{\(\langle\mathbf{0},4,7,10\rangle\)},
  4/{\(\langle\mathbf{0},3,7,10\rangle\)},
  5/{\(\langle\mathbf{0},3,6,9\rangle\)}
} {
    \node[minimum size=1cm] (N\i) at ({90 - 72*(\i-1)}:2cm) {\label};
  }
  \node at (0,0) {\text{root does not move}};

   \draw[->, bend left=15] (N1) to node[auto] {\text{add voice}} (N2);
  \draw[->, bend left=15] (N2) to node[auto] {\shortstack{{\text{1 voice,}}\\{\text{1 semitone}}}} (N3);
  \draw[->, bend left=15] (N3) to node[auto] {\shortstack{{\text{1 voice,}}\\{\text{1 semitone}}}} (N4);
  \draw[->, bend left=15] (N4) to node[auto] {\shortstack{{\text{2 voices,}}\\{\text{1 semitone each}}}} (N5);
  \draw[->, bend left=15] (N5) to node[auto] {\shortstack{{\text{1 voice 1 semitone,}}\\{\text{drop voice}}}} (N1);
\end{tikzpicture}
        \caption{Yet another choice for $\overline{f}$ for consonant triads and the four types of seventh chords in Example~\ref{examp:5classes}, we indicate only the $\overline{f}$-orbit of $C$ for readability. This $\overline{f}$ would require changes to the 2nd and 3rd columns of the minor seventh chord row in Table~\ref{table:seventhchord_sets}: $TI\langle 9,0,4,7\rangle$, and $\langle w+9, w, w+4, w+7\rangle$, and $\langle w-9, w, w-4, w-7\rangle$. This would allow us to have $\langle 0,3,7,10\rangle$ with $w=3$ in in the sketch here of the orbit of $\langle 0,4,7 \rangle$. In this $\overline{f}$, the root stays the same and we have semi-tonal motion, but we have two moving voices twice, and the aforementioned alteration of Table~\ref{table:seventhchord_sets} and the concomitant alteration of the function $f_{Min7Tr}$ in Table~\ref{table:seventhchord_functions}.}
        \label{fig:alternativefbar2}
    \end{minipage}
\end{figure}
\end{remark}








\begin{example}[Analysis of Jazz Standard ``Autumn Leaves'' by Joseph Kosma] \label{AutumnLeaves}
The meta-rotation of Figure~\ref{fig:5classes_smoothvoiceleading} and Example~\ref{examp:5classes} allows us to make a network for the progression in ``Autumn Leaves'' in Figure~\ref{fig:Autumn}. This passage is a falling fifth sequence of seven chords, organized into three pairs of sevenths and a final tonic triad $g$. Additionally, in the soprano there is a sequence of two falling seconds and a falling fourth. In the treble we also see three broken triads (each with a passing tone) including into three seventh chords.

The variety of seventh chords in this passage presents a transformational challenge. The top two equal networks attempt to capture the three falling fifth pairs in the usual way of a product network with equal transformations on opposite sides of each square, but that cannot be attained because $\overline{f}$ is not root preserving regarding minor sevenths and because of the variety of the sevenths. Consequently, the opposite vertical arrows in each square are not equal. Nevertheless, the diagram does succeed in capturing common-tone preservation in each horizontal pair via the horizontal contextual inversions, and the diagram also succeeds with horizontal morphisms solely in $\overline{H}$ and vertical morphisms solely in $\overline{G}$. Our previous selections for the pitch-class segments do mean here that chords on the right of the two top networks are in root position, as in the musical staff. In the top two networks, we have substituted $A^{\o{}7}$ with its alteration $A^{-7}$.

The bottom two equal networks offer a different analysis, working diatonically with the $B\fl$-major scale identified with $\mathbb{Z}_7$, so $0=B\fl$, $1=C$, \dots, and $6=A$. With this $\mathbb{Z}_7$ encoding, the $E\fl$ root position triad is $\langle 3,5,0 \rangle$, the $d$ root position triad is $\langle 2,4,6 \rangle$, and the $c$ root position triad is $\langle 1,3,5 \rangle$ (there is no difference between major and minor). In the bottom two equal networks, the $\overline{f}$ is the inclusion of a root position diatonic triad into the top three notes of a root position diatonic seventh (there are only two $TI$-classes in these networks, the situation of Theorem~\ref{thm:construction_of_action_on_disjoint_union}). The bottom two equal networks capture the falling fifth pairs better than the top two networks because {\it mod 7} on $TI\langle 0,2,4,6 \rangle$ we have that $K_{1,2}$ is the same as $Q_{-4}$ postcomposed with retrograde, denoted $Q_{-4}^{\text{retro}}$. More precisely,
\begin{align*}
K_{1,2}\langle w,\, w+2,\, w+4,\, w+6 \rangle = \langle w+2,\, w,\, w+5,\, w+3\rangle,
\end{align*}
\begin{align*}
Q_{-4} \langle w,\, w+2,\, w+4,\, w+6 \rangle = \langle w+3,\, w+5,\, w,\, w+2 \rangle,
\end{align*}
\begin{align*}
Q_{-4}^{\text{retro}} \langle w,\, w+2,\, w+4,\, w+6 \rangle = \langle w+2,\, w,\, w+5,\, w+3 \rangle ,
\end{align*}
so that the falling fifth pairs can be described by either $K_{1,2}$ or $Q_{-4}^\text{retro}$. The bottom two networks also capture the stepwise descending motion in the soprano better than the top two networks. Namely, all the vertical maps are $T_{-1}$, and the soprano is the same as the roots of $E\fl$, $d$, and $c$ in the left vertical.

\begin{figure}
\begin{center}
\resizebox*{15cm}{!}{\includegraphics{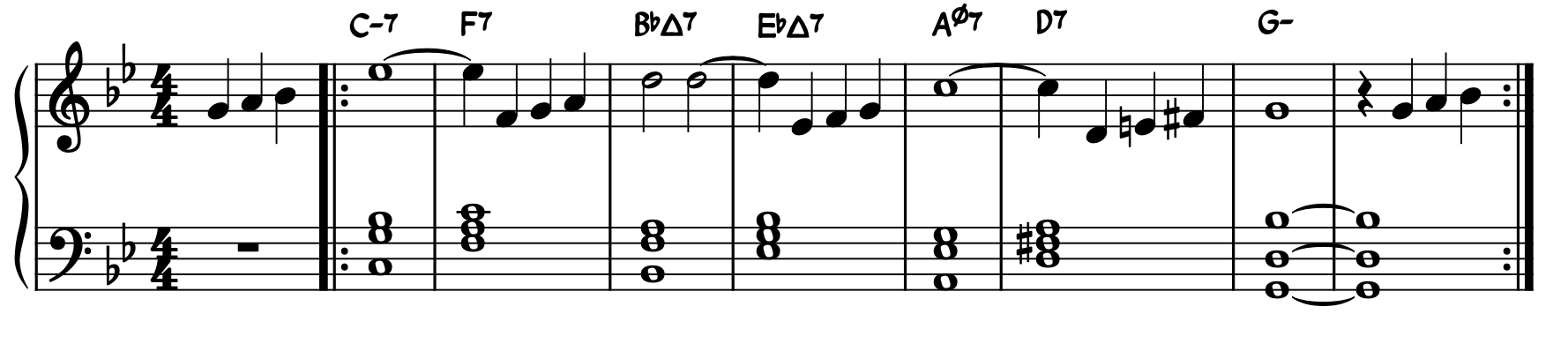}}
\end{center}
\vspace{-.25in}
$$
\begin{array}{c}
\entrymodifiers={=<2.5pc>[o][F-]}
\xymatrix@C=4.5pc@R=3pc{C^{-7} \ar[r]^{\left( \overline{f}\right)^{-1}K_{3,3}} \ar[d]_{\left( \overline{f}\right)^{-2} T_{2}} & F^7 \ar[d]^{\left(\overline{f}\right)^{-1}T_{-2}} \\
B\fl^{\triangle 7} \ar[r]^{K_{3,4}} \ar[d]_{\left(\overline{f}\right)^2T_{-5}} & E\fl^{\triangle 7} \ar[d]^{\left(\overline{f}\right)T_{-1}} \\
A^{-7} \ar[r]_{\left(\overline{f}\right)^{-1}K_{3,3}} \ar[d]_{\left(\overline{f}\right)^{-3} T_{-2} } & D^7  \\
g & *=<0pt,0pt>{} }
\end{array}
\qquad = \qquad
\begin{array}{c}
\text{\small inverted forms} \quad \quad \quad \text{\small transposed forms}\\
\entrymodifiers={+<2mm>[F-]}
\xymatrix@C=4.5pc@R=3pc{\langle 7, 3, 0, 10 \rangle \ar[r]^{\left(\overline{f}\right)^{-1}K_{3,3}} \ar[d]_{\left(\overline{f}\right)^{-2} T_{2}} & \langle 5, 9, 0, 3\rangle \ar[d]^{\left(\overline{f}\right)^{-1}T_{-2}} \\
\langle 9, 5, 2, 10\rangle \ar[r]^{K_{3,4}} \ar[d]_{\left(\overline{f}\right)^2T_{-5}} & \langle 3, 7, 10, 2\rangle \ar[d]^{\left(\overline{f}\right)T_{-1}} \\
\langle 4,0,9,7 \rangle \ar[r]_{\left(\overline{f}\right)^{-1}K_{3,3}} \ar[d]_{\left(\overline{f}\right)^{-3} T_{-2}} & \langle 2, 6, 9, 0 \rangle  \\
\langle 2, 10, 7 \rangle & *=<0pt,0pt>{} }
\end{array}
$$
\vspace{-.25in}
$$
\begin{array}{c}
\entrymodifiers={=<2.5pc>[o][F-]}
\xymatrix@C=3pc@R=3pc{
E\fl \ar[r]^{\overline{f}} \ar[d]^{T_{-1}}
& C^{-7} \ar[r]^{K_{1,2}}_{Q_{-4}^\text{retro}} \ar[d]^{T_{-1}}
& F^7 \ar[d]^{T_{-1}}
\\
d \ar[r]^{\overline{f}} \ar[d]^{T_{-1}}
& B\fl^{\triangle7} \ar[r]^{K_{1,2}}_{Q_{-4}^\text{retro}} \ar[d]^{T_{-1}}
& E\fl^{\triangle7} \ar[d]^{T_{-1}}
\\
c \ar[r]^{\overline{f}}
& A^{\o{}7} \ar[r]^{K_{1,2}}_{Q_{-4}^\text{retro}}
& D^7}
\end{array}
\qquad = \qquad
\begin{array}{c}
\entrymodifiers={+<2mm>[F-]}
\xymatrix@C=3pc@R=4pc{
\langle 3,5,0 \rangle \ar[r]^{\overline{f}} \ar[d]^{T_{-1}}
& \langle 1,3,5,0 \rangle  \ar[r]^{K_{1,2}}_{Q_{-4}^\text{retro} } \ar[d]^{T_{-1}}
& \langle 3,1,6,4 \rangle  \ar[d]^{T_{-1}}
\\
\langle 2,4,6 \rangle  \ar[r]^{\overline{f}} \ar[d]^{T_{-1}}
& \langle 0,2,4,6 \rangle  \ar[r]^{K_{1,2}}_{Q_{-4}^\text{retro} } \ar[d]^{T_{-1}}
& \langle 2,0,5,3 \rangle  \ar[d]^{T_{-1}}
\\
\langle 1,3,5 \rangle  \ar[r]^{\overline{f}}
& \langle 6,1,3,5 \rangle  \ar[r]^{K_{1,2}}_{Q_{-4}^\text{retro} }
& \langle 1,6,4,2 \rangle }
\end{array}
$$
\caption{The musical staff is a textural reduction and harmonic analysis of ``Autumn Leaves", composed by Joseph Kosma, measures 1--8. The first two equal networks are an interpretation of three falling fifth pairs, using the meta-rotation $\overline{f}$ from Figure~\ref{fig:5classes_smoothvoiceleading} and Example~\ref{examp:5classes}. In the two networks we used the alteration $A^{-7}$ in place of $A^{\o{}7}$ (pitch $E$ instead of pitch $E\flat$) in order to have a more symmetric network. The left and right vertical transformations in each square are not the same because $\overline{f}$ is not root preserving when mapping to minor seventh chords, for instance on $A^{-7}$ in the inverted order we have $\left( \overline{f}\right)^{-2} (A^{-7}) = \left(\overline{f}\right)^{-2} (\langle 4, 0, 9, 7 \rangle ) = \langle 4, 0, 9, 5 \rangle = F^{\triangle 7}$. Instead of the $\overline{f}$ from Figure~\ref{fig:5classes_smoothvoiceleading} we could also use the $\overline{f}$ from Figure~\ref{fig:basicexample_multiset}, because diminished sevenths are not in the passage (the powers in the two networks would stay the same).  The bottom two equal diagrams are an alternative {\it mod 7} interpretation involving just two classes: mod 7 triads as $TI\langle 0,2,4\rangle$ and mod 7 sevenths as $TI\langle 0,2,4,6 \rangle$. The mod 7 encoding is the $B\fl$ major scale, so $0=B\fl$. The different $\overline{f}$ in the bottom two figures is the inclusion of a triad into the top of a seventh. In the bottom two networks, we do not replace $A^{\o{}7}$ with $A^{-7}$, but we do signify both the diatonic pitch classes $F$ and $F\sh$ by 4 in the final $D^7$. The bottom two equal networks capture the falling fifth pairs better than the top two networks because {\it mod 7} on $TI\langle 0,2,4,6 \rangle$ we have that $K_{1,2}$ is the same as $Q_{-4}$ composed with retrograde. The bottom two networks also capture the stepwise descending motion in the soprano via $T_{-1}$ and the roots of $E\fl$, $d$, and $c$. }  \label{fig:Autumn}
\end{figure}
Typical voice leadings of diatonic falling fifth sequences of seventh chords involve two alternating types of chord inversions, such as root position chords $S_1$ and four-three (second inversion seventh) chords $S_2$ or six-five (first inversion seventh) chords $S_3$  and four-two (third inversion seventh) chords $S_4$. In our notated version of the ``Autumn Leaves'' progression, we use root position chords throughout, but we still use an alternation of a closer $S_5$ and a wider spacing $S_6$ of the scale degrees.
Thus, in all three cases we may analyze such progressions as flip-flop-cycles in a direct product $\mathbb{Z}_2 \times \mathbb{Z}_7$ generated by a meta-rotation $\overline{f}_i$ between two sets and a scale degree transposition $t_1$. Concretely:
$$\begin{array}{llll}
f_1: S_1 \to S_2: & \langle n, n+2, n+4, n+6 \rangle & \mapsto & \langle n, n+2, n+3, n+5 \rangle \\
f_2: S_3 \to S_4: &\langle n, n+2, n+4, n+5 \rangle & \mapsto & \langle n, n+1, n+3, n+5 \rangle \\
f_3: S_5 \to S_6: & \langle n, n+2, n+4, n+6 \rangle & \mapsto & \langle n, n+4, n+6, n+2 \rangle.
\end{array}$$
We define $\overline{f}_1: S_1 \sqcup S_2 \to S_1 \sqcup S_2$ with $\overline{f_1}|_{S_1} = f_1$ and $\overline{f_1}|_{S_2} = f_1^{-1}$, analogously for $\overline{f_2}: S_3 \sqcup S_4 \to S_3 \sqcup S_4$ and $\overline{f_3}: S_5 \sqcup S_6 \to S_6 \sqcup S_5$.
The diatonic falling fifth progression in our setting of ``Autumn Leaves'' (see Figure~\ref{fig:Autumn}) corresponds to the iteration of the transformation
$t_{-4} \circ  \overline{f_3}$. For the other two examples we obtain the transformations which are shown in the annotation of Figure~\ref{fig:FifthFall}.

\begin{figure}[h]
\begin{center}
\resizebox*{10 cm}{!}{\includegraphics{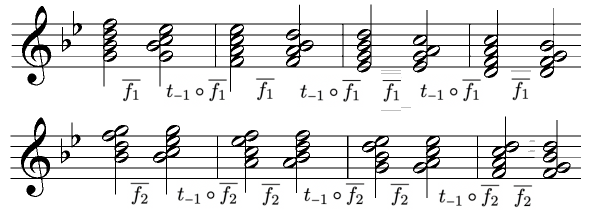}}
\end{center}
\caption{Two kinds of alternating voice-leadings in a diatonic falling fifth sequence of seventh chords: root positions and four-three-chords (upper system), six-five chords and four-two chords (lower system). Both systems show only one half of the complete progressions of length $14$.}
\label{fig:FifthFall}
\end{figure}
The transformations $t_{-1} \circ \overline{f_1}$ and $t_{-1} \circ \overline{f_2}$ are both of order $14$, but they generate these chords in a different order. Below we will see how to choose the appropriate meta-rotation to get the falling fifth sequence as a cyclic orbit with our construction.


Recall that John Clough \cite{cloughFlipFlop} discusses flip-flop-cycles and cyclic groups as theoretical alternatives. The associated cyclic group can in our case we by expressed
in terms of the $J$ function from \cite{cloughdouthett_1991}, recalled on page 81 of \cite{douthettFIPSandDynVoiceLeading} as
$$J_{c,d}^{m}(k)\overset{\text{def}}{=}\left\lfloor{\dfrac{ck+m}{d}} \right\rfloor,$$
where $c$ is the {\it chromatic cardinality}, $d$ is the {\it diatonic cardinality}, and $m$ is the {\it mode index}.
The transformation: $\phi: J_{7, 4}^m \langle 0,1,2,3 \rangle \mapsto J_{7, 4}^{m-1} \langle 0,1,2,3 \rangle$ acts on the 28-element set of the inversions of the generic seventh chords in $\mathbb{Z}_7$ and is cyclic of order $28$. Its square $\phi^2: J_{7, 4}^m \langle 0,1,2,3 \rangle \mapsto J_{7, 4}^{m-2}\langle 0,1,2,3 \rangle $ is of order $14$ and generates both orbits in Figure \ref{fig:FifthFall}. The right choice for the meta-rotation is therefore $\overline{f} = \phi^{14}$ sending $\langle n, n+2, n+4, n+6 \rangle$ to $\langle n+4, n+6, n, n+2 \rangle$ and sending $\langle n, n+2, n+4, n+5 \rangle$ to $\langle n+4, n+5, n, n+2\rangle$. The cyclic generator $\phi^2$ has the form $\phi^2 = t_4 \circ \overline{f}$ and works for both progressions.

\end{example}

\begin{example} \label{examp:generated_scales}
Another example of Theorem~\ref{thm:generalized} can be found in the theory of generated scales. Consider the meta-rotation on the $T$-classes of the tetractys, the pentatonic, the diatonic, and the chromatic pitch class segments (all generated by a perfect fifth). Here we indicate $\overline{f}$ just on the tetractys and its images, but we know that this $\overline{f}$ will commute with all transpositions by Theorem~\ref{thm:conjugation_with_extension}.
\begin{equation} \label{equ:tetractys_images}
\langle F,C,G \rangle \overset{\overline{f}}{\mapsto} \langle F,C,G,D,A \rangle \overset{\overline{f}}{\mapsto}  \langle F,C,G,D,A,E,B \rangle \overset{\overline{f}}{\mapsto} \langle F,C,G, \dots, A\sh \rangle  \overset{\overline{f}}{\mapsto} \langle F,C,G \rangle
\end{equation}
This $\overline{f}$ could also be constructed from a star-shaped diagram with the tetractys translates in the middle. The outcome of Theorem~\ref{thm:generalized} now produces a self-dual group generated by $T_1$ and $\overline{f}$, with $12 \times 4 = 48$ elements, acting simply transitively on the union of translates of the pitch-class segments in \eqref{equ:tetractys_images}. The group is clearly isomorphic to $\mathbb{Z}_{12} \times \mathbb{Z}_{4}$.
\end{example}

\begin{example}[The Flattening Transformation and the Dynamical Voice Leading of \cite{douthettFIPSandDynVoiceLeading}]
This example of Theorem~\ref{thm:generalized} starts with the four $T$-classes of the root-position diatonic seventh chords  $\langle 0, 4, 7, 11\rangle$, $\langle 0, 4, 7, 10 \rangle$,
$\langle 0, 3, 7, 10 \rangle$ and $\langle 0, 3, 6, 10\rangle$, encoded as pc-segs. On each of these four sets $S_i$, we have an action of the transposition group $\langle T_1
\rangle\cong \mathbb{Z}_{12}$. On their disjoint union we have the following meta-rotation $\overline{f}$
\begin{equation} \label{equ:fbar_for_J}
\xymatrix{(\cdot)^{\triangle 7} \ar@{|->}[r]^{\overline{f}} & (\cdot)^7 \ar@{|->}[r]^{\overline{f}} & (\cdot)^{-7} \ar@{|->}[r]^{\overline{f}} & (\cdot)^{\o{}7} \ar@{|->}[r]^{\overline{f}} & (\cdot)^{\triangle 7}}
\end{equation}
preserving the root of each seventh chord:
$$\begin{array}{lll}
\overline{f}\langle n, n+4, n+7, n+11 \rangle & = & \langle n, n+4, n+7, n+10 \rangle \\
\overline{f}\langle n, n+4, n+7, n+10 \rangle & = & \langle n, n+3, n+7, n+10 \rangle \\
\overline{f} \langle n, n+3, n+7, n+10 \rangle & = & \langle n, n+3, n+6, n+10 \rangle \\
\overline{f} \langle n, n+3, n+6, n+10 \rangle & = & \langle n, n+4, n+7, n+11 \rangle
\end{array}.$$
This meta-rotation $\overline{f}$ is almost the same as the meta-rotation in Figure~\ref{fig:alternativefbar2}, but has some important differences.\footnote{The differences between the $\overline{f}$ here in equation \eqref{equ:fbar_for_J} and the $\overline{f}$ in Figure~\ref{fig:alternativefbar2} are that the $\overline{f}$ here in equation \eqref{equ:fbar_for_J} only uses $T$-classes rather than $TI$-classes, the group $G$ is the $T$-group rather than the $TI$-group, and the consonant triads are not included. Also notable is that in \eqref{equ:fbar_for_J}, the dominant seventh chords and half-diminished seventh chords (in root position!) are two different $S_i$ and $S_j$ because $G$ is the $T$-group, whereas in Figure~\ref{fig:alternativefbar2} they are together in one single $S_i$ (with the half-diminished seventh chords in reverse root position).} This $\overline{f}$ has order 4 and we distinguish it from the map $\overline{f_\flat}$, which coincides with $\overline{f}$ on major seventh, dominant seventh and minor seventh chords, but which differs on half-diminished chords: $$\overline{f_\flat}(n, n+3, n+6, n+10) = (n-1, n+3, n+6, n+10).$$ The map $\overline{f_\flat}$ generates a cyclic group of order $48$ acting simply transitively on the set of all diatonic seventh chords. Hence, the relation between the two groups $\langle\overline{f}, T_1 \rangle \cong \mathbb{Z}_4 \times \mathbb{Z}_{12}$ and $\langle \overline{f_\flat} \rangle \cong \mathbb{Z}_{48}$ generalizes the relation between flip-flop-cycles and their associated ``purely" cyclic orbits in \cite{cloughFlipFlop}. The same trajectory through all diatonic seventh chords of length 48 can alternatively be obtained through the successive application of $\overline{f}, \overline{f}, \overline{f}, T_{-1}, \overline{f}, \overline{f}, \overline{f}, T_{-1}, ...$ or $\overline{f_\flat}, \overline{f_\flat}, \overline{f_\flat}, \overline{f_\flat}, \overline{f_\flat}, \overline{f_\flat}, \overline{f_\flat}, \overline{f_\flat}, ....$.
The map $\overline{f_\flat}$ can be interpreted as the ``productive shadow" of the \emph{flattening transformation}. Consider the undecorated abstract note names $\langle C, E, G, B \rangle$ as fixed material for the tonic seventh inside of a diatonic mode; the containing mode decides the decorations on the note names. The flattening transformation simply adds a flat to the containing key signature (corresponding to the ``progression'' of modes); the sequence of resulting pitch-class segments is an ``orbit of the flattening transformation.'' See Figure~\ref{fig:flattening_transformation}.

The flattening transformation can again be formulated in terms of the $J$ function (see Example \ref{AutumnLeaves}):
$$J^m_{12, 7}\left(J^3_{7, 4}(0, 1, 2, 3)\right) = J^m_{12, 7}(0, 2, 4, 6) \quad \mapsto \quad J^{m-1}_{12, 7}\left (J^3_{7, 4}(0, 1, 2, 3) \right) = J^{m-1}_{12, 7}(0, 2, 4, 6).$$
\begin{figure}[H]
\begin{center}
\includegraphics[scale = 0.53]{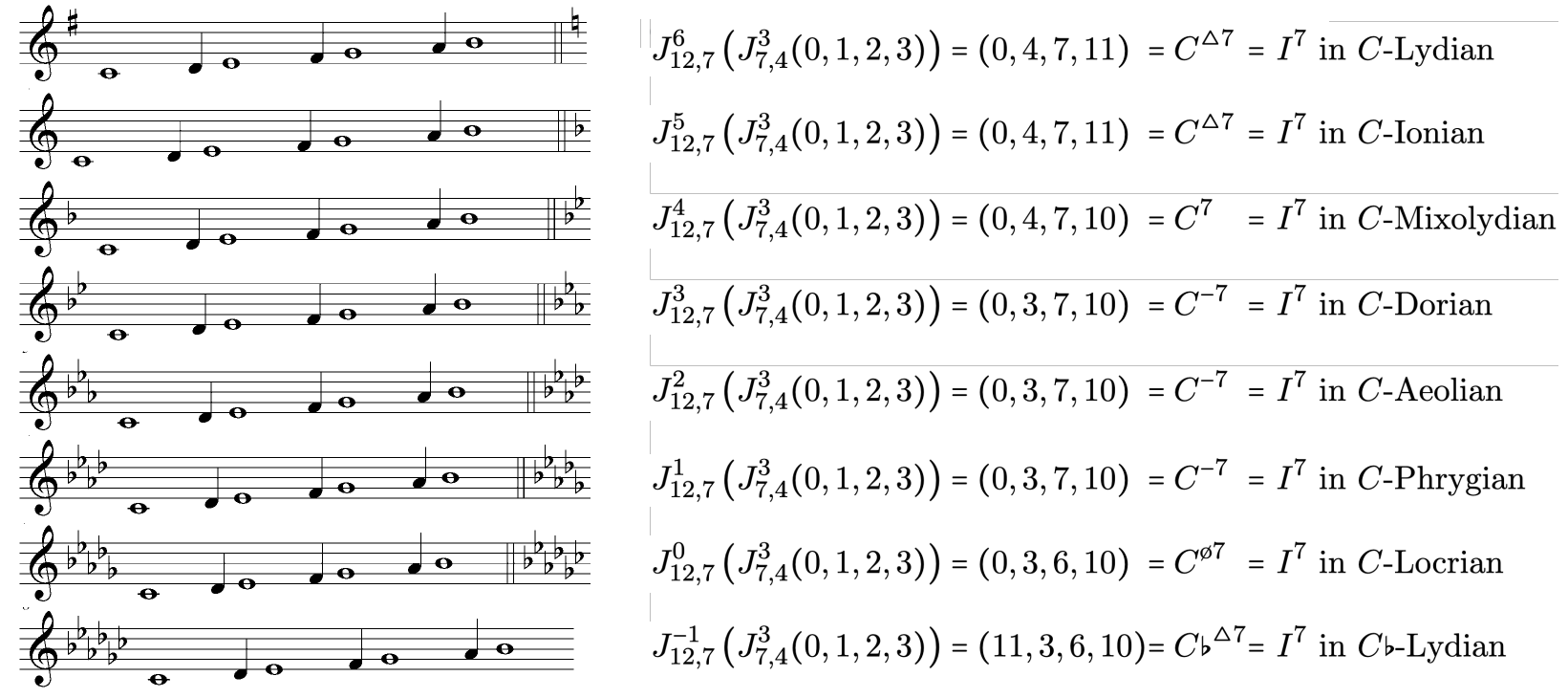}
\end{center}
\caption{The first few sevenths in an orbit of the flattening transformation. The abstract note names $\langle C, E, G, B \rangle$ are held fixed, and the containing key signature repeatedly receives a flat. In some instances, the flattened note is not in the chord, so the chord stays the same. All 48 major, dominant, minor, and half-diminished seventh chords in root position are traversed, with repeats. The cycle with repeats has 84 elements, all in root position.} \label{fig:flattening_transformation}
\end{figure}
The entire cycle is of length $84=12 \times 7$ and traverses all 48 diatonic seventh chords in root position. ``Non-productive transforms'' happen when the flattened note is a non-chord-tone. 
Apart from the long cycle $\langle \overline{f_\flat} \rangle$, a grid network with rows 4 instances of $\overline{f}$ and vertical morphisms $T_{-1}$ is an alternative way to illustrate the flattening transformation.  Optionally, an analyst may insert identity arrows in each row to strictly record the ``non-productive transforms.''
\end{example}

\section*{Acknowledgements}
\addcontentsline{toc}{section}{Acknowledgements}

Thomas Fiore thanks Julian Hook for some insights about the omnibus progression in an email correspondence in June, 2022.

\section*{AI Statement}
\addcontentsline{toc}{section}{AI Statement}

{\it None} of the ideas in this article were generated by AI. {\it None} of the text or figures of this article were generated by AI. We only used AI to to assist with (or debug) LaTeX code for human-created tables and diagrams. In particular, for Figures~\ref{fig:alternativefbar1} and \ref{fig:alternativefbar2} we used Microsoft Copilot version 1.25062.106.0 to assist with typesetting human-made drawings with \texttt{tikz}. We used GPT 4.1 through the University of Michigan to learn how to typeset enclosing circles and squares in \texttt{xymatrix}, to debug Tables~\ref{table:seventhchord_sets} and \ref{table:seventhchord_functions} and Figure~\ref{fig:flattening_transformation}, and to resolve LaTeX package conflicts.

\section*{Funding}
\addcontentsline{toc}{section}{Funding}

The visit of Thomas Fiore and Thomas Noll at the Institut de Recherche Math\'{e}matique Avanc\'{e}e, Universit\'{e} de Strasbourg, July 17--28, 2017 was supported by the University
of Strasbourg Institute for Advanced Study via the grant  Structural Music Information Research (SMIR), PI Moreno Andreatta (CNRS/University of Strasbourg \& Ircam).

Doctoral student Sonia Cannas, also a member of the SMIR Project, was supported by a scholarship of the University of Pavia in a cotutelle agreement with the University of Strasbourg.

Undergraduate students Ethan Bonnell, Hayden Pyle, No\'{e} Rodriguez, and Meredith Williams were participants in the NSF Research Experience for Undergraduates at the University of Michigan-Dearborn in Summer 2022, and were supported by National Science Foundation Grant DMS-1950102. This grant also partially supported mentor Thomas Fiore in Summer 2022.

Thomas Fiore gratefully acknowledges alumni support of the Alexander von Humboldt Foundation for the local costs of his visit in Leipzig February 12--26, 2025 to work with Thomas Noll. Thomas Fiore and Thomas Noll both thank the Max Planck Institute for Mathematics in the Sciences (MPI MiS) in Leipzig, especially Director Anna Wienhard, for hosting the visit. Thomas Noll thanks MPI MiS for local support during the visit.

Fiore's transportation to Leipzig was covered by an internal travel grant of the College of Arts, Sciences, and Letters at the University of Michigan-Dearborn.

\bibliographystyle{plain}
\bibliography{REU_arXiv}

\begin{thebibliography}{10}

\bibitem{arnettbarth_sevenths}
Jacob Arnett and Eric Barth.
\newblock Generalizations of the {\it {t}onnetz}: Tonality revisited.
\newblock In {\em Proceedings of the 2011 Midstates Conference on Undergraduate
  Research in Computer Science and Mathematics}.
  https://personal.denison.edu/~lalla/MCURCSM2011/10.pdf, 2011.

\bibitem{BerryFiore}
Cameron Berry and Thomas~M. Fiore.
\newblock Hexatonic systems and dual groups in mathematical music theory.
\newblock {\em Involve}, 11(2):253--270, 2018.

\bibitem{CannasMCM2017}
Sonia Cannas, Samuele Antonini, and Ludovico Pernazza.
\newblock On the group of transformations of classical types of seventh chords.
\newblock In Octavio~A. Agust\'{i}n-Aquino, Emilio Lluis-Puebla, and Mariana
  Montiel, editors, {\em Mathematics and Computation in Music}, Lecture Notes
  in Computer Science. Heidelberg: Springer, 2017.

\bibitem{childs}
Adrian Childs.
\newblock {M}oving {B}eyond {N}eo-{R}iemannian {T}riads: {E}xploring a
  {T}ransformational {M}odel for {S}eventh {C}hords.
\newblock {\em Journal of Music Theory}, 42(2):191--193, 1998.

\bibitem{cloughFlipFlop}
John Clough.
\newblock {F}lip-{F}lop {C}ircles and {T}heir {G}roups.
\newblock In Jack Douthett, Martha~M. Hyde, and Charles~J. Smith, editors, {\em
  Music {T}heory and {M}athematics: {C}hords, {C}ollections, and
  {T}ransformations}, volume~50 of {\em Eastman Studies in Music}. University
  of Rochester Press, 2008.

\bibitem{cloughdouthett_1991}
John Clough and Jack Douthett.
\newblock Maximally even sets.
\newblock {\em Journal of Music Theory}, 35(1/2):93--173, 1991.

\bibitem{cohn1996}
Richard Cohn.
\newblock Maximally smooth cycles, hexatonic systems, and the analysis of
  late-romantic triadic progressions.
\newblock {\em Music Analysis}, 15(1):9--40, 1996.

\bibitem{cohn1997}
Richard Cohn.
\newblock Neo-{R}iemannian operations, parsimonious trichords, and their
  ``{T}onnetz'' representations.
\newblock {\em Journal of Music Theory}, 41(1):1--66, 1997.

\bibitem{cohnsurvey}
Richard Cohn.
\newblock Introduction to neo-riemannian theory: A survey and a historical
  perspective.
\newblock {\em Journal of Music Theory}, 42(2):167--180, 1998.

\bibitem{conradkeith_splitting}
Keith Conrad.
\newblock Splitting of short exact sequences for groups.
\newblock {\em Unpublished}, Undated.

\bibitem{cransfioresatyendra}
Alissa~S. Crans, Thomas~M. Fiore, and Ramon Satyendra.
\newblock Musical actions of dihedral groups.
\newblock {\em American Mathematical Monthly}, 116(6):479--495, 2009.

\bibitem{douthettFIPSandDynVoiceLeading}
Jack Douthett.
\newblock Filtered point-symmetry and dynamical voice leading.
\newblock In Jack Douthett, Martha~M. Hyde, and Charles~J. Smith, editors, {\em
  Music {T}heory and {M}athematics: {C}hords, {C}ollections, and
  {T}ransformations}, volume~50 of {\em Eastman Studies in Music}. University
  of Rochester Press, 2008.

\bibitem{DouthettSteinbach}
Jack Douthett and Peter Steinbach.
\newblock Parsimonious graphs: A study in parsimony, contextual
  transformations, and modes of limited transposition.
\newblock {\em Journal of Music Theory}, 42(2):241--263, 1998.

\bibitem{DummitFoote}
David~S. Dummit and Richard~M. Foote.
\newblock {\em Abstract algebra}.
\newblock John Wiley \& Sons, Inc., Hoboken, NJ, third edition, 2004.

\bibitem{fiorenoll_SIAM}
Thomas~M. Fiore and Thomas Noll.
\newblock Voicing transformations of triads.
\newblock {\em SIAM Journal on Applied Algebra and Geometry}, 2(2):281--313,
  2018.

\bibitem{fiorenollsatyendraMCM2013}
Thomas~M. Fiore, Thomas Noll, and Ramon Satyendra.
\newblock Incorporating voice permutations into the theory of neo-{R}iemannian
  groups and {L}ewinian duality.
\newblock In {\em Mathematics and computation in music}, volume 7937 of {\em
  Lecture Notes in Comput. Sci.}, pages 100--114. Springer, Heidelberg, 2013.

\bibitem{fiorenollsatyendraSchoenberg}
Thomas~M. Fiore, Thomas Noll, and Ramon Satyendra.
\newblock Morphisms of generalized interval systems and {$PR$}-groups.
\newblock {\em Journal of Mathematics and Music}, 7(1):3--27, 2013.

\bibitem{fioresatyendra2005}
Thomas~M. Fiore and Ramon Satyendra.
\newblock Generalized contextual groups.
\newblock {\em Music Theory Online}, 11(3), 2005.

\bibitem{gollin}
Edward Gollin.
\newblock {S}ome {A}spects of {T}hree-{D}imensional \emph{{T}onnetze}.
\newblock {\em Journal of Music Theory}, 42(2):195--206, 1998.

\bibitem{hookUTTthesis2002}
Julian Hook.
\newblock {\em Uniform Triadic Transformations}.
\newblock PhD thesis, Indiana University, 2002.

\bibitem{hookUTT2002}
Julian Hook.
\newblock Uniform triadic transformations.
\newblock {\em Journal of Music Theory}, 46(1-2):57--126, 2002.

\bibitem{hook2007cross}
Julian Hook.
\newblock Cross-type transformations and the path consistency condition.
\newblock {\em Music Theory Spectrum}, 29(1):1--39, 2007.

\bibitem{KerkezBridges}
Boris Kerkez.
\newblock An extension of the neo-{R}iemannian {$PLR$}-group to seventh chords.
\newblock In Robert Bosch, Douglas McKenna, and Reza Sarhangi, editors, {\em
  Bridges Towson: Mathematics, Music, Art, Architecture, Culture}, pages
  485--488. Phoenix, AZ: Tessellations Publishing, 2012.

\bibitem{kolman}
Oren Kolman.
\newblock Transfer principles for generalized interval systems.
\newblock {\em Perspectives of New Music}, 42(1):150--190, 2004.

\bibitem{LewinGMIT}
David Lewin.
\newblock {\em Generalized Musical Intervals and Transformations}.
\newblock Yale University Press, New Haven, 1987.

\bibitem{popoff2013}
Alexandre Popoff.
\newblock Building generalized neo-riemannian groups of musical transformations
  as extensions.
\newblock {\em Journal of Mathematics and Music}, 7(1):55--72, 2013.

\bibitem{Waller}
Derek~A. Waller.
\newblock Some combinatorial aspects of the musical chords.
\newblock {\em The Mathematical Gazette}, 62(419):12--15, 1978.

\bibitem{ziehn}
Bernhard Ziehn.
\newblock {\em Canonical Studies: a New Technic in Composition = Canonische
  Studien : eine Neue Compositions-Technik}.
\newblock Wm. A. Kaun Music Co. and Richard Kaun Musik Verlag, Milwaukee and
  Berlin, 1912.

\end{thebibliography}

\addcontentsline{toc}{section}{References}

\end{document}